\documentclass[11pt,a4paper]{article}
\usepackage[utf8]{inputenc}
\usepackage[english]{babel}
\usepackage[ruled,section]{algorithm}
\usepackage{makeidx}
\usepackage[font=small]{caption}
\usepackage{amsmath}
\usepackage{amsthm}
\usepackage{amsfonts}
\usepackage{amssymb}
\usepackage{mathrsfs}
\usepackage{graphicx}
\usepackage[margin=25mm]{geometry}
\usepackage{enumerate}
\usepackage[T1]{fontenc}
\usepackage{mathtools}
\usepackage{nicefrac}
\usepackage{braket}
\usepackage{subfig}
\usepackage{csquotes}

\usepackage{simplewick}%
\usepackage{latexsym}%
\usepackage[all]{xy}%
\usepackage[plainpages=false,pagebackref=false]{hyperref}   %2
\usepackage{color}
\usepackage{authblk}
\usepackage{hyperref}
\makeindex
\hypersetup{pdfpagelabels,hyperindex,colorlinks=true,breaklinks=true,bookmarks=true,linkcolor=black,citecolor=black,
	urlcolor=black}
\usepackage{esint}%
\usepackage{perpage}%
\newcommand{\real}{\mathbb{R}}
\newcommand{\n}{\mathbb{N}}
\newcommand{\rn}{ {\mathbb{R}^n} }
\renewcommand{\d}{\; \mathrm{d}}

\newcommand{\intav}[1]{\mathchoice {\mathop{\vrule width 6pt height 3 pt depth  -2.5pt
\kern -8pt \intop}\nolimits_{\kern -6pt#1}} {\mathop{\vrule width
5pt height 3  pt depth -2.6pt \kern -6pt \intop}\nolimits_{#1}}
{\mathop{\vrule width 5pt height 3 pt depth -2.6pt \kern -6pt
\intop}\nolimits_{#1}} {\mathop{\vrule width 5pt height 3 pt depth
-2.6pt \kern -6pt \intop}\nolimits_{#1}}}

\numberwithin{equation}{section}
\newtheorem{teo}{Theorem}[section]
\newtheorem{lem}[teo]{Lemma}
\newtheorem{rmk}[teo]{Remark}
\newtheorem{example}[teo]{Example}

\newtheorem{defin}[teo]{Definition}

\newtheorem{prop}[teo]{Proposition}

\newtheorem{claim}[teo]{Claim}

\newcommand{\defeq}{\mathrel{\mathop:}=}%
\newcommand{\diam}{\mathrm{diam}}
\newcommand{\tr}{\mathrm{Tr}}%

\begin{document}
%%%%%%%%%%%%%%%%%%%%%%

\title{\bf\Large Regularity estimates for fully nonlinear elliptic PDEs with general Hamiltonian terms and unbounded ingredients}
\author[1]{João Vitor da Silva\footnote{jdasilva@unicamp.br}}
\author[2]{Gabrielle Nornberg\footnote{gabrielle@icmc.usp.br}}
\affil[1]{\small
Departamento de Matem\'atica - IMECC,
Universidade Estadual de Campinas}
\affil[2]{\small Instituto de Ciências Matemáticas e de Computação, Universidade de São Paulo}
\date{}

\maketitle

{\small\noindent{\bf{Abstract.}} We develop an optimal regularity theory for $L^p$-viscosity solutions of fully nonlinear uniformly elliptic equations in nondivergence form whose gradient growth is described through a Hamiltonian function with measurable and possibly unbounded coefficients. Our approach treats both superlinear and sublinear gradient regimes in a unified way.

We show $C^{0,\alpha}$, $C^{0,\textrm{Log-Lip}}$, $C^{1,\alpha}$, $ C^{1,\textrm{Log-Lip}}$ and $C^{2,\alpha}$ regularity estimates, by displaying the growth allowed to the Hamiltonian in order to deal with an unbounded nonlinear gradient coefficient, whose integrability in turn gets worse as we approach the quadratic regime.

Moreover, we find proper compatibility conditions for which our regularity results depend intrinsically on the integrability of the underlying source term.
As a byproduct of our methods, we prove a priori BMO estimates; sharp regularity to associated recession and flat profiles under relaxed convexity assumptions; improved regularity for a class of singular PDEs; and a Perron type result under unbounded ingredients.

\medskip

{\small\noindent{\bf{Keywords.}} {Regularity; fully nonlinear PDEs; unbounded Hamiltonian; $L^p$-viscosity solutions.}
	
\medskip
	
{\small\noindent{\bf{MSC2020.}} 35J15, 35D40, 35B65.}

\section{Introduction}\label{Introduction}

In this paper, we address sharp regularity estimates for $L^p$-viscosity solutions of fully nonlinear second order differential equations in the form
\begin{equation}\label{F=f}
  F(x, Du, D^2 u) =  f(x) \quad \text{in} \quad \Omega,
\end{equation}
where $\Omega\subset\rn$ is a bounded domain, and $F$ is a uniformly elliptic operator satisfying suitable growth conditions with respect to the gradient entry. The source term $f$ has appropriate integrability, by going through Lebesgue spaces $L^p(\Omega)$ for $p\in (p_0,+\infty]$, where $p_0$ is the Escauriaza's exponent. \vspace{0.02cm}

Over the past years, fully nonlinear operators with nonlinear gradient growth have been widely recognized and investigated, in particular due to their importance in applications. For instance, several research interests include free boundary problems of obstacle type \cite{KoTa19}, Phragm\'{e}n-Lindel\"{o}f type results \cite{Tate20}; multiplicity of solutions for nonproper equations with natural growth in the gradient \cite{multiplicidade}; stochastic homogenization of Hamilton-Jacobi models \cite{Ahomo}, H\"{o}lder estimates for degenerate PDEs with coercive Hamiltonians \cite{CDLP10}, and boundary regularity for equations and differential inequalities \cite{BGMW}, among others.

Frequently in these studies, regularity results for generic profiles play a key role. In the matter of superlinear gradient growth, H\"{o}lder regularity was developed in \cite{arma2010} (see also \cite{CDLP10}) while $C^{1,\alpha}$ regularity in \cite{Norn19}, both regarding unbounded coefficients in the $L^p$-viscosity sense. Instead, as far as sublinear gradient growth is concerned, its study usually has to do with degenerate equations which are treated via completely different techniques in a proper $C$-viscosity sense; see \cite{BD15, BDL19-2} and \cite{KoKo17} for $C^{1,\alpha}$ and H\"{o}lder regularity results, respectively.

Here we show that a unified treatment can be employed for both sublinear and superlinear gradient regimes in the $L^p$-viscosity sense. We investigate the optimal regularity achieved under minimal integrability imposed on the source term and on the leading coefficients. In particular, we extend the $C^{1,\alpha}$ results in \cite{BD15} and \cite{BDL19-2}, \cite{Norn19} and \cite{Tei14}. More precisely, in view of \cite{BD15, BDL19-2} we cover sublinear growth Hamiltonians and source term which are possibly unbounded near the boundary; while in light of \cite{Norn19} we are able to deal with superlinear growth Hamiltonians which are subquadratic and unbounded in the gradient.
Furthermore, we obtain $C^{\textrm{Log-Lip}}$ and $C^{1,\textrm{Log-Lip}}$ regularity estimates when the source term attains the borderline and asymptotic scenarios, respectively. This kind of structure has stood out recently \cite{daSilRic19, daST17, PT16, Tei14}, and also appears as the optimal regularity for fully nonlinear nonconvex operators with bounded source terms, see \cite{dosPT16} and \cite{ST}.
Up to our knowledge, these did not seem to be available for general  nonlinear gradient growth structures and unbounded coefficients.
In contrast, when we were finishing writing this paper, we learned about a very recent paper \cite{Wangpointwise2020} in which the authors extend \cite{Norn19}, by proving pointwise regularity for equations with superlinear growth controlled by a bounded quadratic growth. In our work, instead we allow the superlinear gradient coefficient to be unbounded, besides considering also the sublinear growth regime.

\smallskip

Next, we introduce our hypotheses. To this end, we recall the definition of \textit{Pucci's extremal operators}
$$
\displaystyle \mathcal{M}_{\lambda, \Lambda}^-(X)=\inf_{\lambda I \le A\le \Lambda I} \mathrm{tr}(AX) \quad  \text{and}\quad  \mathcal{M}_{\lambda, \Lambda}^+(X)=\sup_{\lambda I \le A\le \Lambda I} \mathrm{tr}(AX), \quad \textrm{ for } X\in \mathbb{S}^n ,
$$
with ellipticity constants $0<\lambda\leq \Lambda< \infty$; see \cite[Ch.2, \S 2]{CafCab} for their properties. Here $I$ is the identity matrix and $\mathbb{S}^n$ as the set of symmetric $n\times n$ matrices with real entrances.
\vspace{0.05cm}

Throughout the text $F: \Omega \times\rn\times\mathbb{S}^n \to \mathbb{R}$ will always be a measurable function, and $H:\Omega\times \rn\times\rn\to \real$ stands for the Hamiltonian term. We impose on them the following structural condition:
\begin{itemize}
\item[(A1)]({{\bf Structure}}) $F(\cdot,0,0)\equiv 0$, and for any $r,s\in \real$, $\xi,\eta\in \rn$, and $X,Y\in \mathbb{S}^n$ we have
\begin{equation}\label{SC}
    \mathcal{M}_{\lambda, \Lambda}^-(X-Y)-H(x,\xi,\eta)
    \leq F(x,\xi,X)-F(x,\eta,Y)\leq\mathcal{M}_{\lambda, \Lambda}^+(X-Y)+H(x,\xi,\eta),
\end{equation}
where $H$ assumes one of the following two forms, for some $q,\varrho\geq p$ with $q,\varrho>n$:
\begin{enumerate}
\item ({\it Sublinear regime}) For $0<m<1$,
\begin{center}
$
H(x,\xi,\eta)= b(x)|\xi-\eta|+\mu (x)|\xi-\eta|^m;
$
\end{center}

\item ({\it Superlinear/Quadratic regime}) For $1< m\leq  2$,
\begin{center}
$
H(x,\xi,\eta) = b(x)|\xi-\eta|+\mu (x)|\xi-\eta|\,(|\xi|^{m-1}+|\eta|^{m-1});
$
\end{center}
\end{enumerate}
where $b\in L^{\varrho}_+ (\Omega)$ and $\mu\in L^q_+ (\Omega)$, for some $q, \varrho\in (n,+\infty)$; with $\mu(x)\equiv \mu \geq 0$ if $m=2$.
\end{itemize}
In the linear regime, i.e. $m=1$, we simply write $H(x,\xi,\eta)= b(x)|\xi-\eta|$, for $b \in L^\varrho_+(\Omega)$, $\varrho>n$.

\bigskip

Simple model cases of $F$ we may have in mind are the extremal operators in the form
\begin{align}\label{def Lm}
\mathcal{L}^\pm_m\,[u]\,\defeq \,\mathcal{M}^{\pm}_{\lambda, \Lambda}(D^2 u) \pm b(x)|Du|\pm \mu(x)|Du|^m , \;\; m\in (0,2], \;b\in L^\varrho_+(\Omega),\,\mu \in L^q_+(\Omega),
\end{align}
which often appear in optimal control problems and game theory.
Note that \eqref{def Lm} satisfies (A1), since for $m\in (0,1)$ one has $|\xi|^m-|\eta|^m\leq |\xi-\eta|^m$; while for $m>1$ the convexity of the function $\xi\mapsto |\xi|^m$ yields $|\xi|^m\ge |\eta|^m+m\langle |\eta|^{m-2}\eta,\xi-\eta\rangle$.
But we actually go beyond, by allowing $F$ to be a Hamilton–Jacobi–Bellman (HJB) or Isaac's operator, which appear as infinitesimal generators of the underlying stochastic processes. For example,
we mention the time dependent HJB equation
\begin{center}
$u_t -\mathrm{tr}(A(x)D^2u)+H(x,Du)=0$ \;\; in $\rn \times (0,+\infty)$,
\end{center}
equipped with the initial boundary condition, where $A$ is the diffusion and $H$ is a convex superlinear Hamiltonian, which are subject to probabilistic formulations, see \cite{Ahomo}.

On the other hand, problems with sublinear gradient growth often relate to a class of singular equations, in the sense that they degenerate when the gradient vanishes. For instance, by \cite[Lemmas 5.2, 5.3]{KoKo17}, a suitable viscosity solution of \vspace{-0.1cm}
\begin{center}
	$\mathcal{M}^+(D^2u)+\mu(x)|Du|^m=0$, \;\; with $0<m<1$ and $\mu\in C(\Omega)\cap L^p(\Omega)$,
\end{center}
is a viscosity solution of
	$|Du|^{-m}\mathcal{M}^+(D^2u)+\mu(x)=0$, see details in Section \ref{section p-Lap}.
Problems of this sort might be seen as a fully nonlinear counterpart of the $p$-Laplace equation, hence their relevance.

In turn, equations whose gradient growth is quadratic -- also called natural in the literature --  have appeared since the late eighties. For instance, we quote regularity problems regarding gradient boundary estimates  (Hopf-Oleinik type results) in nondivergence form studied in \cite{LadUralt}. There, they considered differential linear inequalities in an appropriate weak sense with unbounded ingredients; see also \cite{BGMW}
for some recent related results in the fully nonlinear setting.

In general, fully nonlinear equations with natural gradient growth have been an increasing topic of investigation in the last years, specially impulsed by existence results and a multiplicity phenomenon obtained in \cite{arma2010}. To treat multiplicity as in \cite{multiplicidade}, it demands nontrivial techniques from nonlinear analysis and regularity estimates such as half-Harnack inequalities, and a V\'{a}zquez type strong maximum principle for such kind of equations; see \cite{Sir2018} and \cite{SSvazquez}. As discussed in \cite{multiplicidade}, the latter is the main tool to extend the multiplicity results in \cite{multiplicidade}, as well as those in \cite{NSS2020}, to the unbounded coefficients scenario. Indeed, as pointed out in \cite{SSvazquez}, one applies instead the recent V\'{a}zquez maximum principle from \cite{SSvazquez}. This is possible once we have, as in \cite{multiplicidade}, a Perron type result for unbounded coefficients  -- which we prove here to our general class of equations, in addition to regularity for equations with unbounded explicit zero order term from \cite[Theorem 1, Remark 1.2]{Norn19}; see Section \ref{section perron}.

\smallskip

Now, we are in a position to state our main results, accordingly to $p$ in the integrability of the source term $f\in L^p(\Omega)$.
Recall that $\varrho, q$ indicate the integrability of the coefficients  $b\in L^\varrho_+(\Omega)$ and $\mu\in L^q_+(\Omega)$, while $m$ stands for the growth of the gradient in $|Du|^m$. We assume:
\begin{itemize}
	\item[(A2)]({{\bf Exponents relation}})
Let $p>p_0$ and $\varrho \ge p$, $q\geq p$ with $\varrho, q\in (n, +\infty]$, then:
\begin{enumerate}
	\item ({\it Sublinear regime}) For $0<m<1$, if $p< n$ we suppose $q(mp+(1-m)n)> pn$;
	
	\item ({\it Superlinear subquadratic regime}) For $1< m<  2$ we assume $q(2-m)>n$.

\item ({\it Quadratic regime}) For $m=2$ we set $q=+\infty$.
\end{enumerate}
\end{itemize}
Here,
$p_0=p_0 (n, \lambda, \Lambda)\in [{n}/{2},n)$ is Escauriaza's constant (see \cite[Theorem 1]{Esc93}), which gives a validity's range for the generalized maximum principle (GMP for short), see \cite{CCKS, KSmpite2007}.
In the linear gradient regime $m=1$ no additional assumption is necessary, but the standard drift integrability $\varrho \ge p$, $\varrho>n$.

\begin{rmk}\label{rmk bounds m,q}
When $p_0<p<n$ and $0<m<1$, (A3) asks for $q$ to satisfy $q> \frac{pn}{mp+(1-m)n}$; alternatively, by fixing $q$ one gets a bound from below on $m$ given by $m>\frac{n}{q}\frac{q-p}{n-p}$ if $p_0<p<n$.
Instead, when $m\in (1,2)$, we are asking $q$ to be large enough so that $q>\frac{n}{2-m}$; alternatively by fixing $q$ we obtain the bound from above on $m$ given by $m<2-\frac{n}{q}$.
\end{rmk}

Let us define the function $r\in \real\mapsto W(r)$ which will describe all our local estimates accordingly to the norm of an appropriate source term $f$ norm:
\begin{align}\label{def W geral}
 W=W_{u,\mu,m,\Omega}\,(r)\defeq
\begin{cases}
\,  \| u \|_{L^{\infty} (\Omega)}  + r + \|\mu\|^{\frac{1}{1-m}}_{L^q(\Omega)}  &\textrm{if } m\in (0,1)\\
\, \| u \|_{L^{\infty} (\Omega)}  + r  &\textrm{if } m\in [1,2]\end{cases},
\quad \textrm{$W\le W_0$\, if $m\in (1,2]$.}
\end{align}
Accordingly to each situation we might be emphasizing one or more functions/parameters dependence.

First of all, we mention that H\"{o}lder regularity estimates are available for $L^p$-viscosity solutions of \eqref{F=f} under (A1) and (A2). More generally we only need to have solutions of differential inequalities, as we gather in the next theorem statement.
\begin{teo}[\cite{KoKo17, KSexist2009, arma2010}]\label{Calpha known}
Assume (A2). Let $f\in L^p(\Omega)$ and $u$ be a bounded $L^p$-viscosity solution of the inequalities
	\vspace{-0.1cm}
	$$ \mathcal{L}^-_m[u]\leq f(x) \quad \text{and} \quad \mathcal{L}^+_m[u]\geq -f(x) \quad \text{in} \quad \Omega,$$
where $b\in L^\varrho_+(\Omega)$, $\mu \in L^q_+(\Omega)$ for $q,\varrho\geq p >p_0$ with $q,\varrho>n$; and $q=\infty$ if $m= 2$.	
	Then, there exists $\beta\in (0,1)$ depending on $n,p,q,\varrho,\lambda,\Lambda,\|b\|_{L^\varrho(\Omega)}$ such that $u\in C^{0,\beta}_{\mathrm{loc}} (\Omega)$, and
	\begin{align}\label{estim Calpha local}
\|u\|_{C^{0,\beta}(\overline{\Omega^\prime})} \leq CW_{u,\mu,m,\Omega\,}(\,\|f\|_{L^p(\Omega)}) , \textrm{ \;\; for $W$ as in }\eqref{def W geral},
	\end{align}
	for any subdomain $\Omega^\prime \subset\subset \Omega$,
	where $C$ depends on $n,m,p,\varrho,q,\lambda,\Lambda,\beta, \| b \|_{L^\varrho(\Omega)}, \| \mu \|_{L^{q}(\Omega)}, \mathrm{diam}(\Omega)$, and $\mathrm{dist} (\Omega^\prime,\partial\Omega)$; in addition to a $W_0$ in the case $m\in (1,2]$.
\end{teo}

Indeed, the quadratic case $m=2$ for $p\geq n$, $\mu(x)\equiv \mu$ was proved in \cite[Theorem 2]{arma2010}. However, the same proof there also works for $p\in (p_0,n)$ if one replaces ABP by GMP (that is, weak ABP for $p_0<p<n$); alternatively one can combine \cite[Theorem 6.2]{KSweakharnack} with $m=1$ and the exponential change in \cite[proof of Theorem 2]{arma2010} since H\"{o}lder regularity is applicable for solutions of differential inequalities.
Moreover, by \cite[Theorem 6.2]{KSweakharnack}, the superlinear regime $m\in (1,2-\frac{n}{q})$ can be treated in \cite[Theorem 4.5]{KSexist2009} for $n<q<\infty$.
As observed in \cite{KSexist2009}, the assumption $m<2-n/q$ implies $ p(mq-n)> nq(m-1)$ for $m>1$, if $p< n$ and $q<+\infty$, cf.\ (3.1)(iii) needed for accessing $C^{0,\beta}$ results in \cite{KSweakharnack}.
In the particular case $m=1$, we also refer to \cite{Kry20} for existence and H\"older regularity results when the drift term is in $L^n$.

\vspace{0.01cm}

Finally, the sublinear case $m\in (0,1)$ was taken into account in \cite{KoKo17}. As a matter of fact, \cite[Theorem 5.4]{KoKo17} was stated for continuous coefficients, since $p$-Laplacian-like operators were considered. However, following the proof there, the argument for H\"{o}lder regularity can be performed as in \cite{KSweakharnack} for unbounded coefficients, by making use of \cite[Theorem 4.5 and Corollary 4.6]{KoKo17} directly.

\medskip

The first result of this work is to exhibit an explicit exponent which comes from scaling and realizes the H\"{o}lder regularity above, in light of \cite{daST17} and \cite{Tei14}. This gives us the precise dependence on the integrability of the source term $f$.

In order to access our main results, we need to require a sort of oscillation control on $F$ in the $x$ entry. In order to measure it around a point $x_0 \in \Omega$, we define, as in \cite{Caf89, Swiech},
\begin{align} \label{def beta}
\displaystyle \beta(x,x_0)=\beta_F(x,x_0)\defeq \sup_{X\in \mathbf{S}^n\setminus \{0\} } \frac{|F(x,0,X)-F(x_0,0,X)|}{\|X\|}.
\end{align}
Notice that $\beta(\cdot, \cdot)$ is a bounded function by \eqref{SC}. We so consider the following hypothesis, as in \cite{Caf89, Winter}:

\smallskip

\begin{itemize}
	\item[(A3)]({{\bf Control of oscillation}})
	Given $\theta>0$, there exists $r_0=r_0\, (\theta)>0$ such that
	\begin{align}\label{Htheta} \tag{$H_{\theta}$}
	\left( \,\intav{B_r(x_0)\cap\Omega} \;\beta (x,x_0)^p \d x \right)^{\frac{1}{p}} \leq \theta \, ,\;\textrm{ for all } r\leq r_0.
	\end{align}
\end{itemize}

We so state a sharp H\"{o}lder regularity estimate for solutions of \eqref{F=f} under suitable integrability of the data, in light of \cite[Theorem 1]{Tei14} and \cite[Theorem 3.2]{daST17}.

\begin{teo}[{\bf H\"{o}lder estimate}]\label{Calpha regularity estimates geral}
Assume (A1)--(A3). 
Let $f \in L^p (\Omega)$ and $u$ be a bounded $L^p$-viscosity solution of \eqref{F=f}, for $p \in (p_0, n)$.
Then there exists $\theta$ depending on $n,m,p,q,\varrho,\lambda,\Lambda,\|b\|_{L^\varrho(\Omega)}$ such that if \eqref{Htheta} holds for all $r\leq \min \{ r_0, \mathrm{dist} (x_0,\partial\Omega)\}$, for some $r_0>0$ and for all $x_0 \in \Omega$, this implies that $u\in C^{0,\beta}_{\mathrm{loc}} (\Omega)$ and the estimates \eqref{estim Calpha local} are true for $\beta=2-\frac{n}{p}$.
\end{teo}
It is striking that $\beta$ can be chosen independently of $\varrho,q$; differently from H\"older gradient regularity ahead, for which the integrability of $b,\mu$ plays an important role.

The next step is to establish improved quantitative estimates for the borderline scenario $p=n$.
\begin{defin}
For $\Omega^\prime \subset\subset \Omega$ let us consider the quantity
\begin{align}\label{norm LopLip}
[u]_{C^{0, \text{Log-Lip}}(\overline{\Omega^\prime})} \defeq \sup_{x_0\in \Omega^{\prime} \atop{ 0< r\leq \diam(\Omega) }} \,    \inf_{{\frak{a}} \in \real} \frac{\|u-\frak{a}\|_{L^{\infty}(B_r(x_0)\cap \Omega^{\prime})}}{r|\log r|} .
\end{align}
We say that $u\in C^{0,\text{Log-Lip}} (\overline{\Omega^\prime})$ when $u\in C^0(\overline{\Omega^\prime})$ and $[u]_{C^{0, \text{Log-Lip}}(\overline{\Omega^\prime})}<+\infty$.
\end{defin}

We exploit the approach in \cite[Section 4]{daST17}, \cite[Theorems 1.4]{DKM14} and \cite[Theorem 2]{Tei14}, which does not seem available even for equations without lower order terms when $m< 1$.
\begin{teo}[{\bf Log-Lipschitz estimate}]
\label{Log-Lip estimates}
Assume (A1)--(A3). Let $f \in L^n (\Omega)$ and $u$ be a bounded $L^n$-viscosity solution of \eqref{F=f}. Then, there exists $\theta$, depending on $n,m,p,q,\varrho,\lambda,\Lambda,\|b\|_{L^\varrho(\Omega)}$, such that if \eqref{Htheta} holds for all $r\leq \min \{ r_0, \mathrm{dist} (x_0,\partial\Omega)\}$, for some $r_0>0$ and for all $x_0 \in \Omega$, then $u\in C^{0,\text{Log-Lip}}_{\mathrm{loc}} (\Omega)$, and
\begin{align}\label{W do Log-Lip}
[u]_{C^{0, \text{Log-Lip}}(\overline{\Omega^\prime})} \le C W_{u,\mu,m,\Omega\,}(\,\|f\|_{L^n(\Omega)}) , \textrm{ \;\; for $W$ as in }\eqref{def W geral},
\end{align}
for any subdomain $\Omega^\prime \subset\subset \Omega$,
where $C$ depends only on $\,r_0,n,m,p,q,\varrho,\lambda,\Lambda, \| b \|_{L^\varrho(\Omega)}, \| \mu \|_{L^{q}(\Omega)}, \mathrm{diam}(\Omega)$, and $\mathrm{dist} (\Omega^\prime,\partial\Omega)$; and also on $W_0$ when $m\in (1,2]$.
\end{teo}

This kind of borderline regularity turns out to be effective when dealing with a unified  framework comprising a general Hamiltonian as ours.
However, other kinds of approaches may not be successful in this task, by also depending on the growth of the Hamiltonian.
For instance, we mention the regularity regime connected to integral estimates as in \cite{DKM14, Swiech} when $m=1$, under bounded coefficients.
In particular, the method developed in \cite{Swiech} for $W^{1,n}$ regularity is not directly applicable when $m>1$ -- although it can be easily extended to unbounded coefficients when $m=1$ via the variable changes and approximation lemma in \cite{Norn19}; see also the discussion in \cite[Remark 5.2]{Swiech20}. In Section \ref{section a priori BMO} we show that a similar reasoning applies in the sublinear gradient growth regime $m<1$ as well.

Next, we address to $C^{1, \alpha}$ regularity. We stress that the case $m=2$ was obtained in \cite[Theorem 1.1]{Norn19} for bounded $\mu$; while $m\in (0,1)$ was studied in \cite{BD15, BDL19-2} for continuous and bounded coefficients. Here the novelties are both sublinear and superlinear subquadratic Hamiltonians lying in the unbounded scenario, besides the fact that we present a unified proof for all $m\in (0,2]$.

\begin{teo}[\bf {H\"{o}lder gradient estimate}]
\label{C1,alpha regularity estimates geral}
Assume (A1)--(A3). Let $f  \in L^p (\Omega)$, where $p>n$, and $\Omega\subset\rn$ be a bounded domain. Let $u$ be a bounded $L^p$-viscosity solution of \eqref{F=f}. Then, there exists $\alpha\in (0,1)$ and $\theta=\theta (\alpha)$, depending on $n,m,p,q,\varrho,\lambda,\Lambda,\|b\|_{L^\varrho(\Omega)}$, such that if \eqref{Htheta} holds for all $r\leq \min \{ r_0, \mathrm{dist} (x_0,\partial\Omega)\}$, for some $r_0>0$ and for all $x_0 \in \Omega$, this implies that $u\in C^{1,\alpha}_{\mathrm{loc}} (\Omega)$, and
\begin{align}\label{estim C1,alpha local}
\|u\|_{C^{1,\alpha}(\overline{\Omega^\prime})} \leq CW_{u,\mu,m,\Omega\,}(\,\|f\|_{L^p(\Omega)}) , \textrm{ \;\; for $W$ as in }\eqref{def W geral},
\end{align}
for any subdomain $\Omega^\prime \subset\subset \Omega$,
where $C$ depends only on $\,r_0,n,m,p,q,\varrho,\lambda,\Lambda,\alpha, \| b \|_{L^\varrho(\Omega)}, \| \mu \|_{L^{q}(\Omega)}, \mathrm{diam}(\Omega)$, and $\mathrm{dist} (\Omega^\prime,\partial\Omega)$; in addition to $W_0$ in the case $m\in (1,2]$.
\end{teo}

\begin{rmk}\label{rmk global}
Our results are local, but the same approach developed in \cite{Norn19, Winter} can be used to obtain boundary regularity estimates as well, in particular the global H\"older gradient estimates furnish
\begin{align*}
\|u\|_{C^{1,\alpha}(\overline{\Omega})} \leq CW_{u,\mu,m,\Omega\,}(\,\|f\|_{L^p(\Omega)}+\|u\|_{C^{1,\tau }(\partial\Omega)}) , \textrm{ \;\; for $W$ as in }\eqref{def W geral},
\end{align*}
whenever $\Omega$ is a $C^{1,1}$ bounded domain, $u\in C^{1,\tau}(\partial\Omega)$ for some $\tau\in (0,1)$, and \eqref{Htheta} holds for all $x_0\in \overline{\Omega}$.
\end{rmk}

Theorems \ref{Calpha known}--\ref{C1,alpha regularity estimates geral} can be understood in the following way: given any $m\in (1,2)$, there exists $q>n$ such that $m<2-\frac{n}{q}$, and so the desired regularity estimates hold for the pair $(m,q)$. Moreover, $q=q(m)\to +\infty$ as $m\to 2^{-}$. In other words, the integrability required on $\mu$ needs to compensate the growth of $|Du|$, becoming worse as far as we approach the quadratic regime.
Thus, $m=2$, $q=\infty$ corresponds to the remaining case $m\in [2-n/q,2]$ for bounded $\mu$.
On the other hand, by fixing an integrability level $q$ on $\mu$, by Remark \ref{rmk bounds m,q} one finds the optimal admissible growth we may impose on the Hamiltonian function for the sake of achieving the desired regularity.

This seems reasonable and perhaps necessary in view of scaling properties to our class of equations.
We highlight that a rescaling also produces restrictions subjected to (A2) when considering the Krylov-Safonov's regularity theory (see \cite{Kry-Saf79}) -- for instance, by adding a superlinear gradient term in the crucial growth lemmas, see \cite[proof of Theorem 2]{arma2010}.

Moreover, in Section \ref{section intrinsic} we find a proper range allowed to the parameters of the problem $p,q,\varrho,m$ and $n$ in (A2) such that our regularity results intrinsically depend on the integrability of the source term $f$ -- such that it does not matter the exponent $p$ in the $L^p$-viscosity formulation of the problem but on the RHS. Besides being independent mathematical interest, this determines the well-posedness of our statements, namely in terms of simply viscosity notion. It also easily implies a Perron type result for our equations, see Section \ref{section perron}.

In a nutshell, our results extend, generalize, and improve the former ones (H\"{o}lder, Borderline, Gradient and Hessian estimates) in \cite{Norn19, Swiech, Tei14}, by using a special machinery designed for fully nonlinear models with general Hamiltonian and unbounded ingredients.
It is worth mentioning that our geometric estimates are based on Caffarelli's iteration scheme addressed in \cite{Caf89} (see also \cite{CafCab}). However, differently from \cite{Caf89}, \cite{Swiech} and \cite{Winter}, we perform a simplified rescaling of variable, in light of \cite{Norn19}, which allows us to carry out the proof without using a twice differentiability property of viscosity solutions whose validity is unknown for unbounded coefficients.

The quadratic case is peculiar since, after scaling, the quadratic growth term becomes as dominating as the second order term.
For Log-Lipschitz regularity we explore a variation of the techniques developed in \cite{Norn19, W2} (see also \cite{CafHua03}).
Under the sublinear regime, a new stability-like result which we also prove is fundamental.

As far as further regularity is concerned, we present finer and improved asymptotic $C^{1,1}$ and $C^{2,\alpha}$ regularity estimates for the following model problem:
\begin{align}\label{MC_intro}
F(x,\xi,X)\defeq G(x,X)+h(x,\xi,0), \textrm{ where $|h|\le H$, for an $H$ as in }(A1).
\end{align}
In this case, more restrictive assumptions on the Hamiltonian are unavoidable; here we ask it to be bounded, see also \cite{Wangpointwise2020}.
Regarding the oscillation in the $x$-entry of our operator $F$, as in \cite{Caf89, daST17, Tei14} it is usual to impose $\bar{\alpha}$--H\"{o}lder continuity for some $\bar{\alpha}>0$, such as:
\begin{itemize}
	\item[($\tilde{A}$3)]({{\bf H\"{o}lder oscillation}})
	Given $\theta>0$, there exists $r_0=r_0\, (\theta)>0$ such that
	\begin{align}\label{Htheta2}  \tag{$\widetilde{H}_{\theta}$}
\displaystyle	\left( \,\intav{B_r(x_0)\cap\Omega}\; \beta (x,x_0)^n \d x \right)^{\frac{1}{n}} \leq \theta r^{\bar{\alpha}} \, ,\;\textrm{ for all } r\leq r_0 ,
	\end{align}
\end{itemize}
in addition of $C^{2,\bar{\alpha}}$ regularity estimates of its associated pure second order operator:

\begin{itemize}
	\item[(A4)]({{\bf F-smoothness}})
	Given a matrix $X\in \mathbb{S}^n$ with $F(x_0,0,X)=0$, $L^p$-viscosity solutions $h$ of $F(x_0,0,D^2h+X)=0$ in $B_1$ are $C^{2,\bar{\alpha}}_{\mathrm{loc}}(B_1)$ for some $\bar{\alpha}>0$ and satisfy, for a universal $\Theta>0$,
	\begin{align}\label{hip C2,epsilon}
	\|h\|_{C^{2,\bar{\alpha}}(\overline{B}_r)}\leq \Theta\, r^{-(2+\bar{\alpha})}\|h\|_{L^\infty(B_1)}\;\; \textrm{ for all $ r\in (0,1)$  and $x_0\in B_1$}.
	\end{align}
\end{itemize}

\begin{rmk}
	As observed in \cite[Remark 2.3]{Winter}, the oscillation condition in (A3) measured in the $L^n$-norm is equivalent to an oscillation measured in the $L^\infty$-norm. Analogously, the H\"{o}lder oscillation in $(\tilde{A}3)$ can also be measured in terms of the $L^\infty$-norm.
\end{rmk}

\begin{defin}
We say that $u\in C^{1,\text{Log-Lip}} (\overline{\Omega^\prime})$ if $u\in C^1(\overline{\Omega^\prime})$ and $[Du]_{C^{0, \text{Log-Lip}}(\overline{\Omega^\prime})}<+\infty$ where %as in \eqref{norm LopLip}.
\begin{align}\label{norm LopLip}
[Du]_{C^{0, \text{Log-Lip}}(\overline{\Omega^\prime})} \defeq \sup_{x_0\in \Omega^{\prime} \atop{ 0< r\leq \diam(\Omega) }} \,
\inf_{\mathfrak{l} \,(\mathrm{affine})} \frac{\|u-\mathfrak{l}\|_{L^{\infty}(B_r(x_0)\cap \Omega^{\prime})}}{r^2|\log r|} .
\end{align}
\end{defin}
In order to make the notation uniform, set $|w|_{0,\Omega}\defeq \|w\|_{L^\infty(\Omega)}$, $|w|_{\alpha,\Omega}\defeq \|w\|_{C^\alpha(\overline{\Omega})}$ for $\alpha>0$, and
\begin{align}\label{def W geral BMO}
\overline{W}=\overline{W}_{u,\mu,m,\alpha,\Omega}\,(r)\defeq
\begin{cases}
\,  \| u \|_{L^{\infty} (\Omega)}  + r + |\mu|^{\frac{1}{1-m}}_{ \alpha,\Omega}  &\textrm{if } m\in (0,1)\\
\, \| u \|_{L^{\infty} (\Omega)}  + r  &\textrm{if } m\in [1,2]\end{cases},
\quad \textrm{$\overline{W}\le \overline{W}_0$\, if $m\in (1,2]$.}
\end{align}
Our main result in this direction is displayed in the sequel (compare with \cite{CafHua03}).

\begin{teo}[\bf Log-Lipschitz gradient estimate]\label{C1,logLip theorem}
Assume (A1), (A2) with $b,\mu\in L^\infty_{\mathrm{loc}}(\Omega)$, in addition to ($\tilde{A}$3), (A4), and $n<mq$. Let $f \in L^\infty (\Omega)$ and $u$ be a bounded $L^n$-viscosity solution of \eqref{F=f}, where $F$ has the form \eqref{MC_intro}. Then there exists $\theta$ depending on $n,m,p,q,\varrho,\lambda,\Lambda,\|b\|_{L^\varrho(\Omega)}$, such that if \eqref{Htheta2} holds for all $r\leq \min \{ r_0, \mathrm{dist} (x_0,\partial\Omega)\}$, for some $r_0>0$ and for all $x_0 \in \Omega$, then $u\in C^{1,\text{Log-Lip}}_{\mathrm{loc}} (\Omega)$, and
\begin{align}\label{est C1,LogLip}
[Du]_{C^{0, \text{Log-Lip}}(\overline{\Omega^\prime})} \le C\, \overline{W}_{u,\mu,m,\Omega^{\prime\prime}\,}(\,\|f\|_{L^\infty(\Omega^{\prime\prime})}) , \textrm{ \;\; for $\overline{W}$ as in }\eqref{def W geral BMO},
\end{align}
for any subdomains $\Omega^\prime \subset\subset \Omega^{\prime\prime}\subset\subset\Omega$,
where $C$ depends only on $\,r_0,n,m,p,q,\varrho,\lambda,\Lambda, \| b \|_{L^\infty(\Omega^{\prime\prime})}$, $\| \mu\|_{L^\infty(\Omega^{\prime\prime})}$, $\mathrm{diam}(\Omega)$, and $\mathrm{dist} (\Omega^\prime,\partial\Omega^{\prime\prime})$; and further on $\overline{W}_0$ if $m\in (1,2]$.
\end{teo}

\begin{rmk}\label{rmk hip mq>n}
Note that we always have $m>\frac{n}{q}$ if $m\geq 1$, since $q>n$. Thus, the condition $n<mq$ in Theorem \ref{C1,logLip theorem} is only an additional assumption when $0<m<1$.
\end{rmk}

\begin{teo}[\bf Schauder estimate]\label{Schauder theorem}
Assume (A1), (A2), ($\tilde{A}$3), (A4), and $n<mq$. Let $u$ be a bounded $L^n$-viscosity solution of \eqref{F=f}, where $F$ has the form \eqref{MC_intro}. Then, there exists $\theta$ and $\alpha \in (0,\bar{\alpha})$ depending on $n,m,p,q,\varrho,\lambda,\Lambda,\|b\|_{L^\varrho(\Omega)}$,
such that if \eqref{Htheta2} holds for all $r\leq \min \{ r_0, \mathrm{dist} (x_0,\partial\Omega)\}$, for some $r_0>0$ and for all $x_0 \in \Omega$, and $b,\mu,f\in C^\alpha_{\mathrm{loc}} (\Omega)$, then $u\in C^{2,\alpha}_{\mathrm{loc}} (\Omega)$, and
\begin{align}\label{est Schauder local}
\|u\|_{C^{2,\alpha}(\overline{\Omega^\prime})} \le C\overline{W}_{u,\mu,m,\Omega^{\prime\prime}\,}
(\,\|f\|_{C^\alpha({\Omega^{\prime\prime}} \,) }
) , \textrm{ \;\; for $\overline{W}$ as in }\eqref{def W geral BMO},
\end{align}
for any subdomains $\Omega^\prime \subset\subset \Omega^{\prime\prime}\subset\subset\Omega$,
	where $C$ depends only on $\,r_0,n,m,p,q,\varrho,\lambda,\Lambda,\alpha, \|b\|_{C^\alpha(\overline{\Omega^{\prime\prime}})}$, $\|\mu\|_{C^\alpha (\overline{\Omega^{\prime\prime}})}$, $\mathrm{diam}(\Omega)$,  and $\mathrm{dist} (\Omega^\prime,\partial\Omega^{\prime\prime})$; in addition to $\overline{W}_0$ when $m\in (1,2]$.
\end{teo}

We stress that our geometric approach is particularly refined and flexible in order to be employed in a wide class of problems and Hamiltonians. In effect, we believe that it can be extended to both divergence and nondivergence type operators, with either degenerate or singular structures, or even free boundary problems of elliptic and parabolic type, just to mention a few.

The main tool in order to accomplish Hessian-type regularity is H\"{o}lder regularity estimates for our equations together with equivalence notions of viscosity solutions, that is, our Theorem \ref{C1,alpha regularity estimates geral} and Proposition \ref{prop Lp iff Lupsilon} ahead.
Essentially, it reduces any type of further regularity to be an easy consequence of the corresponding regularity estimates for equations without lower coefficients; in particular it simplifies the local approach developed in \cite{Wangpointwise2020} when $m\ge 1$, and up to our knowledge is completely new for $m<1$.

\smallskip

The rest of the paper is structured as follows. In Section \ref{Preliminaries} we recall some known results used along the text.
In Section \ref{section intrinsic} we present some fundamental results of the theory regarding equivalency of viscosity notions for the class of equations we consider.
In Section \ref{section Approx and Scal} we work on the pivotal scaling and approximation tools, while in Sections \ref{proof main th}, \ref{section further regularity} we give a detailed proof of the main theorems. Finally, Section \ref{section examples} is devoted to examples and applications. In particular, we consider two model cases involving nonconvex operators for which our results are sharp. We also establish \textit{a priori} BMO estimates, some scenarios for Sobolev integral estimates, in addition to regularity for a class of fully nonlinear singular PDEs which degenerates when the gradient vanishes.

\section{Preliminaries}\label{Preliminaries}

Let us start by recalling some definitions and comments about viscosity solutions. \smallskip

In all what follows, we consider an operator $F$ as in \eqref{SC} -- in particular, having the ``sign'' of $\Delta$.

\begin{defin}\label{def Lp-viscosity sol}%[{\bf $L^p$-viscosity solution}]
Let $f\in L^p_{\textrm{loc}}(\Omega)$. We say that $u\in C(\Omega)$ is an $L^p$-viscosity subsolution $($resp. supersolution$)$ of $F(x,u,Du,D^2u)=f(x)$ in $\Omega$ if for $\phi\in  W^{2,p}_{\mathrm{loc}}(\Omega)$, it follows that
\begin{align}\label{limSubsolution1}
\mathrm{ess.}\varlimsup_{y\to x} \{F(y,u(y),D\phi(y),D^2\phi (y))-f(y)\} \geq 0
\;
(\mathrm{ess.}\varliminf_{y\to x} \{F(y,u(y),D\phi(y),D^2\phi (y))-f(y)\} \leq 0 )
\end{align}
whenever $u-\phi$ assumes its local maximum (resp.\ minimum) at $x\in\Omega$.
\end{defin}
It is worth mentioning that the Definition \ref{def Lp-viscosity sol} is formulated for any $p>\frac{n}{2}$. Such a restriction makes all test functions $\phi\in W^{2,p}_{\mathrm{loc}}(\Omega)$ continuous and having a second order Taylor expansion, see \cite{CCKS}.\smallskip
\smallskip

Throughout the text we always consider $b\in L^{\varrho}_+ (\Omega)$, $\mu \in L^q_+ (\Omega)$, and $f\in L^p (\Omega)$ such that
$q, \varrho\geq p > p_0$ and $q, \varrho>n,$
where $p\in (p_0,+\infty)$ is the exponent in Definition \ref{def Lp-viscosity sol} of $L^p$-viscosity solutions.
Anyway $q,\varrho\ge p$ is the minimal condition required on $b$ and $\mu$ to make sense the definition of $L^p$-viscosity solution.

Note that if $\mu\in L^\infty_+(\Omega)$ then in particular $\mu\in L^q_+(\Omega)$ for $q>n$.
In this case we adopt the convention of writing $q$ in a unified way for $m\in (1,2]$: if $m=2$ we understood $n/q$ as being $0$; while for $m<2$ we just take into account some large $q<+\infty$ verifying (A2).

\begin{rmk}\label{Remark C-visc}
When $F$ and $f$ are continuous in $x$, \eqref{limSubsolution1} reduces to an evaluation at the point $x$ provided $\phi\in C^2(\Omega)$. In this case, we say that $u$ is a \textit{$C$-viscosity} sub(super)solution of \eqref{F=f}, see, \cite{user}. These notions are equivalent if for instance $F$ satisfies (A1) with $b, \mu \equiv 0$, $p\geq n$, see, \cite[Proposition 2.9]{CCKS}. In particular, when $F$ only depends on $X$ we will use the $L^p$ and $C$ viscosity concepts interchangeably.
\end{rmk}

We say that $u\in C(\Omega)$ is a \textit{strong} subsolution (resp.\ supersolution) of \eqref{F=f} if $u\in W^{2,p}_{\mathrm{loc}}(\Omega)$ and $u$ satisfies the differential inequality $F(x,Du,D^2u)\ge f(x)$ ($\le f(x)$) at almost every point $x\in \Omega$. This notion is equivalent to Definition \ref{def Lp-viscosity sol} for $u\in W^{2,p}_{\mathrm{loc}}(\Omega)$ when $m\ge 1$, cf.\  \cite[Theorem 3.1, Proposition 9.1]{KSweakharnack}.
In Section \ref{section intrinsic} ahead we extend this equivalence for all $m\in (0,2]$ and possibly unbounded $\mu$, under the exponents relation in (A2).

In any sense, a \textit{solution} $u \in C(\Omega)$ of \eqref{F=f} is simultaneously a subsolution and a supersolution of it.

\smallskip

Next, we recall the Alexandrov-Bakelman-Pucci (ABP for short) maximum principle with unbounded ingredients, in both sublinear and superlinear gradient growth regimes.

\begin{prop}[{\bf{ABP sublinear regime} \cite[Theorem 3.3]{KoKo17}}]\label{wABPsublin}
Assume $p_0 < p \leq \varrho$, $\varrho > n$, and (A2) with $0<m<1$.
If $b\in L_{+}^\varrho(\Omega)$, $\mu \in L^q_{+}(\Omega)$, $f \in L^p_{+}(\Omega)$, and $u \in C(\overline{\Omega})$ is an $L^p$-viscosity subsolution (resp.\ supersolution) of
\begin{align*}
\mathcal{L}^+_m\,[u] \geq f(x)\;\; \mathrm{in}\quad \Omega^{+} \quad
(\, \text{resp}. \,\,\,\mathcal{L}^-_m\,[u] \leq f(x)\;\; \mathrm{in}\;\; \Omega^{-} \,),
\end{align*}
where $ \Omega^{\pm} = \{x \in \Omega: \pm  u(x)>\max_{\partial \Omega} u^{\pm}\}$, then we have
\begin{align*}
\max_{\overline{\Omega}} u \leq \max_{\partial \Omega} u^{+} +C\left\{\|f\|_{L^p(\Omega^{+})}+ \|\mu\|^{\frac{1}{1-m}}_{L^q(\Omega^{+})} \right\}
\;\left(\,\min_{\overline{\Omega}} u \geq -\max_{\partial \Omega} u^{-} -C\left\{\|f\|_{L^p(\Omega^{-})}+ \|\mu\|^{\frac{1}{1-m}}_{L^q(\Omega^{-})} \right\}\,\right),
\end{align*}
for a positive constant $C=C(n, \lambda, \Lambda, p, q, \varrho, m, \|b\|_\varrho,\mathrm{diam}(\Omega))$.
\end{prop}

\begin{prop}[{\bf ABP superlinear regime \cite[Theorem 2.6]{KSweakharnack}}]\label{wABPsuperlin}
Let $m\in (1,2]$, $f \in L^p(\Omega)$, $b\in L_{+}^\varrho(\Omega)$, $\mu \in L_{+}^q(\Omega)$; with $\mu\in L^\infty_+(\Omega)$ if $m=2$. Let $u \in C(\overline{\Omega})$ be an $L^p$-viscosity subsolution of
\vspace{-0.1cm}
\begin{center}
$\mathcal{L}^{+}_m[u] \geq  f(x)$\; in \; $\Omega$.
\end{center}
\begin{enumerate}[(i)]
\item For $q \geq  p> n$, there exist constants $\delta = \delta(n, \varrho,\lambda, \Lambda, m, p,q, \|b\|_\varrho)>0 $ and $C =C(n, \lambda, \Lambda, m, p,q,\varrho,$  $\|b\|_\varrho, \mathrm{diam}(\Omega))>0$ such that
\vspace{-0.3cm}
\begin{align*}
\textrm{if }\;\; \|f\|_p^{m-1}\|\mu\|_q<\delta \quad
 \textrm{ then }\quad \max_{\overline{\Omega}} u \leq \max_{\partial \Omega} u^+ +C\left(\|f\|_n + \|f\|^m_p\|\mu\|_q \right).
\end{align*}
\vspace{-0.5cm}

\item Let $p_0 < p \leq n < q$ satisfy $p> \frac{nq(m-1)}{mq-n}$, and set $a_0 = 0$, $ a_k = \sum_{j=0}^{k-1} m^j$ for $k \ge 1$. There exist an integer
$N = N(n, m, p, q)\geq 1$, and constants $\delta = \delta(n, \lambda, \Lambda, m, p, q, \|b\|_\varrho) > 0$ and $C =C(n, \lambda, \Lambda, m, p, q,\varrho, \|b\|_\varrho, \mathrm{diam}(\Omega))> 0$ such that
\vspace{-0.3cm}
\begin{align*}
\textrm{if }\;\; \|f\|_p^{m^N(m-1)}\|\mu\|^{m^N}_q<\delta\quad
 \textrm{ then }\quad
 \max_{\overline{\Omega}} u \leq \max_{\partial \Omega} u^+ +C\sum_{k=0}^{N+1} \|f\|^{m^k}_p\|\mu\|^{a_k}_q .
\end{align*}
\end{enumerate}
\vspace{-0.3cm}
A similar result with minima holds for $L^p$-viscosity supersolutions of $\mathcal{L}^{-}_m[u] \leq  f(x)$ in $\Omega$.
\end{prop}

A consequence of superlinear ABP is the stability of the notion of viscosity solutions, i.e., the limit of a sequence of viscosity solutions turns out to be a viscosity solution of the limiting equation.

\begin{prop}[{\bf Stability for superlinear growth}]\label{stability}
	Let $F$, $F_k$ operators satisfying (A1), $b\in L^\varrho_+ (\Omega), \mu\in L^q_+ (\Omega)$, $f, \, f_k\in L^p(\Omega)$. Assume (A2) with $m\in [1,2]$. Let $u_k\in C(\Omega)$ be an $L^p$-viscosity subsolution $($resp.\ supersolution$)$ of
	$$
	F_k(x,Du_k,D^2u_k)\geq f_k(x) \;\;\;\textrm{in} \;\;\;\Omega\quad (\text{resp}.\leq f(x))\quad\textrm{ for all }\; k\in \n .
	$$
	Suppose  $u_k\rightarrow u$ in $L_{\mathrm{loc}}^\infty  (\Omega)$ as $k\rightarrow \infty$ and for each ball $B\subset\subset \Omega$ and $\varphi\in W^{2,p}(B)$, setting
	\begin{align*}
	g_k(x)\defeq F_k(x,D\varphi,D^2\varphi)-f_k(x) \,\,\,\,\, \text{and}\,\,\,\,
	g(x)\defeq F(x,D\varphi,D^2\varphi)-f(x),
	\end{align*}
	we have
	$ \| (g_k-g)^+\|_{L^p(B)} , (\| (g_k-g)^-\|_{L^p(B)}) \rightarrow 0 $ as $ k\rightarrow \infty.$
	Then $u$ is an $L^p$-viscosity subsolution (resp. supersolution) of
	$$
	F(x, Du, D^2u)\geq f(x) \quad (\text{resp.} \,\,\leq f(x))\,\,\, \text{in} \quad \Omega.
	$$
	Furthermore, if $F$ and $f$ are continuous in $x$, then it is enough that the above holds for every $\varphi\in C^2 (B)$; case in which  $u$ is a $C$-viscosity subsolution (resp. supersolution) of \eqref{F=f}.
\end{prop}

For the proof of Proposition \ref{stability} when $m>1$ we refer \cite{arma2010} and \cite[Proposition 9.4]{KSweakharnack}.
It does not seem to hold in general in the sublinear regime.
However, in the context of Lemma \ref{AproxLem} proof we show that this is true under a smallness assumption on the $L^q$-norm of the function $\mu$.

Next, we recall some results about pure second order profiles $F(D^2 u)$, i.e. uniformly elliptic operators $F$ depending only on $X$ (so Lipschitz continuous in $X$) and satisfying $F(0)=0$.

Let us consider solutions of the operator $F(X)$ in the $C$-viscosity sense, in view of Remark \ref{Remark C-visc}.

\begin{prop}\label{prop F(X)}
\begin{enumerate}[(i)]
\item If $F(D^2 u)=0$ in $B_1$, then $u\in C^{1,\bar{\alpha}} (\overline{B}_{1/2})$ for some universal $\bar{\alpha}\in (0,1]$ and there exists $K_2>0$, depending on $n,\lambda$ and $\Lambda$, such that
	$\|u\|_{C^{1,\bar{\alpha}} (\overline{B}_{1/2})} \leq K_2 \, \|u\|_{L^\infty (B_1)};$
\vspace{-0.1cm}	
	
\item For $\psi\in C(\partial B_1)$, there exists a unique solution $u\in C(\overline{B}_1)\cap C^{1,\bar{\alpha}}_{\mathrm{loc}} ({B}_{1})$ of	$F(D^2 u)= 0$  in  $B_1$, $u = \psi $  on $\partial B_1$, satisfying $\|u\|_{C^{1,\bar{\alpha}} (\overline{B}_{1/2})} \leq K_2 \, \|u\|_{L^\infty (B_1)} , $ for some universal $\bar{\alpha}\in (0,1)$ and $K_2>0$ depending only on $n,\lambda$ and $\Lambda$.
\end{enumerate}
\end{prop}
For a proof, see \cite[Corollaries 5.4, 5.7]{CafCab}, \cite{Tru88}, see also \cite[Propositions 2.9--2.11]{Norn19}.

We highlight that, without convexity/concavity assumptions on $F$, the $C^{1, \alpha}$ regularity in Proposition \ref{prop F(X)} (i) is, in general, the optimal one. In effect, this is accomplished due to Nadirashvili and Vl\u{a}du\c{t}'s counterexamples, see \cite{Na3} and references therein. See Section \ref{section examples} for additional comments about $C^{2, \alpha}$ estimates to homogeneous problems with regular coefficients.

Next, we recall a global version of Theorem \ref{Calpha known} which is used in compactness arguments.
\begin{prop}[\cite{KoKo17, KSexist2009, arma2010}]\label{Cbeta global}
Assume the hypotheses of Theorem \ref{Calpha known} and, in addition, $u\in C (\overline{\Omega})\cap C^\tau (\partial\Omega)$ and $\Omega$ satisfies a uniform exterior cone condition with size $L>0$, then there exists $\beta_0= \beta_0 (n,m,p,q,\varrho,\lambda, \Lambda, L,\|b\|_{L^\varrho(\Omega)}) \in (0,1)$ and $\beta = \min (\beta_0, \frac{\tau}{2})$ such that
\begin{align}\label{estim Calpha global}
\|u\|_{C^{0, \beta}(\overline{\Omega})} \leq K_1
W_{u,\mu,m,\Omega\,}(\,\|f\|_{L^p(\Omega)}+ \|u\|_{C^\tau (\partial\Omega)}) , \textrm{ \;\; for $W$ as in }\eqref{def W geral},
\end{align}
where $K_1$ depends on $n,m,p,q,\varrho,\varrho,\lambda,\Lambda,L, \|b\|_{L^\varrho(\Omega)}, \|\mu\|_{L^q(\Omega)}, \mathrm{diam} (\Omega)$; in addition on $W_0$ if $m\in (1,2]$.
Here $K_1$ remains bounded if these quantities are bounded.
\end{prop}

To finish the section, we recall some useful inequalities, which we are going to use without mentioning throughout the text, in order to deal with the superlinear and sublinear power terms. For $a, b \ge 0$,
$$
\left\{
\begin{array}{rcl}
  a^m+b^m\leq (a+b)^m\leq 2^{m-1}(a^m+b^m) & \text{if} & m\geq 1 \\
  2^{m-1}(a^m+b^m)\leq (a+b)^m\leq a^m+b^m & \text{if} & m\in (0,1].
\end{array}
\right.
$$

\section{Equivalence notions of viscosity solutions}\label{section intrinsic}

In this section, we analyze the conditions on $p,q,\varrho,m,$ and $n$ for which our regularity results do not depend on the exponent $p$ in the $L^p$-viscosity sense of solution, but on the integrability of $f$.

We mention that such difficulty never appears in \cite{Tei14} because lower order terms are not considered there.
Bounded Hamiltonians with linear gradient growth (i.e.\ $m=1$) can be treated in light of \cite[Theorem 2.1(iii)]{CKSS}.
Its extension to unbounded coefficients is given in \cite[Proposition 2.9]{tese} when the exponents are larger than $n$; while for $m=2$ and bounded $\mu$ it is proven in \cite[Proposition 2(iii)]{Swiech20}.

In what follows, we provide a proof that covers the preceding results, but also extend them to sublinear and superlinear gradient growths with unbounded coefficients in a unified fashion.

For the sake of generality and further reference, throughout this section, exclusively, we consider our equations depending also on $u$, without restrictions on the continuity modulus. For instance,
\begin{align}\label{H SC on u}
\textrm{$F_r(x,\xi,X)\defeq F(x, r,\xi,X)$ satisfies (A1), \;
	$|F(x,r,\xi,X)-F(x,s,\xi,X)|\le d(x)\,\omega (|r-s|)$,}
\end{align}
for all $r,s\in \real$, where $d\in L^p_+(\Omega)$, and $\omega$ is a continuous function with $\omega (0)=0$.

\begin{prop}\label{prop Lp iff Lupsilon}
Assume (A2), \eqref{H SC on u},  in addition to
	\begin{align}\label{H equiv visc}
q>\varrho>p \, , \;\;\; (m-1) (n-p)\varrho q \le pn (q-\varrho) \;\;\quad\textrm{if\, $m\in (1,2)$}.
	\end{align}Let $f\in L^\upsilon_{\mathrm{loc}}(\Omega)$, for some $\upsilon >p$.	Then $u$ is an $L^p$-viscosity subsolution (resp.\ supersolution) of
	\begin{align}\label{eq F com u =f}
	\textrm{$F(x,u,Du,D^2u)\ge f(x)$\; in\, $\Omega$\qquad  (resp.\ $F(x,u,Du,D^2u)\le f(x)$ in $\Omega$)}
	\end{align}
	if and only if $u$ is an $L^\upsilon$-viscosity subsolution (supersolution) of \eqref{eq F com u =f}.
\end{prop}

\begin{rmk}
Note that if $m\in (1,2)$ and $q>\varrho>p\ge n$ then \eqref{H equiv visc} is always satisfied.
\end{rmk}

First we need the following lemma.
The case $m=1$ is proven in \cite[Lemm 2.10]{tese}, while for $m=2$ and constant $\mu$ in \cite[Proposition 2(ii)]{Swiech20}. Here, we give a unified proof for $m\in (0,2]$, in the spirit of \cite{tese}.
\begin{lem}\label{lema strict}
	Assume (A2), \eqref{H SC on u}, and
	\eqref{H equiv visc}.
	Then the maximum $($minimum$)$ in Definition \ref{def Lp-viscosity sol} can be replaced by a strict maximum $($minimum$)$.
\end{lem}

\begin{proof}
	Suppose on the contrary that there exist $\varepsilon,r >0$, $x_0\in \Omega$, and $\phi\in W^{2,p}(B_{2r}(x_0))$ such that $u-\phi$ attains a local maximum at $x_0$, say $u-\phi \leq (u-\phi) (x_0)$ in $B_{2r}=B_{2r}(x_0)$, but
	\begin{align}\label{strict eq varphi}
	F(x,u(x),D\phi (x), D^2\phi (x))\leq f(x)-\varepsilon\quad\textrm{a.e. in }B_{2r}(x_0).
	\end{align}	
	
	Let us consider $\delta\in (0,1)$ such that $4\delta r^2 (2+n\Lambda)\le \varepsilon/8$, and
	\begin{center}
		$\mathcal{\widetilde{L}}^+[u]=\mathcal{M}^+(D^2 u)+\widetilde{b}(x)|Du|$,	\qquad
		$\widetilde{b}(x)=
		\begin{cases}
		b(x) \;\textrm{ if } m\in (0,1]\\
		b(x)+\mu (x)|D\phi (x)|^{m-1}\;\textrm{ if } m\in (1,2],
		\end{cases}
		$
		
		\smallskip
		
		$g_\delta (x) \defeq
		\begin{cases}
		-4\delta \widetilde{b}(x) |x|^3 +\varepsilon /8 \;\textrm{ if } m\in (0,1]\\
		-4\delta \widetilde{b}(x) |x|^3 -(4\delta \widetilde{b}(x) |x|^3)^{m-1}\mu (x)+\varepsilon /8  \;\textrm{ if } m\in (1,2].
		\end{cases}$
	\end{center}
	
	We first claim that our hypothesis \eqref{H equiv visc} imply $\widetilde{b}\in L^\varrho(B_{2r})$. Indeed, first notice that this is obvious if either $m\in (0,1]$ or $m=2$, by the definition of $\widetilde{b}$ and (A2).  Instead, when $m\in (1,2)$ we apply H\"{o}lder's inequality to get 
\begin{center}
		$\|\mu |D\phi|^{m-1}\|_{L^{\varrho}(B_1)} \le \|\mu\|_{L^q(B_1)} \,\| |D\phi|^{m-1}\|_{L^{q^\prime}(B_1)}= \|\mu\|_{L^q(B_1)}\, \| D\phi\|^{m-1}_{L^{(m-1)q^\prime}(B_1)}$\quad  where\, $\frac{1}{\varrho}=\frac{1}{q}+\frac{1}{q^\prime}$.
	\end{center}
	The case $p>n$ is simpler since $\phi \in W^{2,p}\subset C^1$. If $p=n$ then Sobolev inequality yields $D\phi \in W^{1,n}\subset L^{s}$ for all $s<+\infty$, in particular for $s=(m-1)\frac{\varrho q}{q-\varrho}<+\infty$, and so $\|\widetilde{b}\|_{L^{\varrho}(B_1)} \le \|b\|_{L^{\varrho}(B_1)} +\|\mu\|_{L^q(B_1)} \| D\phi\|^{m-1}_{L^r(B_1)}<\infty$.
	On the other hand, if $p_0<p<n$ one uses $D\phi \in W^{1,p}\subset L^{p^\star}$ for $p^\star=\frac{pn}{n-p}$ since $(m-1)\frac{\varrho q}{q-\varrho}\le p^\star$, thus
	$\|\widetilde{b}\|_{L^{\varrho}(B_1)} \le  \|b\|_{L^{\varrho}(B_1)} +C\|\mu\|_{L^q(B_1)} \| D^2\phi\|^{m-1}_{L^{p^\star}(B_1)}<\infty$.
	
	\medskip
	
	Set $\phi_\delta \defeq \phi +\delta |x|^4+\psi_\delta$, where $\psi_\delta$ as the strong solution of
	\begin{align}\label{eq psi delta}
	\left\{
	\begin{array}{rclcc}
	\widetilde{\mathcal{L}}^+[\psi_\delta]+\mu (x)|\psi_\delta|^m &=& g_\delta (x) & \mbox{in}& B_{\gamma r} \\
	\psi_\delta &=& 0 & \mbox{on}& \partial B_{\gamma r},
	\end{array}
	\right.
	\end{align}
	given by \cite[Proposition 2.4]{KSexist2009} if $m\in[1,2]$, and by \cite[Theorem 3.1]{KoKo17} if $m\in (0,1)$.
	Here $\gamma\in (0,1]$ is taken so that
	$
	\|\mu\|_{L^q(\Omega)} \|g_1\|_{L^p(\Omega)} \gamma\leq  \varepsilon_1,
	$
	for $\varepsilon_1$ defined in \cite{KSexist2009} when $m\in (1,2]$; while $\gamma=1$ for $m\in (0,1]$.
	
	\begin{claim}\label{lema quadro}
		$u-\phi_\delta$ attains a strict local maximum in $B_r$ for small $\delta>0$.
	\end{claim}
	
	\begin{proof} We have
		$\mathcal{\widetilde{L}}^+[\psi_\delta\, \delta^{-\kappa} ]=g(x) $ a.e.\ in $B_{\gamma r}$ where $\kappa=1$ if $m\in (0,1]$ while $\kappa=m-1$ if $m\in (1,2]$,
		and $g(x)\defeq g_\delta (x)\delta^{-\kappa}$.
		Since $g\in L^\varrho(\Omega)$, note that Chebyshev's inequality
		$|\{|g|>t\}|\leq t^{-\varrho} \|g\|_{L^{\varrho}(B_1)}$ yields
		\begin{align*}
		|B_{\gamma r}\cap \{g<0\}|&=\begin{cases}
		\, |B_{\gamma r}\cap \{4 \widetilde{b}(x)|x|^3 >+\frac{\varepsilon}{8\delta}  \, \}| \;\textrm{ if } m\in (0,1]
		\smallskip\\
		\, |B_{\gamma r}\cap \{
		4\delta^{2-m} \widetilde{b}(x) |x|^3 +(4 |x|^3)^{m-1}\mu (x)>\frac{\varepsilon}{8\delta^{m-1}}  \, \}| \;\textrm{ if } m\in (1,2];
		\end{cases}\\
		&\le \frac{C}{\varepsilon^\varrho}
		\begin{cases}
		\,{\delta}^\varrho \,\| \widetilde{b}\|_{L^{\varrho}(B_1)}  \;\;\textrm{ if } m\in (0,1]
		\smallskip\\
		\, \delta^{(m-1)\varrho}\,
		\{\| \widetilde{b}\|_{L^{\varrho}(B_1)}+  \|\mu\|_{L^{\varrho}(B_1)} \} \; \;\textrm{ if } m\in (1,2] ;
		\end{cases}\\
		& \le {C} \,\{\| \widetilde{b}\|_{L^{\varrho}(B_1)}+  \|\mu\|_{L^{\varrho}(B_1)} \} \, ({\delta^{\kappa}}/\varepsilon)^\varrho.
		\end{align*}
		Next, by $L^\infty$ estimate in the respective existence theorems produces
		\begin{align*}
		\frac{\psi_\delta}{\delta^\kappa} &\leq C_{A}\,\|g^-\|_{L^p (B_r)} \leq C_A \, |B_r\cap \{g<0\}| ^{\frac{1}{p}-\frac{1}{\varrho}}\, \|g^-\|_{L^\varrho (B_r)}\\
		&\leq  {C}
		\,(\| \widetilde{b}\|_{L^{\varrho}(B_1)}+  \|\mu\|_{L^{\varrho}(B_1)} )^{\frac{\varrho-p}{p\varrho}} \,
		\{\varepsilon + \widetilde{b}\|_{L^{\varrho}(B_1)}+  \|\mu\|_{L^{\varrho}(B_1)} \}\,
		({\delta^{\kappa}}/\varepsilon)^{\frac{\varrho-p}{p}} \xrightarrow[\delta\rightarrow 0^+]{} 0\,
		\end{align*}
		for all $x\in B_{\gamma r}$, since $\varrho>p$. In particular, $\psi_\delta (0) < \delta^\kappa (\gamma r)^4$ for small $\delta >0$, so we have on $ \partial B_{\gamma r}$ that $ u-\phi_\delta =u-\phi -\delta^\kappa r^4 < (u-\phi) (0)  -\psi_\delta (0)= (u-\phi_\delta)(0).$
		Hence $u-\phi_\delta$ attains a strict local maximum at some point in $B_{\gamma r}$, say in the ball $B_s (x_0)\subset B_{\gamma r}(0)$, and the claim is proved.
	\end{proof}
	
	Therefore, by applying the definition of $L^p$-viscosity solution with strict local maximum at the point $x_0$, with respect to the test function $\phi_\delta$, one finds
	\begin{center}
		$F(x,u(x), D\phi_\delta (x), D^2\phi_\delta (x) )\geq f(x) -\varepsilon/2$ \;a.e. $x\in B_s (x_0)$.
	\end{center}
	So, using the latter and \eqref{strict eq varphi}, we obtain
	\begin{align*}
	\varepsilon/2&\leq F(x,u(x), D\phi_\delta (x), D^2\phi_\delta (x) ) - F(x,u(x), D\phi (x),D^2 \phi (x)) \\
	&\leq \mathcal{M}^+ (D^2(\delta |x|^4+\psi_\delta) )
	+ H(x, D\phi_\delta, D\phi )\\
	&\leq  4\delta |x|^2 (2+n\Lambda)+ 4\delta  \widetilde{b}(x)|x|^3 +(4\delta |x|^3)^{m-1}\mu (x) +\mathcal{\widetilde{L}}^+ [\psi_\delta ] +\mu (x) |D\psi_\delta|^m \leq \varepsilon /4,
	\end{align*}
	a.e.\ $x\in B_s (x_0)$. Here we have used \eqref{eq psi delta} and the choice of $\delta$, in addition to
	\begin{center}
		$|D\phi_\delta|^{m-1} +|D\phi|^{m-1}\le |D(\phi_\delta -\phi)|^{m-1}+2|D\phi|^{m-1}\le (4\delta|x|^3)^{m-1}+|D\psi_\delta|^{m-1}+2|D\phi|^{m-1}$
	\end{center}
	in the case $m\in (1,2]$, which derives a contradiction.
\end{proof}

\begin{proof}[Proof of Proposition \ref{prop Lp iff Lupsilon}]
	It is immediate by Definition \ref{def Lp-viscosity sol} that $L^p$-viscosity solutions are $L^\upsilon$-viscosity, since $f\in L^\upsilon_{\mathrm{loc}}(\Omega)$, $\upsilon>p$. Suppose by contradiction that $u$ is an $L^\upsilon$-viscosity but not $L^p$-viscosity subsolution of \eqref{eq F com u =f}. Then, there exist $x_0\in \Omega$, $\varepsilon>0$, and $\varphi\in W^{2,p}(B_{2r}(x_0) \setminus W^{2,\upsilon}(B_{2r}(x_0))$ such that $u-\varphi$ attains a local maximum at $x_0$ in $B_{2r}=B_{2r}(x_0)$, but
	\begin{align}\label{eq abs LniffLp}
	F(x,u,D\varphi,D^2\varphi)\leq f(x)-\varepsilon\;\; \textrm{ a.e. in } B_r(x_0).
	\end{align}
	By Lemma \ref{lema strict} we may assume there exists $\delta>0$ so that  $(u-\varphi)(x_0)=0$ and $u-\varphi\leq -\delta$ on $\partial B_r(x_0)$.	
	
	\vspace{0.05cm}
	
	Let $\varphi_k\in W^{2,\upsilon}({B}_{2r})$ with $\varphi_k \rightarrow \varphi$ in $W^{2,p}(B_r)$. Note that $\varphi_k \rightarrow \varphi$ uniformly in $B_{2r}$ since $p>n/2$. Then, for large $k$ we have
	$(u-\varphi_k){|_{\partial B_r}} \le -\delta/ 2 \leq (u-\varphi_k) (x_0)$.
	Hence, the definition of $u$ as an $L^\upsilon$-viscosity solution of \eqref{eq F com u =f} yields
	\begin{align}\label{eq2 LniffLp}
	F(x,u,D\varphi_k,D^2\varphi_k)\geq f(x)-\varepsilon/2\;\; \textrm{ a.e. in } B_r(x_0).
	\end{align}
	
	By putting together \eqref{eq abs LniffLp} and \eqref{eq2 LniffLp} we find, for large $k$,
	\begin{align*}
	n(\lambda +\Lambda) |D^2 (\varphi_k-\varphi)| + H(x, D\varphi_k,D\varphi)\ge  F(x,u,D\varphi_k,D^2\varphi_k)-F(x,u,D\varphi,D^2\varphi)\geq \varepsilon/2\; \textrm{ a.e. in } B_r(x_0).
	\end{align*}
	
	Case 1) $p\in (p_0,n)$. Note that, since $p\ge 1$, and $p<n$ then in particular $n>1$, i.e.\ $n\ge 2$.
	
	\smallskip
	
	Now, we integrate in $B_r(x_0)$ both nonnegative sides of the inequality above, by writing $H(x,\xi,\eta)=b(x)|\xi-\eta|+V(x,\xi,\eta)$, and use H\"{o}lder inequality to obtain
	\begin{align}\label{LniffLp integral}
	\varepsilon/2\, |B_r| &\le n(\lambda +\Lambda)\int_{B_r(x_0)} |D^2 (\varphi_k-\varphi)|
	+\int_{B_r(x_0)} H(x,\varphi_k,\varphi) \d x \nonumber \\
	& \le n(\lambda +\Lambda) |\Omega|^{1-\frac{1}{p}} \|D^2 (\varphi_k-\varphi)\|_{L^p(B_r)} +
	\|b\|_{L^\varrho (\Omega)} \|D(\varphi_k-\varphi)\|_{L^{\varrho^\prime}} +\int_{B_r}  V(x,\varphi_k,\varphi) \d x ,
	\end{align}
	where $\varrho^\prime =\frac{\varrho}{\varrho-1}<\frac{n}{n-1}\le \frac{pn}{n-p}=p^\star$, since $\varrho>n$ and $p\ge 1$. Thus, Sobolev continuous inclusion yields
	\begin{align}\label{LniffLp2 integral}
	\|D(\varphi_k-\varphi)\|_{L^{\varrho^\prime}(B_r(x_0))} \le |\Omega|^{\frac{1}{\varrho^\prime}-\frac{1}{p^\star}} \|D(\varphi_k-\varphi)\|_{L^{p^\star}(B_r(x_0))}\le  C\, \| D^2(\varphi_k-\varphi)\|_{L^p(B_r(x_0))}.
	\end{align}
	
	\vspace{0.05cm}
	
	Now, we claim that
	\begin{align}\label{LniffLp3 integral}
	\int_{B_r(x_0)} V(x,\varphi_k,\varphi) \d x\, \le  C\, \| D^2(\varphi_k-\varphi)\|^{\min(m,1)}_{L^p(B_r(x_0))} \quad \textrm{for }m\in (0,2].
	\end{align}
	Indeed, if $m\in (0,1]$, then
	\begin{align*}
	\int_{B_r(x_0)} V(x,\varphi_k,\varphi) \d x\,&\le\int_{B_r(x_0)} \mu(x)|D
	\varphi_k- D\varphi|^m \d x \, \le \|\mu\|_{L^q (\Omega)} \||D(\varphi_k-\varphi)|^m\|_{L^{q^\prime}}\\
	& =\|\mu\|_{L^q (\Omega)} \|D(\varphi_k-\varphi)\|^m_{L^{mq^\prime}}
	\le \|\mu\|_{L^q (\Omega)} |\Omega|^{\frac{1}{q^\prime}-\frac{m}{p^\star}}\|D(\varphi_k-\varphi)\|^m_{L^{p^\star}(B_r(x_0))},
	\end{align*}
	by H\"{o}lder's inequality as in \eqref{LniffLp2 integral}, since $m q^\prime \le q^\prime =\frac{q}{q-1}<\frac{n}{n-1} \le \frac{np}{n-p}=p^\star$ for $q>n$ and $p\ge 1$. Thus, again Sobolev embedding permits us concluding \eqref{LniffLp3 integral}.
	
	On the other hand, if $m\in (1,2)$ then we use the generalized H\"{o}lder inequality to write
	\begin{align*}
	\int_{B_r(x_0)} V(x,\varphi_k,\varphi) \d x\,&\le \int_{B_r(x_0)} \mu(x)|D (\varphi_k-\varphi)| (|D\varphi|^{m-1} + |D\varphi_k|^{m-1}) \d x \\
	&\le \|\mu(x) \|_{L^q(B_r)} \| D (\varphi_k-\varphi)\|_{L^{q_1}(B_r)} \{\,\|D\varphi\|^{m-1}_{L^{(m-1)q_2}(B_r)} + \|D\varphi_k\|^{m-1}_{L^{(m-1)q_2}(B_r)} \}  ,
	\end{align*}
	where $\frac{1}{q}+\frac{1}{q_1}+\frac{1}{q_2}=1$, and we choose $q_1=q_2(m-1)$. So,
	$q_1= \frac{mq}{q-1}<\frac{mn }{n-2+m} \le n\le  \frac{pn}{n-p}=p^\star$, since $q>\frac{n}{2-m}$ from (A2), and $n\ge 2$ (recall that we are in Case 1). Then, again Sobolev's continuous inclusion produces \eqref{LniffLp3 integral}.
	
	\vspace{0.07cm}
	
	Finally, if $m=2$ it is simpler, since $2\le n\le p^\star$, and
	\begin{align*}
	\int_{B_r(x_0)} V(x,\varphi_k,\varphi) \d x\,&\le \|\mu \|_{L^\infty (\Omega)}\int_{B_r(x_0)} |D (\varphi_k-\varphi)| (|D\varphi| + |D\varphi_k|) \d x \\
	& \le \|\mu \|_{L^\infty (\Omega)} \|D(\varphi_k-\varphi)\|_{L^2(B_r)} \{\,\|D\varphi\|_{L^2(B_r)} + \|D\varphi_k\|_{L^2(B_r)} \}  .
	\end{align*}
	
	Hence, by plugging the integral estimates in \eqref{LniffLp2 integral} and \eqref{LniffLp3 integral} into \eqref{LniffLp integral}, we obtain
	\begin{center}
		$\varepsilon/2 \,|B_r|\,\le C\, \|\varphi_k-\varphi\|_{W^{2,p}(B_r(x_0))} \to 0$ \;\; as  $k\rightarrow +\infty$,
	\end{center} which is impossible for $\varepsilon >0$ fixed.
	
	\smallskip
	
	Case 2) $p\ge n$.
	In this case Sobolev embedding yields  $W^{2,p}\subset W^{2,n}\subset W^{1,s}$ for all $s<+\infty$; the proof is easier, since we do not have problems with the integrability of the gradient to be less than $p^\star$.
\end{proof}

\section{The essential arranging tools}\label{section Approx and Scal}

Our first step towards accessing regularity estimates of gradient type is to approximate our equation by the corresponding homogeneous one with frozen coefficients, which already enjoys ``good regularity estimates''.

\subsection{Approximation Lemma}\label{section lemmas}

Let us consider, as in \cite[Chapter 8]{CafCab} and \cite[Lemma 3.1]{Norn19}, a slightly smaller version of $\beta_F$ given by
\begin{align} \label{def beta bar}
\displaystyle \bar{\beta}(x,x_0)=\bar{\beta}_F (x,x_0) \defeq \sup_{X\in \mathbf{S}^n} \frac{|F(x,0, X)-F(x_0,0,X)|}{\|X\|+1}.
\end{align}
We point out that the control of oscillation considering  \eqref{def beta bar} in place of \eqref{def beta} usually produces estimates which are not optimal, in the sense of appearing $1$ on the right hand side. This is irrelevant when only \textit{a priori} bounds are considered, as in \cite{multiplicidade}, while it could be an issue in some stability arguments.

In what follows, $\bar{\beta}_F$ will be an auxiliary function used in approximation arguments.

Our first goal is to approximate solutions of \eqref{F=f} through first order perturbations. The method essentially relies on stability-like and compactness arguments, cf.\
\cite[Lemma 6.3]{CCKS}, \cite[Theorem 1]{DD19}, \cite[Lemma 2.6]{daST17}, \cite[Theorem 2.2]{dosPT16}.

The quadratic regime $m=2$ for $p>n$ was treated in \cite[Lemma 3.1]{Norn19}. Here, we need to work a bit more in that reasoning to obtain an approximation result applicable under the sublinear regime. We also cover the quadratic growth when $p=n$, since this is needed to prove Log-Lipschitz estimates.

\begin{lem}\label{AproxLem}
	Let $F$ satisfy (A1)-(A2) in $B_1$, $f \in L^p (B_1)$. Let $\psi \in C^\tau (\partial B_1)$ with $\|\psi\|_{C^\tau (\partial B_1)}\leq K_0$, and $\frak{B}\in \rn$.
	Then, for every $\varepsilon>0$, there exists $\delta\in (0,1)$, $\delta=\delta (\varepsilon,n,p,q,\varrho,\lambda,\Lambda,\tau,K_0)$, such that
\begin{align}\label{delta Ap Lemma}
\max\left\{\,\|{\bar{\beta}}_F(\cdot,0)\|_{L^p(B_1)}\, ,\;\|f\|_{L^p(B_1)}\, ,\; \|\mu\|_{L^q(B_1)}(|\frak{B}|^m+1) \, ,\;		\|b\|_{L^\varrho(B_1)}(|\frak{B}|+1) \,\right\}\leq \delta
\end{align}
imply that any two $L^p$-viscosity solutions $v$ and ${h}$ of
	\begin{align*}
	\left\{
	\begin{array}{rclcc}
	F(x,Dv+\frak{B},D^2v)&=& f& \mbox{in} & B_1 \\
	v &=& \psi \; & \mbox{on} & \partial B_1
	\end{array}
	\right., \;
\;	\left\{
	\begin{array}{rclcc}
	F(0,0,D^2 {h})&=& 0 & \mbox{in} & B_1 \\
	{h} &=&\psi & \textrm{on} & \partial B_1
	\end{array}
	\right.
	\end{align*}
	satisfy $\|v-{h}\|_{L^\infty (B_1)}\leq \varepsilon$.
\end{lem}

\begin{rmk}
	We have omitted the zero order term in \eqref{SC} for simplicity of notation, but a similar argument could also include linear perturbations in the form $L(x)=\frak{A}+\frak{B}\cdot x$, $\frak{A}\in \real$; see \cite{Norn19} for details.
\end{rmk}

\begin{proof}
Let us prove that for all $\varepsilon>0$, there exists $\delta\in (0,1)$ satisfying \eqref{delta Ap Lemma}, with $\delta\leq \widetilde{\delta}^{\frac{1}{2m^N}}$, where $\widetilde{\delta}$ and $N = N(n, m, p,\varrho, q)$ ($N=0$ if $p>n$) are the constants from Proposition \ref{wABPsuperlin} if $m>1$. 
	\vspace{0.03cm}	
	
By contradiction, we assume that the conclusion of Lemma \ref{AproxLem} is not true. Then, there exist $\varepsilon_0>0$ and a sequence of operators $F_k$ satisfying (A1) for $b_k\in L^\varrho_+(B_1), \mu_k \in L_{+}^q(B_1)$, $f_k\in L^p(B_1)$, $\frak{B}_k\in \rn$;
	also $\delta_k\in(0,1)$ satisfying $(2\delta_k^2)^{m^{N}}\leq \widetilde{\delta}_k$
	with $\widetilde{\delta}_k$, $N$ ($N=0$ if $p>n$) from Proposition \ref{wABPsuperlin} so that
	\begin{center}
		$\max\left\{\|{\bar{\beta}}_{F_k}(\cdot,0)\|_{L^p(B_1)}\, ,\;\|f_k\|_{L^p(B_1)}\, ,\; \|\mu_k\|_{L^q(B_1)}(|\frak{B}_k|^m+1) \, ,\;
		\|b_k\|_{L^{\varrho}(B_1)}(|\frak{B}_k|+1) \right\}\leq \delta_k$,
	\end{center}
	with $\delta_k\rightarrow 0$; and $v_k$, ${h}_k\in C(\overline{B}_1)$ $L^p$-viscosity solutions of
	\begin{align*}
	\left\{
	\begin{array}{rclc}
	F_k(x,Dv_k+\frak{B}_k,D^2v_k)&=&f_k & B_1 \\
	v_k &=& \psi_k  & \partial B_1
	\end{array}
	\right.
	, \;
	\left\{
	\begin{array}{rclc}
	F_k(0,0,D^2 {h}_k)&=& 0 &  B_1 \\
	h_k &=& \psi_k & \partial B_1
	\end{array}
	\right.
	\end{align*}
	where $\|\psi_k\|_{C^\tau (\partial B_1)}\leq K_0$, but
	$\|v_k-h_k\|_{L^\infty (B_1)} >\varepsilon_0.$
	\vspace{0.1cm}
	
Let us first show that there exists a universal positive constant $C_0$ such that
\begin{align}\label{estim Linfty LemApr}
\textrm{$\|v_k\|_{L^\infty (B_1)}\,, \;\|{h}_k\|_{L^\infty (B_1)}\leq C_0$
	\; for large $k$.}
\end{align}
Indeed,
	$
	\mathcal{M}^-(D^2{h}_k) \leq 0\leq \mathcal{M}^+(D^2{h}_k) \;\;\text{in}\; B_1
	$
yields
$\|{h}_k\|_{L^\infty (B_1)}\leq \|\psi_k\|_{L^\infty (\partial B_1)}\leq K_0.$ 	For $v_k$ we set
\begin{center}
$\mathcal{{L}}_k^\pm[w]=\mathcal{M}^\pm(D^2w)\pm b_k(x)|Dw|$, \;\;\quad $g_k(x)=b_k(x)|\frak{B}_k|+2\mu_k(x)|\frak{B}_k|^{m}$.
\end{center}	
Note that $v_k$ is an $L^p$-viscosity solution of the inequalities
	$$
	\mathcal{{L}}_k^+[v_k]+2{\mu}_k(x)|D v_k|^m  +g_k\geq f_k \geq \mathcal{{L}}_k^-[v_k]-2{\mu}_k(x) |Dv_k|^m-g_k\quad \text{in} \quad  B_1,
	$$
with
\vspace{-0.5cm}
\begin{center}
	$(\,\|f_k\|_{L^p(B_1)}+\|g_k\|_{L^p(B_1)} ) \,\|\mu_k\|_{L^q(B_1)}  \le 2\delta_k^2\leq  \widetilde{\delta}_k^{\,\frac{1}{m^{N}}}$  \;\;\;  for all $ k\in\n$ \quad if $m\in (1,2]$.
\end{center}
\vspace{0.1cm}
Then, by applying ABP in $B_1(0)$ with RHS equal to $|f_k|+g_k$ yields
$$
  \|v_k\|_{L^\infty(B_1)} \leq\| v_k\|_{L^\infty(\partial B_1)} +C_A^k \, W_k, \quad
W_k=\left\{
\begin{array}{lcl}
   \|{f}_k\|_{L^p(B_1)}+\|{g}_k\|_{L^p(B_1)}+ \|\mu_k\|_{L^q(B_1)}^{\frac{1}{1-m}} & \text{if} & m\in (0,1)\vspace{0.1cm} \\
  \|{f}_k\|_{L^p(B_1)}+\|{g}_k\|_{L^p(B_1)} & \text{if} & m\in [1,2].
\end{array}
\right.
$$
Since $\|\tilde{b}_k\|_{L^{\min(q,\varrho)}(B_1)}\leq C$ for large $k$, then $v_k$ also verifies \eqref{estim Linfty LemApr}.
Next, by global $C^{0, \beta}$ estimate in Proposition \ref{Cbeta global}, there exists $\beta\in (0,1)$ such that
	\begin{align}\label{estim Cbeta LemAp}
	\|v_k\|_{C^{0, \beta} (\overline{B}_1)}\,, \;\|{h}_k\|_{C^{0, \beta}(\overline{B}_1)}\leq C .
	\end{align}
	Here, $\beta$ and $C>0$ do not depend on $k$, since $\|\mu_k\|_{L^q(B_1)}$, $\|b_k\|_{L^\varrho(B_1)}$, $\|f_k\|_{L^p(B_1)}$, $\|{b}_k\|_{L^{\min(\varrho,q)}(B_1)}\le 1$ for large $ k$.	Thus, up to subsequences, one has $v_k\rightarrow v_\infty$ and ${h}_k\rightarrow {h}_\infty$ in $C(\overline{B}_1)$ as $k\rightarrow\infty$, for some $v_\infty, \; {h}_\infty \in C(\overline{B}_1)$ with $v_\infty = {h}_\infty = \psi_\infty$ on $\partial B_1$. Further, a subsequence of $F_k(0,0,X)$ uniformly converges on compact sets of $\mathbb{S}^n$ to some uniformly elliptic operator $F_\infty$.
	
	\vspace{0.05cm}
	
	Next, we claim that both $v_\infty$ and ${h}_\infty$ are $C$-viscosity solutions of	$F_\infty (D^2 u)=0 $ in $B_1$, $u=\psi_\infty  $ on $\partial B_1$.	
	For ${h}_\infty$ this follows by taking the uniform limits at the equation satisfied by $h_k$. For $v_\infty$ we infer that this follows by a stability argument. To see this, for any $\varphi\in C^2(B_1)$ we set
	$$
	\Phi_k(x)\defeq F_k(x, D\varphi+\frak{B}_k, D^2 \varphi)-f_k(x)-F_\infty (D^2 \varphi),
	$$
	so we verify that $\|\Phi_k\|_{L^p(B_1)}\to 0$ as $k\rightarrow\infty$.
	In fact, by splitting
	$\Phi_k= \Xi_1+ \Xi_2 +\Xi_3-f_k$, where
$$
\left\{
\begin{array}{rcl}
  \Xi_1 & \defeq & F_k(x, D\varphi+\frak{B}_k, D^2 \varphi)- F_k(x,0, D^2 \varphi) \\
  \Xi_2 & \defeq & F_k(x,0, D^2 \varphi)- F_k(0,0, D^2 \varphi) \\
  \Xi_3 & \defeq & F_k(0,0, D^2 \varphi) - F_\infty (D^2 \varphi),
\end{array}
\right.
$$
we have $\|\Xi_i\|_{L^p(B_1)}\to 0$, $i=1,2,3$, since
	\begin{center}
		$|\Xi_1|\leq 2\mu_k(x) (|D\varphi |^m+|\frak{B}_k|^m)+ b_k(x)( |D\varphi|+|\frak{B}_k|)$ and $\|\Xi_2\|_{L^p(B_1)} \leq \|{\bar{\beta}}_{F_k}(\cdot,0)\|_{L^p(B_1)}(\|D^2 \varphi \|_{\infty} +1)$.
	\end{center}
Hence, if $m\in [1,2]$ it is enough to apply Proposition \ref{stability}.
	\smallskip
	
	On the other side, if $m\in (0,1)$ we work directly with the definition of $C$-viscosity solution for $v_\infty$.
	To show the subsolution case, assume on the contrary that there exists $x_0\in B_1$, $\varepsilon ,\, \delta >0$ and a ball $B\subset B_1$ centered at $x_0$ such that $F_\infty(D^2\varphi)<-\varepsilon$ in $B$, with $v_\infty -\varphi$ having a local maximum at $x_0$. w.l.g.\ we can assume it to be zero and strict, i.e.\ $(v_\infty-\varphi)(x_0)=0$, $v_\infty-\varphi <0$ in $B\setminus \{x_0\}$, with $v_\infty-\varphi < -2\delta$ on $\partial B$ (see for instance Lemma \ref{lema strict} for equations without lower order terms).
Since $v_k\to v_\infty$ uniformly, we have $(v_k-\varphi)(x_0)>-\delta> (v_k-\varphi)_{|{\partial B}}$ for large $k$.
	
Now, by \cite[Theorem 3.1]{KoKo17}, let $\varphi_k$ be a strong solution of
	\begin{align*}
	\mathcal{L}^+_{m,k}\,[ \varphi_k]=-\Phi_k^+ \;\;\;\mathrm{in} \;B, \qquad \varphi_k=0 \;\;\;\mathrm{on} \; \partial B , \qquad \varphi_k\in W^{2,p}_{\mathrm{loc}} (B)\cap C(\overline{B}),
	\end{align*}
	where $\mathcal{L}^+_{m,k}$ is the operator in \eqref{def Lm} for $b_k,2\mu_k$ in place of $b,\mu$. Since $\|\mu_k\|_{L^p(B)}\to 0$ as $k\to +\infty$, then ABP in the sublinear regime (Proposition \ref{wABPsublin}) yields $\|\varphi_k \|_{L^\infty (B)}\to 0$. Further,
	\begin{align*}
	F_k(x,& D\varphi +D\varphi_k+\frak{B}_k, D^2\varphi +D^2\varphi_k))- f_k\\
	& \leq  \{F_k(x, D\varphi+\frak{B}_k, D^2 \varphi)- f_k\} +\mathcal{L}^+_{m,k}\,[\varphi_k] \leq \Phi_k-\Phi_k^+ + F_\infty(D^2\varphi)  \leq -\varepsilon\; \textrm{  a.e.\ in }B.
	\end{align*}
	But this produces a contradiction with the definition of $v_k$ as a subsolution, since for large $k$,
	\begin{align*}
	(v_k-(\varphi +\varphi_k))(x_0)&=(v_\infty-\varphi)(x_0)+\{ (v_k-v_\infty)(x_0)+\varphi_k(x_0) \}>-{\delta}\\ &>(v_\infty-\varphi)\mid_{\partial B}+\,\{ (v_k-v_\infty)+\varphi_k \}\mid_{\partial B} \;=(v_k-(\varphi +\varphi_k))\mid_{\partial B}
	\end{align*}
	and $v_k-(\varphi +\varphi_k)$ would have a maximum in the interior of $B$.

Analogously, one shows the supersolution case. Thus, $v_\infty$ satisfies $F_\infty(D^2v_\infty)=0$ in the $C$-viscosity sense in $B_1$.
	Finally, by Proposition \ref{prop F(X)} (ii) we have $v_\infty=h_\infty$ in $B_1$, which contradicts $\|v_\infty-{h}_\infty\|_{L^\infty (B_1)}\geq \varepsilon_0$, and so concludes the proof.
\end{proof}

Finally, we move to a scaling approach in order to put our equation in terms of the smallness assumptions prescribed by the approximations lemmas above.

\subsection{Normalization through scaling}\label{section scaling}

Notice that if $m\neq 2$, the normalization/reduction reasoning in Lemma \ref{AproxLem} follows through a standard scaling argument (see, \cite{Caf89, CafCab, Swiech}), which does not make use of local H\"{o}lder estimates (Theorem \ref{Calpha known}).
Nevertheless, the quadratic regime $m=2$ is similar to the one treated in \cite{Norn19}, in light of \cite{W2}.

Let $\Omega^\prime\subset\subset\Omega$ be a domain.
Consider $K_1$ and $\beta$ the pair given by the $C^\beta$ local superlinear estimate (Theorem \ref{Calpha known}) for $\Omega^\prime $, related to the initial $n,p,\lambda, \Lambda,\mu,\|b\|_{L^p(\Omega)}$, $\mathrm{dist}(\Omega^\prime,\partial\Omega)$ and $C_0$ such that
\begin{align}
\|u\|_{C^{0,\beta}(\overline{\Omega^\prime})} \leq
K_1 \, W_{u,\mu,m,\Omega\,}(\,\|f\|_{L^p(\Omega)}) , \textrm{ \;\; for $W$ as in }\eqref{def W geral}.
\end{align}

Also, let $K_2\geq 1$ and $\bar{\alpha}$ be the constants of Proposition \ref{prop F(X)}(i) associated to $n,\lambda,\Lambda$ in the ball $B_1(0)$.
w.l.g.\ we can suppose $K_1\geq \widetilde{K}_1$ and $\beta\leq \widetilde{\beta}$, where $\widetilde{K}_1,\widetilde{\beta}$ is the pair of $C^{0, \beta}$ local estimate in the ball $B_1$ with respect to an equation with given constants $n,m,p,q,\varrho,\lambda,\Lambda$ and bounds on the coefficients $\|\mu\|_{L^q(B_2)}\leq 2$ and $\|b\|_{L^\varrho(B_2)}\leq 1$, for solutions $u$ in $B_2$ with $\|u\|_{L^\infty(B_2)}\leq 1$.

\vspace{0.1cm}

Let us bear in mind some constants. First, let $\theta$ and $r_0 = r_0 (\theta)$ be such that \eqref{Htheta} holds for all $r\leq \min \{ r_0, \mathrm{dist} (x_0,\partial\Omega)\}$, for all $x_0 \in \Omega$.
We also fix constants $\delta>0$ and $\kappa=\kappa(m,p,q,\varrho,K_2)>0$.
These $\theta,\delta,\kappa$ will be chosen appropriately in the proof of the main theorems.

\vspace{0.1cm}

For ease of notation say $0\in \Omega^\prime$ and set $s_0=\min (r_0, \mathrm{dist}(0,\partial\Omega^\prime))$. Let $\sigma\leq \frac{s_0}{2}$ be such that
\begin{align} \label{sigma}
\textstyle{\sigma^{\upsilon} M \leq \frac{\delta}{\kappa}}
\end{align}
where $M= \max{ \{ 1, \|{b}\|_{L^{\varrho} (\Omega)} \} }$ and
\vspace{0.07cm}
\begin{center}
	$\upsilon=
	\min\left\{2-\iota-\frac{n}{p}\, ,\, 1-\frac{n}{\varrho} \right\}$ \;\; with \; $\iota=1$\, if $p>n$,\; while $\iota<2-\frac{n}{p}$\, if $p_0<p\le n$.
\end{center}
\vspace{0.07cm}
Up to making $\upsilon$ smaller and $M$ larger, we can assume in addition, in the superlinear regime $m>1$, that
\vspace{0.07cm}
$$
\left\{
\begin{array}{lcl}
  \upsilon\le 2-m-\frac{n}{q} \quad \text{and} \quad M\ge  W_0^{m-1}\|\mu\|_{L^q(\Omega)} & \text{if} & m\in (1,2-\frac{n}{q}) \\
  \upsilon\le \beta \quad \text{and}\quad M\ge (1+2^\beta K_1)W_0 \|\mu\|_{L^q(\Omega)} & \text{if} & m=2
\end{array}
\right.
$$
Here $\sigma$ depends on $s_0,\delta,n,p,q,\varrho ,m ,\|b\|_{L^{\varrho} (\Omega)},\|\mu\|_{L^{q} (\Omega)},K_1,K_2,\kappa,W_0$.
We highlight that $\sigma$ does not depend on $W_0$ if $m\le 1$.

\vspace{0.15cm}

In particular, $\overline{B}_{2\sigma} (0)\subset \Omega^\prime$ and we may define
\begin{align*}
N=N(0)=
\begin{cases}
\displaystyle\;\sigma^\iota W + \sup_{x\in B_2} |u(\sigma x) - u(0)|;
&\textrm{ if } m=2,\\
\; W &\textrm{ if } m\neq 2.
\end{cases}
\end{align*}
Then, by construction and $C^\beta$ local estimate (Theorem \ref{Calpha known}), if $m=2$ then $N$ is uniformly bounded by
\begin{align}\label{N}
\sigma^\iota W \leq N \leq (\sigma + 2^\beta K_1 \sigma ^\beta)W
\leq (1+2^\beta K_1) W_0\,\sigma ^\beta .
\end{align}

Similarly to \cite[Claim 3.2]{Norn19}, one sees that the function
\begin{align*}
\widetilde{u}(x)=
\begin{cases}
\; \{ u(\sigma x)-u(0) \}/N & \textrm{ if } m=2,\\
\;{u(\sigma x)}/{N}  & \textrm{ if } m\neq 2,
\end{cases}
\end{align*}
is an $L^p$-viscosity solution of $\widetilde{F}[\,\widetilde{u}\,]=\widetilde{f}(x)$ in $B_2$, where
$\widetilde{F}(x,\eta,X)= \frac{\sigma^2}{N} F\left( \sigma x,\frac{N}{\sigma}\eta, \frac{N}{\sigma^2} X \right)$, $\widetilde{f}(x)={\sigma^2} f(\sigma x)/N
$, $\widetilde{F}$ satisfying $(A1)$ for $\widetilde{b}(x)= \sigma b(\sigma x)$ and $\widetilde{\mu}(x)=N^{m-1}\sigma^{2-m}\mu(\sigma x)$.
Then, from \eqref{sigma},
\begin{center}
$\|\widetilde{u}\|_{L^\infty (B_2)} \leq 1$,	\qquad $\|\widetilde{b}\|_{L^{\varrho} (B_2)}= \sigma^{1-\frac{n}{\varrho}} \|{b}\|_{L^{\varrho} (B_{2\sigma})} \leq \frac{\delta}{\kappa} $,\medskip
\\
$\|\widetilde{f}\|_{L^p (B_2)}=\frac{\sigma^{2-\frac{n}{p}}}{N}\|{f}\|_{L^p (B_{2\sigma})}\leq \sigma^{2-\iota-\frac{n}{p}} \frac{\|{f}\|_{L^p (\Omega)}}{W}\le \delta$,
\end{center}
\vspace{-0.2cm}
\begin{align*}
\|\widetilde{\mu}\|_{L^q(B_2)}=N^{m-1}\sigma^{2-m-\frac{n}{q}}\|\mu\|_{L^q(B_{2\sigma})}\vspace{0.3cm}\le
\begin{cases}
\;(1+2^\beta K_1)W_0 \|\mu\|_{L^q(\Omega)} \, \sigma^{\beta}&\textrm{ if } m=2, \; q=\infty,	\vspace{0.2cm} \\
\; W_0^{m-1}\|\mu\|_{L^q(\Omega)} \, \sigma^{2-m-\frac{n}{q}}
&\textrm{ if }  m\in [1,2-\frac{n}{q}),	\vspace{0.2cm}\\
\;\frac{\|\mu\|_{L^q(\Omega)}}{W^{1-m}}\, \sigma^{2-m-\frac{n}{q}}
&\textrm{ if }  m\in (0,1),
\end{cases}
\end{align*}
thus $\|\widetilde{\mu}\|_{L^q(B_2)}\le \frac{\delta}{\kappa}$ for the choice of $\varrho, q$ given in (A2).
Moreover,\vspace{-0.1cm}
\begin{center}
	$\|{\bar{\beta}}_{\widetilde{F}}(\cdot,0)\|_{L^p (B_1)}=\left( \intav{B_\sigma(0)} \beta_F(y,0)^p \mathrm{d}y \right)^{{1}/{p}} \leq \delta $,\;\; by choosing $\theta=\delta$.
\end{center}
Therefore, under (A2),\vspace{-0.1cm}
\begin{align}\label{eq scaling delta}
\textstyle \|\widetilde{u}\|_{L^\infty (B_2)} \leq 1,
\;\;
\|{\bar{\beta}}_{\widetilde{F}}(\cdot,0)\|_{L^p (B_1)} \leq \delta,
\;\;
\|\widetilde{\mu}\|_{L^q(B_2)}\le \frac{\delta}{\kappa} ,
\;\;
\|\widetilde{b}\|_{L^{\varrho} (B_2)}\leq \frac{\delta}{\kappa}, \quad
\|\widetilde{f}\|_{L^p (B_2)}\leq\delta .
\end{align}
In particular, $\widetilde{F},\, \widetilde{u},\, \widetilde{\mu},\, \widetilde{b}$
(with $ \frak{B},\frak{C}=0$) satisfy Lemma \ref{AproxLem} hypotheses. Further, if for instance we show that $\|\widetilde{u}\|_{C^{1,\alpha}(\overline{B}_1)} \leq C$, then we get
$
\| {u} \|_{C^{1,\alpha}(\overline{B}_\sigma)}\leq C W ,
$
where $C$ depends on $\sigma$. Finally, we apply a covering argument for $\Omega^\prime$. For the other norms the reasoning is analogous.

We just write $F,u,M, \mu,b$ as the shorthand notation for $\widetilde{F},\widetilde{u}, \widetilde{\mu},\widetilde{b}$. In other words, in what follows we may assume w.l.g.\ that the original data $F,u,M, \mu,b$ satisfy the smallness regime \eqref{eq scaling delta}.

\section{General regularity theory}\label{proof main th}

The strategy for proving optimal regularity estimates is based on a refined compactness method from the geometric tangential analysis in \cite{Caf89, CCKS, Swiech, Tei14, W2}. It relies on control  of oscillation decay, which permits us to access the corresponding regularity theory available for pure second order operators with frozen coefficients.

\subsection{H\"{o}lder regularity: Proof of Theorem \ref{Calpha regularity estimates geral}}\label{section optimal Calpha}

We start proving a fundamental lemma which provides the first step in the iteration process. The core idea is to approximate solutions of \eqref{F=f} by constants, as in \cite{daST17} and \cite{Tei14}.

\begin{lem}\label{lemma3.1} Let $u$ be an $L^{p}$-viscosity solution for \eqref{F=f} in $B_2$ with $\|u\|_{L^{\infty}(B_2)}\leq 1$. Given $0<\alpha_0<1$, there exist $\delta >0$ and $0<\gamma\ll 1$, depending only on $n,  \lambda, \Lambda, \alpha_0$, so that if \eqref{delta Ap Lemma} and (A2) are in force for $ p_0 < p < n $, then\vspace{-0.2cm}	\begin{center}
		$\displaystyle \sup_{B_{\gamma}} |u - \frak{a}| \leq \gamma^{\alpha_0}$ for some $\frak{a} \in \mathbb{R}$, with $|\frak{a}| \leq C(n, \lambda, \Lambda)$.
	\end{center}
\end{lem}

\begin{proof} We fix  $\varepsilon=\varepsilon (n, \lambda, \Lambda, \delta) > 0$ to be chosen, and $0<r_0\leq 1$ as the constant which appears in the hypotheses of Theorem \ref{Calpha regularity estimates geral}. Let $h$ be a solution of $F(0,0, D^2 h)=0$ in $B_2$ satisfying
$
\sup_{B_{1}}|u-h|\leq \varepsilon,
$
	given by Lemma \ref{AproxLem}.
	This $\varepsilon$ produces a $\delta= \delta(\varepsilon)$ in \eqref{delta Ap Lemma}, which we assume $\delta\leq \varepsilon$.
	By the regularity theory available for $h$,  see for instance  \cite{Caf89}, \cite{CafCab} and \cite{Tru88}, one can estimate
	\begin{equation}\label{eq3.3}
|h(x)-h(0)|\leq C |x| \,\,\,\;\;\mbox{for all} \,\,\, |x| < {1},\qquad 	|h(0)| \leq c_0, \quad C=C(n, \lambda, \Lambda) .
	\end{equation}
For $\frak{a} = h(0)$, and the constant $c_0$ given in \eqref{eq3.3}, it follows then
	\begin{center}
		$\displaystyle \sup_{B_{\gamma}} |u-\frak{a}| \leq  \sup_{B_{\gamma}} |u-h|+  \sup_{B_{\gamma}} |h-h(0)| \le \delta + c_0\gamma .
	$
	\end{center}
Now, it is enough setting
$\gamma =\min\{r_0, ({1}/{2c_0})^{\frac{1}{1- \alpha_0}}\}$ and $ 2\varepsilon = \gamma^{\alpha_0}$ to conclude the proof.
\end{proof}

\begin{proof}[{\bf Proof of Theorem \ref{Calpha regularity estimates geral}}] Set $\alpha_0 = 2-\frac{n}{p}\in (0,1)$. Via the scaling process in Section \ref{section scaling}, we may assume \eqref{eq scaling delta} with $\kappa=1$, where $\delta$ is the universal constant from Lemma \ref{lemma3.1}. Then, take $\theta  = \delta$ and assume $r_0=1$. Our goal is to iterate the conclusion of Lemma \ref{lemma3.1}. For a fixed $y \in B_{1}$ we claim that there exists a convergent sequence of real numbers $\{\frak{a}_k\}_{k \geq 1}$, such that
\begin{equation}\label{eq4.2}
  \sup_{B_{r_k}(y)} |u(x) - \frak{a}_k| \leq \rho^{k\alpha_0}, \quad \textrm{ for } \; r_k \defeq \gamma^k,
\end{equation}
where the radius $0<\gamma \ll 1$ is given by Lemma \ref{lemma3.1}, for $\alpha_0$ as above.

Let us perform an induction argument. Notice that Lemma \ref{lemma3.1} gives the first step of induction, for $k = 1$.  Now suppose the $k^{th}$ step is verified. We define
\begin{center}
$	v_k(x) =  \frac{u(y + r_k x)-\frak{a}_k}{r_k^{\alpha_0}}$ \;\; and \;	$ F_k(x, \xi, X) = r_k^{2- \alpha_0} F( \,{  y+ r_k x \, , \, r_k^{-1+ \alpha_0} {\xi}} \,, \, r_k^{-2+ \alpha_0}X\, ).$
\end{center}
Then, $F_k$ fulfills $(A1)$ for $b_{F_k}(x)=r_k\, b(y+r_k x)$, and $\mu_{F_k}(x)=r_k^{1+(1-\alpha_0)(1-m)}\mu(y+r_kx)$. Moreover  by the induction hypothesis, $v_k$ satisfies $|v_k|\leq 1$ and
\begin{center}
		$	F_k(x, Dv_k, D^2v_k) = r_k^{2- \alpha_0} f(y + r_k x)  =: f_k(x)$ \;\; in $B_2$,
\end{center}
in the $L^p$-viscosity sense. One easily computes
$$
\|f_k\|_{L^{p}(B_1)} =  r_k^{2-\frac{n}{p}-\alpha_0}\|f\|_{L^{p}(B_{r_k})} \leq \|f\|_{L^{p}(B_1)} \leq \delta,\quad
\|b_{F_k}\|_{L^{\varrho}(B_1)} =  r_k^{1-\frac{n}{\varrho} } \|b\|_{L^{\varrho}(B_{r_k})} \leq \delta,
$$
$$
 \|\mu_{F_k}\|_{L^{q}(B_1)} =  r_k^{1-\frac{n}{q}+(1-\alpha_0)(1-m)} \|\mu\|_{L^{q}(B_{r_k})} \leq \delta, \;\;
\| \bar{\beta}_{F_k}\|_{L^p(B_1)} \leq \max\{1, r_k^{{2-\alpha_0}}\}\, \|\beta_F\|_{L^p(B_{r_k}(y))}  \leq  \delta,\smallskip
$$
since $q, \varrho>n $, $\alpha_0 = 2-\frac{n}{p}$, and \begin{center}
	$A \defeq 1-\frac{n}{q}+(1-\alpha_0)(1-m)>0$, \quad for all $m \in (0, 2]$.
\end{center}
In fact, $ A> 2-\frac{n}{q}-m >0$ if $m\in (1,2)$; while $ A= 2-\frac{n}{p} >0$  by using $\frac{n}{q}=0$ if $m=2$.\smallskip

Then Lemma \ref{lemma3.1} provides the existence of real number $\overline{\frak{a}}_k$, with $|\overline{\frak{a}}_k| \leq C$, such that
$
\sup_{B_{\rho}} |v_k - \overline{\frak{a}}_k| \leq \rho^{\alpha_0}.
$
Thus, we set $\frak{a}_{k+1} \defeq \frak{a}_k + r_k^{\alpha_0}\,\overline{\frak{a}}_k$\, and rescale back to the unit ball, by obtaining the $(k+1)^{th}$ step of the induction process. In addition, we have that
$ |\frak{a}_{k+1} - \frak{a}_k|\leq C r_k^{\alpha_0},
$
so the sequence $\{\frak{a}_k\}_{k \geq 1}$ converges, with $\frak{a}_k \rightarrow u(y)$. Also,
\begin{equation}\label{eq4.3}
 |u(y) - \frak{a}_k|\leq \frac{C r_k^{\alpha_0}}{1- r_k^{\alpha_0}} .
\end{equation}
Finally, for $0 < r < \gamma$, let $k$ such that $x \in B_{r_k}(y)\setminus B_{r_{k+1}}(y)$. It follows from \eqref{eq4.2}, \eqref{eq4.3} that
$$
\displaystyle	\sup_{B_r(y)}  \frac{|u(x)-u(y)|}{|x-y|^{\alpha_0}}  \leq  \sup_{B_r(y)} \frac{|u(x)-\frak{a}_k| + |u(y) - \frak{a}_k|}{|x-y|^{\alpha_0}} \leq \left(1 + \frac{C}{1- \gamma^{\alpha_0}}\right)\sup_{B_r(y)} \frac{ {r_k}^{\alpha_0}}{|x-y|^{\alpha_0}}\leq  \frac{C_{\alpha_0}}{\gamma^{\alpha_0}}.
$$
	The last bound, combined with a standard covering argument provide the desired estimate.
\end{proof}

\smallskip

\subsection{Borderline and Gradient regularity: Proof of Theorems \ref{Log-Lip estimates} and \ref{C1,alpha regularity estimates geral}}\label{proof gradient regul}

In this section, we provide a unified proof to treat both cases $p=n$ and $p>n$.

\begin{proof}
\textbf{Step 1: Choice of constants}
\medskip

Assume $p\geq n$ and consider a fixed constant $\alpha$ such that
\begin{align}\label{choice alpha C1,alpha}
\textstyle 0<\alpha\le \min\left\{ 1-\frac{n}{\varrho}\,,\, 1-\frac{n}{q}\, , \, 1-\frac{n}{p}\right\}\;\,\,\, \text{and} \,\,\,\alpha<\bar{\alpha}\;\, \textrm{ if }\; p>n,\qquad
\text{and} \quad \alpha=0 \; \textrm{ if }\;p=n.
\end{align}
Then, we pick up $\gamma=\gamma(\alpha,\bar{\alpha},K_2)\in (0,{1}/{4}]$, in addition to
\begin{align} \label{gamma}
2^{2+\bar{\alpha}}K_2\, \gamma^{\bar{\alpha}}\leq \gamma^{\alpha}\quad\textrm{ if }\; p>n.
\end{align}
Next, we define
\vspace{-0.3cm}
\begin{align} \label{epsilon}
\varepsilon=\varepsilon (\gamma)=
K_2\, (2\gamma)^{1+\bar{\alpha}} .
\end{align}
Such an $\varepsilon$ provides a $\delta=\delta(\varepsilon)\in (0,1)$, the constant of Lemma \ref{AproxLem}, which w.l.g.\ we assume to satisfy
\begin{align} \label{delta}
\kappa\,\delta \leq \gamma^{\alpha}.
\end{align}
where the constant $\kappa$, given in Section \ref{section scaling}, is defined by
\begin{align}
\textstyle
\kappa =2\,(K_2 +K_2^m+K^2),\;\; \textrm { where\; $K={K_2}\,\frac{\gamma^{-\alpha}}{1-\gamma^\alpha} +{K_2}\,\frac{\gamma^{-1-\alpha}}{1-\gamma^{1+\alpha} }\ge K_2\ge 1$.
}
\end{align}

Finally, notice that $\sigma$, chosen in Section \ref{section scaling}, then depends also on $\alpha,\bar{\alpha},\beta$, since $\delta=\delta(\varepsilon)$.

\bigskip
\textbf{Step 2: Iterations scheme}
\medskip

Now, we can proceed with Caffarelli's iterations as in \cite{Caf89, CafCab, Swiech}, which consists of finding a sequence of linear functions $\frak{l}_k (x)=\frak{a}_k +\frak{b}_k \cdot x$ such that
\vspace{0.2cm}

$(i)_k \;\;\| u-\frak{l}_k \|_{L^\infty (B_{r_k})} \leq r_k^{1+\alpha}$,
\vspace{0.2cm}

$(ii)_k \;\;|\frak{a}_k-\frak{a}_{k-1}| \leq K_2 \,r_{k-1}^{1+\alpha}\;,\;\; |\frak{b}_k - \frak{b}_{k-1}|\leq K_2 \,r_{k-1}^\alpha$,
\vspace{0.2cm}

$(iii)_k \;\;|(u-\frak{l}_k)(r_k x)-(u-\frak{l}_k)(r_k y)|\leq (1+3K_1) \, r_k^{1+\alpha} |x-y|^\beta \;\; \textrm{for all } x,y\in B_1 $, \vspace{0.2cm}\\
for $r_k=\gamma^k$ for some $\gamma\in(0,1)$, for all $ k\geq 0$, with the convention that $\frak{l}_{-1}\equiv 0$.
\vspace{0.11cm}

In the case $p>n$, this implies the existence of $\frak{a}\in \real$, $\frak{b}\in \rn$ such that $\frak{a}_k\to \frak{a}$ and $\frak{b}_k\to \frak{b}$, since $|\frak{b}_k - \frak{b}| \leq  K_2 \frac{\gamma^{\alpha k}}{1-\gamma^\alpha}$ and
$|\frak{a}_k - \frak{a}|\leq K_2\frac{\gamma^{k(1+\alpha)}}{1-\gamma^{1+\alpha}}$.
Notice that $|\frak{b}_k| \leq\frac{K_2}{{\gamma^\alpha}(1-\gamma^\alpha)}\leq K$ and $|\frak{l}_k(x)|\leq   K$ for all $x\in B_1$.
Thus, for each $x\in B_1$ there exists $k\geq 0$ such that $r_{k+1} < |x|\le r_k$ and
\begin{center}
$| u(x) - u(0) - D u(0)\cdot x | \leq C|x|^{1+\alpha}$, \; and $|D u(0)| \leq C$,
\end{center}
since by definition of a differentiable function one has $\frak{a}=u(0)$, $\frak{b}=D u(0)$.

\vspace{0.07cm}

Instead, if $p=n$ then $\alpha = 0$ corresponds to the existence of $\frak{a}\in \real$ such that $\frak{a}_k\to \frak{a}$, since
$|\frak{a}_k - \frak{a}|\leq K_2\frac{r_k}{1-\gamma}$. Also,
$|\frak{b}_k| \leq  \sum_{j=1}^{k}|\frak{b}_j - \frak{b}_{j-1}| \leq kK_2 $\, for $\frak{b}_0=0$.
Then, for $x \in B_1\setminus \{0\}$ there exists $k\geq 0$ such that $r_{k+1} < |x|\le r_k$ and
\begin{center}
$|u(x)-u(0)| \leq |u-\frak{l}_k| + |u(0) - \frak{a}_k| + |\frak{b}_k|r_k \leq K_2\, r_k\left( k+1 + \frac{1}{1-\gamma}\right) \leq C|x|\,|\ln|x|| $.
\end{center}

Hence, for any $x_0\in\Omega$, we set $s_0=\min \{r_0,\mathrm{dist}(x_0,\partial\Omega^\prime)\}$ and $N=N (x_0)$ to get for the original $u$,
\begin{center}
$| u(x) - u(x_0) - D u(x_0)\cdot (x-x_0) | \leq CW|x-x_0|^{1+\alpha}$,\,\;\; $|D u(x_0)| \leq CW$\quad if \,$p>n$;
\vspace{0.2cm}

$| u(x)-u(x_0)| \leq CW|x-x_0|\,|\ln|x-x_0||$ \quad if \,$p=n$,
\end{center}
which implies the desired regularity estimates in the ball ${B}_{\sigma}(x_0)$.
Thus, for the complete local estimate, we just take finitely many such points in order to cover $\Omega^\prime$, see also \cite{Norn19} for further details.
\smallskip

Let us prove $(i)_k-(iii)_k$ by induction on $k$.
For $k=0$ we set $\frak{a}_0 = \frak{b}_0 = 0$. Recall that $\beta$ and $K_1$ are the constants from the $C^\beta$ local estimate (Theorem \ref{Calpha known}) in $B_1$ such that
$
\|u\|_{C^\beta(\overline{B}_1)}\leq \widetilde{K}_1 (1+\delta+1)\leq 3K_1,
$
which implies $(iii)_0$.
Obviously $(i)_0$ and $(ii)_0$ are satisfied too.\vspace{0.1cm}
As the induction step, we suppose the items $(i)_k-(iii)_k$ valid in order to construct $\frak{a}_{k+1}$ and $\frak{b}_{k+1}$ for which $(i)_{k+1}-(iii)_{k+1}$ hold. Define
\begin{align}\label{vk do lk}
v(x)=v_k({x}) = \frac{(u-\frak{l}_k)(r_k {x})}{r_k^{1+\alpha}}\, , \;\textrm{ for all }\, {x}\in B_2 .
\end{align}
Note that $(i)_k$ says precisely that $ |v({x})| \leq 1$ for all $x\in B_1.$ Further, from this and $(iii)_k$ we get
$$
\|v\|_{C^\beta (\overline{B}_1)}=\|v\|_{L^\infty ({B}_1)}+ \sup_{\substack{ x,y \in B_1 \\ x \neq y }} \frac{|v(x)-v(y)|}{|x-y|^\beta}  \leq 2+3K_1=:K_0.
$$

Similarly to \cite{Norn19},
$v$ is an $L^p$-viscosity solution of
\begin{center}
$F_k(x, Dv+\frak{B}_k, D^2 v) = f_k(x)$\;\; in $B_1$,
\end{center}
where $f_k(x)=r_k^{1-\alpha} f(r_k x)$, for $\frak{B}_k=r_k^{-\alpha}\frak{b}_k$, and
\begin{center}
$F_k(x,\eta,X)=r_k^{1-\alpha} F(r_kx,r_k^\alpha \eta,r_k^{\alpha-1} X)$,
\end{center}
which satisfies $(A1)$ for $b_{F_k}(x)=r_k b(r_kx)$, and $\mu_{F_k}(x)=r_k^{1+(m-1)\alpha}\mu(r_kx)$.
Thus, using the choice of $\alpha$, together with the estimates of the tilde-like coefficients, and $r_k\leq 1$, we have
\begin{center}
$\|\bar{\beta}_{F_k} (\cdot,0)\|_{L^p(B_1)}\leq  \delta$; \quad $\|f_k\|_{L^p(B_1)}\leq r_k^{1-\frac{n}{p}-\alpha} \|f\|_{L^p(B_{r_k})}\leq {\delta}$,
\end{center}
by the choice of $\alpha$ in Section \ref{section scaling}, for $p\geq n$.
Also, when $p>n$ one has
\begin{center}
$\|b_{F_k}\|_{L^\varrho (B_1)}(|\frak{B}_k|+1)=r_k^{1-\alpha-\frac{n}{\varrho}}( |\frak{b}_k|+1)\|b\|_{L^\varrho(B_{r_k})}\leq 2K\|b\|_{L^\varrho(B_1)}\leq \delta$,
\vspace{0.2cm}

$\|\mu_{F_k}\|_{L^q(B_1)} (|\frak{B}_k|^m+1)=r_k^{1-\alpha-\frac{n}{q}}( |\frak{b}_k|^m +1)\|\mu\|_{L^q(B_{r_k})} \leq 2K^2\|\mu\|_{L^q(B_1)} \leq \delta$.
\end{center}
On the other hand, if $p= n$ then $\alpha=0$ and	
$$\|b_{F_k}\|_{L^\varrho (B_1)}(|\frak{B}_k|+1)=r_k^{1-\frac{n}{\varrho}}( |\frak{b}_k|+1)\|b\|_{L^\varrho(B_{r_k})}\leq 2Kk\,r_k^{1-\frac{n}{\varrho}}\|b\|_{L^\varrho(B_1)}\leq \delta;$$
\vspace{-0.5cm}
\begin{align*}
\|\mu_{F_k}\|_{L^q(B_1)} (|\frak{B}_k|^m +1)  =  r_k^{1-\frac{n}{q}}( |\frak{b}_k|^m + 1)\|\mu\|_{L^q(B_{r_k})}
 \leq  2K^2 k^m r_k^{1-\frac{n}{q}}\|\mu\|_{L^q(B_1)} \leq \delta .
\end{align*}
Thus, $F_k,v,\mu_{F_k},b_{F_k}, \frak{B}_k$  also satisfy Lemma \ref{AproxLem} hypotheses.

\medskip

Now, we take $h=h_k\in C(\overline{B}_1)$ the $C$-viscosity solution of
$F_k (0,0,D^2 h) =0$ in $B_1$, with
$h = v$  on $ \partial B_1$.
By ABP we have $\|h\|_{L^\infty (B_1)}\leq \|h\|_{L^\infty (\partial B_1)}\leq 1$ and by the $C^{1,\bar{\alpha}}$ local estimate (Proposition \ref{prop F(X)} (i)), $\|h\|_{C^{1,\bar{\alpha}} (\overline{B}_{1/2})}\leq K_2\,\|h\|_{L^\infty (B_1)} \leq K_2$.
Hence, by Lemma \ref{AproxLem} applied to $F_k,v,\mu_{F_k},b_{F_k}$, $\psi=v\mid_{\partial B_1},\tau=\beta$, $K_0 $ and $h$ we get, for $\varepsilon$ given in \eqref{epsilon}, that
$
\|v-h\|_{L^\infty (B_1)}\leq \varepsilon.
$

\smallskip

Define $\overline{\frak{l}}(x)=\overline{\frak{l}}_k(x)=h(0)+Dh(0)\cdot x$ in $B_1$, then,
\begin{align} \label{wlinfty local}
\|v-\overline{\frak{l}}\,\|_{L^\infty (B_{2\gamma})}\leq \gamma^{1+\alpha}.
\end{align}
In fact, by the choice of $\gamma\leq \frac{1}{4}$ in \eqref{gamma}, we have for all $x\in B_{2\gamma}(0)$ that
\begin{align*}
|v(x)-\overline{\frak{l}}\,(x)|\leq  |v(x)-h(x)| + |h(x) - h(0) - D h(0)\cdot x|
\leq 2 K_2 \,(2\gamma)^{1+\bar{\alpha}}\leq \gamma^{1+\alpha}.
\end{align*}

However, inequality \eqref{wlinfty local} and the definition of $v$ imply
\begin{align*}
|u(r_k {x}) -\frak{l}_k(r_kx) - r_k^{1+\alpha} h(0) - r_k^{1+\alpha} D h(0) \cdot  x |\leq r_k^{1+\alpha} \gamma^{1+\alpha} = r_{k+1}^{1+\alpha}\, \;\; \textrm{ for all } x \in B_{2\gamma}\,,
\end{align*}
which is equivalent to
\begin{align*}
|u(y) -\frak{l}_{k+1}(y)|\leq r_k^{1+\alpha} \gamma^{1+\alpha} = r_{k+1}^{1+\alpha}\, \;\; \textrm{ for all } y=r_kx \in B_{2\gamma r_k}= B_{2r_{k+1}}\, ,
\end{align*}
where $\frak{l}_{k+1}(y)\defeq \frak{l}_k(y)+r_k^{1+\alpha} h(0) + r_k^{\alpha} D h(0) \cdot  y \,$. Then, we define
$$
\frak{a}_{k+1} = \frak{a}_k + h(0)\, r_k^{1+\alpha} , \quad
\frak{b}_{k+1} = \frak{b}_k + D h(0)\, r_k^{\alpha},
$$
from where we obtain $(i)_{k+1}$. Further, $|\frak{a}_{k+1} - \frak{a}_k|\leq K_2 \, r_k^{1+\alpha}$, $|\frak{b}_{k+1} - \frak{b}_k|\leq K_2 \, r_k^{\alpha}$, which is $(ii)_{k+1}$. To finish we observe that, in order to prove $(iii)_{k+1}$, it is enough to show
\begin{align}\label{iii}
\|v-\overline{\frak{l}}\,\|_{C^\beta (\overline{B}_\gamma)}\leq (1+3K_1) \,\gamma^{1+\alpha-\beta}.
\end{align}
Indeed, if $x,y\in B_1$ and \eqref{iii} is true, then
\begin{align*}
|(u-\frak{l}_{k+1})(r_{k+1}x)-(u-\frak{l}_{k+1})(r_{k+1}y)| \leq (1+2K_1) \, r_{k+1}^{1+\alpha} \,|x-y|^\beta.
\end{align*}

We obtain \eqref{iii} by applying Theorem \ref{Calpha known} to $w=v-\overline{\frak{l}}$ which is an $L^p$-viscosity solution of
\begin{center}
$\mathcal{L}_{k,m}^- [w]\leq g_k(x) ,\quad\mathcal{L}_{k,m}^+ \geq - g_k(x)$\;\; in $B_2$
\end{center}
where $\mathcal{L}_{k,m}^\pm$ is the operator given in \eqref{def Lm} with $b,\mu$ replaced by ${b}_{F_k}$, $2\mu_{F_k}$, $g_k(x)=|f_k(x)|+b_{F_k}(x)|Dh(0)|+2\mu_{F_k}(x)|Dh(0)|$, and using that $|\overline{\frak{l}}(x)|\leq |h(0)|+|Dh(0)|\,|x|\leq \|h\|_{C^{1,\bar{\alpha}} (\overline{B}_{1/2})}\leq K_2$ for all $ x \in B_1$,
we have
\begin{align*}
\|\,g_k\|_{L^p(B_1)} \leq \|f_k\|_{L^p(B_1)} +\|b_{F_k}\|_{L^p(B_1)} K_2+2\|\mu_k\|_{L^p(B_1)}K_2^ m\leq (1+K_2+2K_2^m)\,\delta \leq \gamma^{\alpha},
\end{align*}
by H\"{o}lder inequality and definition of $\delta$ in \eqref{delta}.
Thus, using the estimate above and \eqref{wlinfty local} in the $C^\beta$ local estimate (Theorem \ref{Calpha known}), properly scaled to the ball of radius $\gamma$, yields
\begin{align*}
[w]_{\beta,\overline{B}_\gamma} &\leq
\begin{cases}
\gamma^{-\beta} \widetilde{K}_1 \,( \, \|w\|_{L^\infty ({B_{2\gamma}})} +\gamma^{2-\frac{n}{p}} \|g_k\|_{L^p ({B_{2\gamma}})} \,) \quad\textrm{ if } m\ge 1,\vspace{0.15cm}\\
 \gamma^{-\beta} \widetilde{K}_1 \,( \, \|w\|_{L^\infty ({B_{2\gamma}})} +\gamma^{2-\frac{n}{p}} \|g_k\|_{L^p ({B_{2\gamma}})} + \gamma^{\frac{2-{n}/{q}}{1-m}} \|\mu\|_{L^q(B_{2\gamma})}^{\frac{1}{1-m}}  \,) \quad\textrm{ if } 0<m<1,
\end{cases} \\
\vspace{0.2cm}
&\le \begin{cases}
\gamma^{-\beta}  {K}_1 \,( \,\gamma^{1+\alpha} + \gamma^{2-\frac{n}{p}}\gamma^\alpha   \,)\quad\textrm{ if } m\ge 1,\vspace{0.15cm}\\
\gamma^{-\beta}  {K}_1 \,( \,\gamma^{1+\alpha} + \gamma^{2-\frac{n}{p}}\gamma^\alpha + \gamma^{\frac{2-{n}/{q}}{1-m}}  \,)\quad\textrm{ if } 0<m<1, \vspace{0.15cm}
\end{cases}\\
&\le \gamma^{-\beta}  {K}_1 \,( \,\gamma^{1+\alpha} + \gamma^{2-\frac{n}{p}}\gamma^\alpha + \gamma^{1+\alpha})\le \;3  K_1 \,\gamma^{1+\alpha-\beta} ,
\end{align*}
by using $\alpha\leq 1-\frac{n}{q}$\, when $0<m<1$.
Hence,
$
\|w\|_{C^\beta (\overline{B}_\gamma)}=\|w\|_{L^\infty ({B_\gamma})} +
[w]_{\beta,\overline{B}_\gamma}
\leq (1+3K_1)\,\gamma^{1+\alpha-\beta}
$.
\end{proof}

\section{Further regularity estimates}\label{section further regularity}

Now we exploit finer regularity estimates coming from the methods and results obtained in the preceding sections. Throughout this section we assume $F$ has the form \eqref{MC_intro}, for bounded coefficients $b,\mu$.

\subsection{Hessian type regularity: Proof of Theorems \ref{C1,logLip theorem} and \ref{Schauder theorem}}\label{proof Schauder}

Note that an $L^n$-viscosity solution is in particular an $L^p$-viscosity solution for all $p>n$. These two notions are in fact equivalent by Proposition \ref{prop Lp iff Lupsilon} when  $f,b,\mu$ are bounded.
In particular, we already know that a bounded solution $u$ considered either in Theorem \ref{C1,logLip theorem} or Theorem \ref{Schauder theorem}, for $f\in L^\infty(\Omega)$, enjoys $C^{1,\alpha}$ regularity estimates. For instance, in the local case one already has estimate \eqref{estim C1,alpha local}.

\begin{lem}\label{RHS lemma}
Assume hypothesis ($\tilde{A}$3), and let $u$ be an $L^p$-viscosity solution of \eqref{F=f}, where $F$ has the form \eqref{MC_intro}, $p>n$, and $b, \mu \in L^\infty(\Omega)$. Then $u$ is also an $L^p$-viscosity solution of
\begin{center}
	$G(x,D^2 u)=g(x)$\; in $\Omega$,
\;\; where $|g|\le |f|+b|Du|+\mu |Du|^m \in L^p_{\mathrm{loc}}(\Omega)$.
\end{center}
\end{lem}

\begin{proof} When $m\ge 1$ this is the content of \cite[Claim 4.2]{Norn19}. Although the arguments are similar, here we give a unified proof comprising also the case $m\in (0,1)$ for the sake of completeness.
	
First we observe that by our $C^{1,\alpha}$ local regularity estimates over the set $\Omega^{\prime\prime}\subset \subset\Omega$ we get $u\in C^{1,\alpha}(\overline{\Omega^\prime})$ for any $\Omega^\prime\subset\Omega^{\prime\prime}$, and $\|u\|_{C^{1,\alpha}(\overline{\Omega^\prime})}\le CW(\|f\|_{L^p(\Omega^{\prime\prime})})$, where $W$ is given in \eqref{def W geral}. 

Now, to prove the lemma we suppose by contradiction that $u$ is not an $L^p$-viscosity subsolution of $G(x,D^2 u)\ge g$ (the supersolution case is analogous). Then, there exist $\phi\in W^{2,p}(B_r)$,  $x_0\in \Omega$, and $r,\epsilon\in (0,1)$ such that $u-\phi$ attains a local maximum at $x_0$, but $G(x, D^2 \phi) \le g(x)-\epsilon$ a.e.\ in $B_r(x_0)$. On the other hand, the definition of viscosity subsolution of \eqref{F=f}, where $F$ is as in \eqref{MC_intro}, yields
\begin{center}
$G(x,D^2\phi)+b_0|D\phi|+\mu_0|D\phi|^m \ge f(x)-\epsilon/2$ \; a.e.\ in $B_r(x_0)$, \; where  $b(x)\le b_0$, $\mu(x)\le \mu_0$.
\end{center}
Thus one gets for a.e.\ in $B_r(x_0)$, and for all $s\le r$,
	\begin{center}
$-\epsilon \ge G(x,D^2 \phi)-g(x)\ge -b_0|Dh-D\phi|-\mu_0\, \sigma_0(Dh,D\phi)-\epsilon/2 $,
	\end{center} where $\sigma_0(\xi,\eta)=|\xi-\eta|^{m}$ if $m\leq 1$, and $\sigma_0(\xi,\eta)=|\xi-\eta|(|\xi|^{m-1}+|\eta|^{m-1})$ if $m>1$. Anyway, since $u-\psi\in C^1(\overline{B}_r)$ with $D(h-\phi)(x_0)=0$ one gets a contradiction by letting $s\to 0$.
\end{proof}

Now we are able to conclude the proof of Hessian regularity estimates.\smallskip

Let us take a bounded $C^{1,1}$ domain $\Omega^{\prime\prime}$ such that $\Omega^\prime\subset\subset\Omega^{\prime\prime}\subset\subset\Omega$.

Firstly we note that the function $g$ in Lemma \ref{RHS lemma} satisfies $g\in L^\infty(\Omega^{\prime\prime})$ when $b,\mu,f\in L^\infty(\Omega^{\prime\prime})$ (Theorem \ref{C1,logLip theorem}), while $g\in C^\alpha (\overline{\Omega^{\prime\prime}})$ when $b,\mu,f\in C^\alpha (\overline{\Omega^{\prime\prime}})$ (Theorem \ref{Schauder theorem}).
Next, we have that $u$ is an $L^n$-viscosity solution of $G(x,D^2 u)=g$ in $\Omega^{\prime\prime}$ by Lemma \ref{RHS lemma} and Proposition \ref{prop Lp iff Lupsilon}.
Then we obtain the corresponding regularity from the second order operator $G(x,X)$ without lower order terms.
In the first case $u\in C^{1,\mathrm{Log-Lip}}(\Omega^\prime)$ by \cite{ST}; while in the second case $u\in C^{2,\alpha}(\overline{\Omega^\prime})$ by \cite{Caf89} and \cite{CC03}.

For the corresponding estimates we argue as follows.
To make the proof shorten let us just write $[u]$ to mean the norm of $u$ either in  $C^{1,\mathrm{Log-Lip}}(\Omega^\prime)$ or $C^{2,\alpha}(\overline{\Omega^\prime})$; we also denote simply $|f|_\alpha ,\, |b|_\alpha , \,|\mu|_\alpha$ with $\alpha\ge 0$, for the norms in either $L^{\infty}(\Omega^{\prime\prime})$ when $\alpha=0$, while in $C^{\alpha}(\overline{\Omega^{\prime\prime}})$ when $\alpha>0$. With these conventions, we derive from \cite{Caf89, ST} that
\begin{center}
$[u]\le C \,\{\,\|u\|_{L^\infty (\Omega^{\prime\prime})} +|g|_\alpha \,\}$, \;\; $\alpha \in [0,1)$,
\end{center}
for some universal positive constant $C$ depending only on $n,\lambda,\Lambda$, in addition to $\bar{\alpha}, \theta, r_0$ from ($\tilde{A}$3), and on $\Omega^{\prime\prime}$, $\mathrm{dist}(\Omega^\prime, \partial\Omega)$. Now it remains to estimate $|g|_\alpha$ from Lemma \ref{RHS lemma} in terms of the $C^{1,\alpha}$-norm of $u$ from Theorem \ref{C1,alpha regularity estimates geral}, in order to obtain estimates \eqref{est C1,LogLip} and \eqref{est Schauder local}.

Since $u\in C^{1}(\overline{\Omega^{\prime\prime}})$, we consider $G(x,D^2 u)=g$\, in $\Omega^{\prime\prime}$, and estimate
\begin{center}
$|g|_\alpha\,\le\, |f|_\alpha+ |b|_\alpha\, \|u\|_{ C^{1,\alpha}(\overline{\Omega^{\prime\prime}})} +|\mu|_\alpha\, \|u\|_{ C^{1,\alpha}(\overline{\Omega^{\prime\prime}})}^m$.
\end{center}
In the case $m=1$ then we are done by the $C^{1,\alpha}$ estimates from \cite{Norn19} given in \eqref{estim C1,alpha local}. \smallskip

When $m\in (1,2]$ we assume $\overline{W}([f])\le \overline{W}_0$ where $\overline{W}$ is defined in \eqref{def W geral BMO}. Now, since
\begin{center}
	$\|u\|_{C^{1,\alpha}(\overline{\Omega^{\prime\prime}})}^m \le   (C \overline{W}_0)^{m-1} \, \|u\|_{C^{1,\alpha}(\overline{\Omega^{\prime\prime}})}$,
\end{center}
we derive the estimates \eqref{est C1,LogLip} and \eqref{est Schauder local} directly.

Instead, if $m\in (0,1)$ we need to work a little bit more. Let us denote $\omega=\|u\|_{L^\infty(\Omega)}+|f|_\alpha$, so
\begin{center}
$[u]\le C\, \left\{\,\omega + |b|_\alpha\, C (\omega +|\mu|_\alpha^{\frac{1}{1-m}} ) +|\mu|_\alpha \, C (\omega +|\mu|_\alpha^{\frac{1}{1-m}} )^m \,\right\}$.
\end{center}
Then, we split the analysis in two cases: either
(i)	 $\omega \le |\mu|_\alpha^{\frac{1}{1-m}}$, or (ii) $\omega \ge |\mu|_\alpha^{\frac{1}{1-m}}$. \medskip

In case (i) we have $|\mu|_\alpha (\omega +|\mu|_\alpha^{\frac{1}{1-m}} )^m\le 2^m |\mu|_\alpha\, |\mu|_\alpha^{\frac{m}{1-m}} =2^m |\mu|_\alpha^{\frac{1}{1-m}} $; while case (ii) produces
\begin{center}
	$|\mu|_\alpha(\,\omega +|\mu|_\alpha^{\frac{1}{1-m}} )^m \le \, |\mu|_\alpha \,\omega^m +|\mu|_\alpha^{\frac{1}{1-m}}$,
\end{center}
and by writing \, $|\mu|_\alpha \, \omega^m = |\mu|_\alpha^{\frac{1}{1-m}} (\, {\omega}\,{|\mu|_\alpha^{\frac{1}{m-1}}})^{m}\le\, |\mu|_\alpha^{\frac{1}{1-m}}\, {\omega}\,{|\mu|_\alpha^{\frac{1}{m-1}}}=\omega$, since $m<1$. Hence, the estimates \eqref{est C1,LogLip} and \eqref{est Schauder local} are also achieved in this case.

\begin{rmk}
When $b,\mu \in L^\infty (\Omega)$, then $u$ is pointwise twice differentiable a.e.\ in the sense of possessing a second order Taylor expansion, see \cite{CCKS}. Indeed, for $m<1$, $u$ is an $L^p$-viscosity subsolution of
	\begin{align}\label{estim Holder m<1}
	\textrm{$\mathcal{M}^+(D^2u)+(b+m)|Du|\ge -f^-(x) - (1-m)\,\mu^{\frac{1}{1-m}} $ \;\; in $\Omega^{\prime\prime}$,}
	\end{align}
	since $\mu |Du|^m\le (1-m)\mu^{\frac{1}{1-m}}+m|Du|$ by Young's inequality. Then the twice superdifferentiable a.e.\ for $u$ comes for instance from \cite{CCKS}, since $b$ is bounded. The subdifferentiable case for supersolutions is analogous. For the superlinear gradient growth with $m\in (1,2]$ see \cite[Corollary 7.4]{multiplicidade}.
\end{rmk}

\subsection{Sobolev and BMO regularity estimates}\label{section a priori BMO}

In this section, we exploit some finer BMO estimates for solutions which are a priori in suitable Sobolev spaces and satisfying the respective estimates. They might be seen as an improvement of Log-Lipschitz regularity, in light of \cite{DKM14} and \cite{PT16}, see also \cite{CafHua03}.

\begin{defin}\label{def BMO and p-BMO}
	For $p\in [1,+\infty)$, we say that $f\in p\text{--BMO}(\Omega)$ if $f$ is $L^p$-integrable in $\Omega$ and
	\begin{align}\label{eq:def p-BMO}
[f]_{p\text{-BMO}(\Omega)} \defeq \sup_{B\subset \Omega \textrm{ (ball) }} \left( \, \intav{B} |\, f(x)-f _{B}\,|^pdx\right)^{1/p} <+\infty, \quad \text{with} \;\; f_{B}=\intav{B}f(x)dx.
	\end{align}
	When $p=1$ one says simply that $f\in\text{BMO}(\Omega)$.
\end{defin}
As a consequence of John-Nirenberg inequality, BMO functions are $L^p$-integrable for all $p\in (1,+\infty)$ and the norms $p$--BMO and BMO are equivalent. Namely, it holds:
\begin{align}\label{equiv BMO e p-BMO}
\textstyle \mathrm{C}^{-1}_p [f]_{\text{BMO}(\Omega)} \leq [f]_{p\text{--BMO}(\Omega)} \leq \mathrm{C}_p [f]_{\text{BMO}(\Omega)} \quad \textrm{ for $f\in \text{BMO}(\Omega)$},
\end{align}
where the constant $\mathrm{C}_p>0$ depends only upon $p$ and $n$.
In particular, $\text{BMO}(\Omega)\subset L^p(\Omega)$ for all $p\in [1,+\infty)$ whenever $\Omega$ is a bounded domain.

\begin{defin}\label{def Ws,p estimates}
	We say that the equation \eqref{F=f} enjoys $W^{s,p}$ estimates if any bounded viscosity solution $u$ of \eqref{F=f} with $f\in L^p(\Omega)$ belongs to $ W^{s,p}_{\mathrm{loc}}(\Omega)$ and satisfies
\begin{center}
	$\|u\|_{W^{s,p}(\Omega^\prime)} \leq C\,W(\|f\|_{L^p(\Omega)})$, \quad for any $\Omega^\prime\subset\subset\Omega$,
\end{center}
where $W$ is given in \eqref{def W geral}, and
 $C$ depends on $n,p,\varrho, q,\lambda,\Lambda,\mu,\| b \|_{L^\varrho (\Omega)}, \|\mu\|_{L^q(\Omega)}, \mathrm{dist}(\Omega^\prime, \partial\Omega), \mathrm{diam} (\Omega)$, also on $C_1$ so that $W\le C_1$ if $m>1$.

If, in addition, $u=\psi$ on $\partial\Omega$ for some $\psi \in W^{s,p}(\Omega)$, then $u\in W^{s,p}(\Omega)$, and
\begin{center}
	$\|u\|_{W^{s,p}(\Omega)} \leq C\, W(\|f\|_{L^p(\Omega)}+ \|\psi\|_{W^{s,p} (\Omega)} )$,
\end{center} where $W$ is given in \eqref{def W geral}, for a constant depending also on $\partial\Omega$.
\end{defin}

We notice that $W^{s,p}$ estimates for solutions of \eqref{F=f} in the sense of Definition \ref{def BMO and p-BMO} yield BMO regularity estimates for the gradient and for the Hessian. Despite this could be known among specialists, we give a proof for the sake of completeness in order to illustrate an application of the iterations method.

\begin{teo}[{\bf A priori $\text{BMO}$ estimates}]\label{BMO a priori}
Let $u$ be a viscosity solution to \eqref{F=f}, where $F$ has the form \eqref{MC_intro}. Suppose that \eqref{F=f} enjoys $W^{s, p}$ estimates in the following sense:
	\begin{enumerate}[(i)]
\item If $f \in L^n(\Omega)$, $s=1$, and $p=n$, then $|D u| \in \text{n--BMO}_{\text{loc}}(\Omega)$ and it satisfies
		\begin{equation}\label{a priori BMO Du}
		[D u]_{\textrm{n--BMO}(\Omega^\prime)} \le C{W}(\|f\|_{L^n(\Omega)}) , \textrm{ for any $\Omega^\prime\subset \subset \Omega$,  with $W$ as in }\eqref{def W geral};
		\end{equation}
		\item If $f \in L^\infty(\Omega)$, $s=2$, and $p>n$, then $D^2 u \in p \text{--BMO}_{\text{loc}}(\Omega)$ and it holds
		\begin{equation}\label{a priori BMO D2u}
		[D^2 u]_{p\textrm{--BMO}(\Omega^\prime)} \le C \overline{W}(\|f\|_{L^\infty (\Omega)}),  \textrm{ for any $\Omega^\prime\subset \subset \Omega$, with $\overline{W}$ as in }\eqref{def W geral BMO}.
		\end{equation}
	\end{enumerate}
\end{teo}

\begin{proof}
First, we perform a scaling on the data $b, \mu, f, \bar{\beta}_F$ to ensure the smallness assumption \eqref{eq scaling delta} in the ball $B_2$.	For $(i)$, we approximate the rescaled $u$ by a viscosity solution of $G(0, D^2 {h}) = 0$ in $ B_2$; and use the linear functions  $\frak{l}_k (x)=\frak{a}_k +\frak{b}_k \cdot x$ to iterate $v_k$ defined in \eqref{vk do lk} for $\alpha =0$ (see section \ref{proof gradient regul}).
	Hence, by invoking $W_{\text{loc}}^{1, n}$ estimates we get
	\begin{center}
		$
	\|D v_k\|_{L^n\left(B_{{1}}\right)} \leq C\, \{\|v_k\|_{L^{\infty}(B_2)}+\|f_k\|_{L^n(B_2)}\} \leq C.
	$
	\end{center}
	Therefore, for $B^k=B_{r_k}(0)$ and $r_k=\gamma^k$, we have
\begin{center}
	$\left(\intav{B^k}|D u-\mathfrak{b}_k|^n\right)^{\frac{1}{n}} = \left(\intav{B^k}|D(u-\mathfrak{l}_k)|^n\right)^{\frac{1}{n}} \leq \|D v_k\|_{L^n(B_1)} ,$
\end{center}
thus
	$$
\textstyle	\left(\intav{B^k}|D u-(Du)_{B^k}|^n\right)^{\frac{1}{n}} \le \left(\intav{B^k}|D u-\mathfrak{b}_k|^n\right)^{\frac{1}{n}} +
	\left(\intav{B^k}| \mathfrak{b}_k-(Du)_{B^k}|^n\right)^{\frac{1}{n}} \le  2\|D v_k\|_{L^n(B_1)}\le C .
	$$
	Finally, for $r > 0$, we choose an integer $k$ such that $0 < \rho^{k+1} < r \le \rho^k$, and return to the original $u$. We thereby obtain the BMO estimate of the gradient in \eqref{a priori BMO Du} after a covering argument.
	
	For $(ii)$ we approximate $u$ by a viscosity solution of $	G(0, D^2 {h}+\frak{c}_k) = g(0)$ in $ B_2$ and iterate the quadratic polynomial functions \begin{center}
		$\frak{P}_k (x)=\frak{a}_k+\frak{b}_k\cdot x +\frac{1}{2}(\frak{c}_k x,x)$, \quad   via \; $v_k=\frac{u_k(r_k x)-\frak{P}_k (r_k x)}{r_k^{2}}$,
	\end{center}
see the method in \cite{ST}. Thus, $W^{2, p}$ estimates yield
\begin{center}
	$	\|D^2 v_k\|_{L^p(B_{{1}})} \leq C\,\{\,\|v_k\|_{L^{\infty}(B_2)}+\|f_k\|_{L^\infty(B_2)}\} \leq C.	$
\end{center}
	As a consequence, one gets for $B^k=B_{r_k}(0)$,
	\begin{center}
		$	\left(\intav{B^k}|D^2u-\frak{c}_k\,|^p\right)^{\frac{1}{p}} = \left(\intav{B^k}|D^2(u-\frak{P}_k)|^p\right)^{\frac{1}{p}} \leq \|D^2 v_k\|_{L^p(B_{{1}})} ,
	$
	\end{center}
	and so
	$$\textstyle \left(\intav{B^k}|D^2 u-(D^2u)_{B^k}|^n\right)^{\frac{1}{n}} \le \left(\intav{B^k}|D u-\frak{c}_k|^n\right)^{\frac{1}{n}} +
	\left(\intav{B^k}
	|\frak{c}_k-(D^2u)_{B^k}|^n\right)^{\frac{1}{n}}  \le  2\,\|D^2 v_k\|_{L^n(B_1)} \leq C.
$$
	Hence, for $r > 0$ we pick up an integer $k$ so that $0 < r_{k+1} < r \le r_k$. Then the previous sentence, the rescaling back to $u$, and the equivalence on the $p-$BMO and BMO norms in \eqref{equiv BMO e p-BMO} provide the desired BMO estimate \eqref{a priori BMO D2u} after a covering argument.
\end{proof}

Examples of equations satisfying $W^{s,p}$ estimates as in Definition \ref{def Ws,p estimates} for linear gradient growth with bounded coefficients can be found in \cite[Theorems 1.4]{DKM14} and \cite[Corollary 2.2]{Swiech}.
Moreover, in \cite{Swiech20} it was pointed out that the same regularity estimates could be achieved for unbounded coefficients by using the variable change in \cite{Norn19}. Additionally, \cite[Theorem 1.1]{Kry20} addresses local $W^{2,p}$ estimates for linear and fully nonlinear elliptic equations with $L^n$-drift. We also refer to \cite{daSilRic19} and \cite{PT16} for $W^{2,p}$ and BMO global/local estimates in the case $m=1$ and bounded coefficients under relaxed convexity assumption;
see also the discussion in Section \ref{section flat} in order to include an unbounded $\mu (x)$.

In the sequel, we bring forward a class of equations with nonlinear gradient growth for which we do obtain such crucial $W^{s,p}$ regularity estimates (in the sense of Definition \ref{def Ws,p estimates}) required in Theorem \ref{BMO a priori}.
We begin by ensuring the validity of \textit{a priori} estimates in a unified way.

\begin{prop}[{\bf Generalized Nagumo's lemma}]\label{Nagumo}
Let $\Omega\subset\rn$ be a bounded $C^{1,1}$ domain. Let $F$ satisfy (A1)--(A3), and $f\in L^p(\Omega)$, $p>n$.
	Suppose that there exists $\theta>0$ such that \eqref{Htheta} holds for some $r_0>0$ and for all $x_0 \in \overline{\Omega}$. Then, strong solutions of \eqref{F=f} satisfy $W^{2,p}$ estimates, for a constant also depending on $r_0$.
\end{prop}

\begin{proof} We consider the global case; the local one it is analogous. Let $\psi\in W^{2,p}(\Omega)$ with $u=\psi$ on $\partial\Omega$, then $u\in C^{1,\alpha}(\overline{\Omega})$ by Theorem \ref{estim C1,alpha local} and Remark \ref{rmk global}.
	Next, $u$ satisfies $F(x,0,D^2 u)=g(x)$ a.e.\ in $\Omega$,
	where $$g(x)\defeq f(x)-F(x,Du,D^2u)+F(x,0,D^2u)\in L^p (\Omega),$$
	since
	$|g(x)|\leq H(x,Du,0) \in L^p(\Omega).$
	Now, by the estimates in \cite{Swiech, Winter} one finds
	\begin{align*}
	&\hspace{3cm}\|u\|_{W^{2,p}(\Omega)} \leq C\, \{ \| u \|_{L^{\infty} (\Omega)} + \|g \|_{L^p (\Omega)}  + \|\psi\|_{W^{2,p} (\Omega)} \}\\
	&\leq C\,\{ \,\| u \|_{L^{\infty} (\Omega)} + \|f \|_{L^p (\Omega)} +\|\mu\|_{L^p(\Omega)} \|u\|_{C^1(\overline{\Omega})}^m + \|\psi\|_{W^{2,p} (\Omega)}+\|b\|_{L^p(\Omega)}\|u\|_{C^1(\overline{\Omega})} \,\}.
	\end{align*}
	Now, it is matter of using our $C^{1,\alpha}$ estimates in Remark \ref{rmk global} to deduce
	\begin{align}\label{est nagumo2}
	\textrm{	 $\|u\|_{W^{2,p}(\Omega)} \le C W(\|f\|_{L^p(\Omega)}+\|\psi\|_{C^{1,\tau}(\partial\Omega)})$, \; where\,  ${W}\le W_0$ is given in \eqref{def W geral}.}
	\end{align}
Indeed, we argue exactly as in the end of Section \ref{proof Schauder}. The only difference is that now the computations and constant dependence lie on $\|\mu\|_{L^q(\Omega)}$ and $\|b\|_{L^\varrho(\Omega)}$ in place of $|b|_\alpha$ and $|\mu|_\alpha$ there.
\end{proof}

\begin{prop}[{\bf $W^{2,p}$ regularity}]\label{prop regu W2,p} Assume (A1)--(A3). Then the equation \eqref{F=f}, where $F$ has the form \eqref{MC_intro} with $h\equiv H$, enjoys  $W^{2,p}$ estimates, in the sense of Definition \ref{def Ws,p estimates}, whenever $G$ is convex or concave in the $X$ entry, $f\in L^p(\Omega)$ for $p>n$, and $u$ is a bounded $L^p$-viscosity solution of \eqref{F=f}.
\end{prop}

\begin{proof}
	We just consider the local case, since the global one is analogous. By the a priori estimates in Proposition \ref{Nagumo} it is enough to show the regularity conclusion $u\in W^{2,p}_{\mathrm{loc}}(\Omega)$.
	We already know that $u\in C^{1,\alpha}_{\mathrm{loc}}(\Omega)$ by Theorem \ref{C1,alpha regularity estimates geral}.
	Let us take sequences $b_k,\mu_k\in L^\infty_+(\Omega)$ with $b_k\to b$ in $L^\varrho(\Omega)$ and $\mu_k\to \mu$ in $L^q(\Omega)$, and let $u_k$ be the $L^p$-viscosity solution of
	\begin{align}\label{eq:claim visc k}
	\textrm{$G(x,D^2 u_k) + b_k (x)|Du_k|+ \mu_k (x)|Du_k|^m  =f(x)$ \;in $B_{2r}$, \quad $u_k=u$ \;  on $\partial B_{2r}$ ,}
	\end{align}
	given by \cite{KoKo17} and \cite{KSexist2009}, for $m\in (0,1)$ and $m\in[1,2]$ respectively. Here the solution $u_k$ exists up making $r$ smaller. Moreover, $u_k\in C^{1}(\overline{B}_{2r})$ by Remark \ref{rmk global}.
	
	We claim that $u_k\in W^{2,p}_{\mathrm{loc}}(B_{2r})$. Indeed, by Lemma \ref{RHS lemma}, $u_k$ is an $L^p$-viscosity solution of
	\begin{align}\label{eq:claim visc k 2}
	\textrm{$G(x,D^2 u_k)  =g(x)$ \;in $B_{2r}$, \quad where \; $g(x)\defeq f(x)- b_k (x)|Du_k|- \mu_k (x)|Du_k|^m $,}
	\end{align}
and then apply the $W^{2,p}$ regularity results in \cite{Caf89} for equations without zero order coefficients. 	
	Then, $u_k$ is the unique $L^n$-viscosity solution of \eqref{eq:claim visc k} by \cite[Theorem 1.1(iii)]{arma2010}. Also, Proposition \ref{Nagumo} yields
	$
	\|u_k\|_{W^{2,\varrho}(\Omega)} \leq C_k  W_k,
	$
	where $W_k$ is given in \eqref{def W geral} with $u,b,\mu$ replaced by $u_k,b_k,\mu_k$.
	Here $C_k\le C$, since $b_k$ and $\mu_k$ are uniformly bounded in $L^p$ norm.
	
	Now, by ABP we obtain
	$\|u_k\|_{L^\infty (B_{2r})}\leq \|u_k\|_{L^\infty (\partial B_{2r})} + C\, \|f\|_{L^p(B_{2r})}$, for
	a constant independent of $k$. Thus, we deduce 	$\|u_k\|_{W^{2,\varrho}(B_{2s})} \leq C$, for any $s<r$. Hence there exists $\bar{u}\in C^{1}(\overline{B}_{2s})$ such that $u_k\rightarrow \bar{u}$ in $C^{1}(\overline{B}_{2s})$, for any $s<r$.
	Then, Proposition \ref{stability} ensures that $\bar{u}$ is a viscosity solution of \eqref{F=f} in $B_{2r}$, where $F$ has the form \eqref{MC_intro}. Now, since $W^{2,p}(B_{2s})$ is reflexive, there exists $\tilde{u}\in W^{2,p}(B_{2s})$ such that $u_k$ converges weakly to $\tilde{u}$. By uniqueness of the limit, $\tilde{u}=\bar{u}$ a.e.\ in $B_{2s}$, and so $\bar{u}\in W^{2,p}(B_{2s})$, for all $s<r$.
	Finally, since $u$ is already a viscosity solution of \eqref{F=f}--\eqref{MC_intro}, the uniqueness assertion in \cite[Theorem 1.1(iii)]{arma2010} yields $u=\bar{u}\in W^{2,p}(B_{2s})$.
	
	Finally, we proceed with a covering argument to conclude $u\in W^{2,p}(\Omega^\prime)$, for any $\Omega^\prime\subset\subset\Omega$.
\end{proof}

We finish this section by exploiting $W^{1,p}$ regularity estimates for general solutions of \eqref{F=f} in the sublinear gradient regime.

\begin{prop}[{\bf $W^{1,p}$ regularity}]\label{prop W1,p regularity}
	Assume $F$ satisfy (A1)--(A3), with $m\le 1$. Let $f\in L^p (\Omega)$ and $u$ be a bounded $L^p$-viscosity solution of	\eqref{F=f}. Then, there exists $\theta>0$ depending on $n,m,p,\varrho, q,\lambda,\Lambda,\|b\|_{L^\varrho(\Omega)}$, such that if $F$ satisfy \eqref{Htheta} for some $r_0>0$, for any $r\leq \min \{ r_0, \mathrm{dist} (x_0,\partial\Omega)\}$, and for all $x_0 \in \Omega$, this implies that:
	\begin{enumerate}[(i)]
\item if $p=n$, then $u\in W^{1,Q}_{\mathrm{loc}}(\Omega)$ and $\|u\|_{W^{1,Q}({\Omega^\prime})} \leq C W(\|f\|_{L^n(\Omega)})$,  for all \, $Q<\infty$;
		
\item  if $p\in (p_0,n)$, then $u\in W^{1,Q}_{\mathrm{loc}}(\Omega)$ and 		$\|u\|_{W^{1,Q}({\Omega^\prime})} \leq CW (\|f\|_{L^p(\Omega)})$,	for all \, $Q<p^\star=\frac{pn}{n-p}$;
	\end{enumerate}
for any subdomain $\Omega^\prime \subset\subset \Omega$, where $W$ is given by \eqref{def W geral}.
	Here $C$ depends on $ r_0,n,p, \varrho, m,Q,\lambda,\Lambda, \| b \|_{L^\varrho (\Omega)}$, $\| \mu \|_{q},\mathrm{diam} (\Omega)$, $\mathrm{dist} (\Omega^\prime,\partial\Omega)$.
	If, in addition, $u\in C^{1} (\partial\Omega)$ for some $\tau\in (0,1)$, then regularity up to the boundary is true if $F$ satisfies \eqref{Htheta} for some $r_0>0$, for all $x_0\in \overline{\Omega}$.
	The respective estimates hold with $\Omega^\prime = \Omega$ by replacing $W$ by $W+\|u\|_{C^{1,\tau }(\partial\Omega)}$ on the right hand side.
\end{prop}

\begin{proof} Let us consider the local case. The case of boundary regularity estimates follows in a similar way, by using instead the approach in \cite{Winter}. We fix a domain $\Omega^\prime \subset\subset \Omega$.
	
	Observe that $p^\star >n$ since $p> n/2$. We are going to show that $u\in W^{1,Q}$ (with the respective estimates) for $Q <p^\star$, by understanding $p^\star=\infty$ if $p=n$.
	The proof is similar to the one in \cite{Swiech}, by applying instead the variable changes and approximation lemma in \cite{Norn19} together with our rescaling in Section \ref{section scaling}; we sketch the proof in what follows for the reader's convenience.
	
	We take $y\in \Omega^\prime$ and set $s_0(y)\defeq \min (r_0, \mathrm{dist}(y,\partial\Omega^\prime))$. Recall that $r_0 = r_0 (\theta)$ is such that \eqref{Htheta} holds, for all $r\leq s_0= \min \{ r_0, \mathrm{dist} (y,\partial\Omega)\}$, where $y \in \Omega$ is arbitrary.
	
	We fix an $\alpha$ with $0<\alpha\leq\min\{\bar{\alpha},\beta ,2-{n}/{p_0}\}$. Then, let $p^\prime$ so that $p>p^\prime>p_0$, and set
	\begin{center}
		$
	\textstyle	N(y)=\| u \|_{L^{\infty} (B_{s_0 (y)})} +
		\sup_{r\leq s_0}
		\left\{ r^{1-\alpha}
		\left( \,\intav{B_r (y)}
		|f(x)|^{p^\prime} \mathrm{d} x \right)^{1/{p^{\prime}}} \right\}\theta^{-1} + c(m)$\, a.e.\ $y\in \Omega^\prime$,
	\end{center}
	where $c (1)=0$, and $c (m)=\|\mu\|_{q}^{\frac{1}{m-1}}$ if $m<1$.
We infer that $N\in L^Q(\Omega)$ for   
$Q\leq \frac{np^\prime}{n-p^\prime (1-\alpha)}$, and
\begin{center}
$\|N\|_{L^Q(B_{s_0})} \le C \{\, \|u\|_{L^\infty (\Omega)}+\|f\|_{L^p(\Omega)} +c(m)\,\}=CW.$
\end{center}
In fact, this follows exactly as in the calculations at \cite[p.14]{Swiech}, since adding the constant $c(m)$ does not change the integrability.
In particular, $N(y)< \infty$ for a.e.\ $y\in \Omega^\prime$.

Now, we go back to Section \ref{section Approx and Scal} and then replace the balls centered at $0$ by balls centered at $y$, $\sigma\leq \frac{s_0(y)}{2}$ with $\overline{B}_{2\sigma} (y)\subset \Omega^\prime$.	
If $N=N(y)<\infty$, then $\widetilde{u}(x)=\frac{u(\sigma x)}{N}$	is an $L^p$-viscosity solution of $\widetilde{F}[\,\widetilde{u}\,]=\widetilde{f}(x)$ in $B_2(y)$, where $\widetilde{F}$ is given in Section \ref{section scaling}, in which we may assume the smallness condition \eqref{eq scaling delta} on the coefficients. Note that $\widetilde{u}$ is also an $L^{p^\prime}$-viscosity solution by using our Proposition \ref{prop Lp iff Lupsilon}.
	In addition,
	$\|{\bar{\beta}}_{\widetilde{F}}(\cdot,0)\|_{L^{p^\prime} (B_1(y))} \leq \delta $ by choosing $\theta=\delta/4$, and
	\begin{center}
		$
\textstyle		r^{1-\alpha}
		\left(\, \intav{B_r (y)}
		|\widetilde{f}(x)|^{p^\prime} \mathrm{d} x \right)^{1/{p^{\prime}}} \leq
		\frac{\sigma^{1+\alpha-\beta}}{W}
		(\sigma r)^{1-\alpha}
		\left(\intav{B_{\sigma r} (y)}
		|f(x)|^{p^\prime} \mathrm{d} x \right)^{1/{p^{\prime}}}  \leq \delta ,
		$
	\end{center}
	for all $r\leq 2$. This implies that
	$\tilde{N}(y)=N_{\tilde{u},\tilde{f}}(y)\le 1+\delta+c(m)<\infty$. \smallskip
	
From this point forward, we may proceed as Section \ref{proof gradient regul} to obtain the desired regularity and estimates around the point $y$ via Caffarelli's iterations scheme for $\widetilde{u}$. Moreover, as in \cite[Theorem 2.1]{Swiech},
\begin{center}
$ \left(\int_{\Omega^\prime} \frac{|u(x+y)-u(y)|^Q}{|x|^Q} \d y \right)^\frac{1}{Q} \le \|N\|_{L^Q(B_{s_0})} \le CW.$
\end{center}
The conclusion is finally obtained by passing to limits $p^\prime\rightarrow p$ and $\alpha \rightarrow 0$.	
\end{proof}

The same method above does not seem to work in the superlinear gradient regime since $N(y)$ is not necessarily a bounded function. It is an open issue to understand whether or not $W^{1,p}$ regularity estimates hold in this case.

\section{Sharp regularity and applications}\label{section examples}

In this section our goal is twofold. In the first two sections, namely \ref{section recession} and \ref{section flat}, we illustrate some examples for which we can access sharp $C^{1,\alpha}$ regularity and beyond.
On the other hand, in the last two sections we treat two applications for the general theory of PDEs.
More precisely, in Section \ref{section p-Lap} we discuss about how to obtain regularity estimates for singular equations from the sublinear gradient growth scenario; while in Section \ref{section perron} we derive a Perron type result for our equations, which turns out to be an essential tool in proving multiplicity results related to quadratic gradient growth regime.
\vspace{0.05cm}

First of all, by a \textit{sharp} $C^{1,\alpha}$  regularity result for operators with unbounded coefficients, we mean a regularity result for which $\bar{\alpha}=1$, where $\bar{\alpha}$ is the exponent in Proposition \ref{prop F(X)}.
In this case, we may we may import from  the homogeneous scenario, the explicit $\alpha$ choice in \eqref{choice alpha C1,alpha},  under an unbounded drift and nonlinear gradient terms.
In other words, solutions of \eqref{F=f} will belong to $C^{1, \hat{\alpha}}$, where
\begin{align}\label{sharp hat alpha}
\textstyle \hat{\alpha} = \min \left\{ 1-\frac{n}{p} \,, \, 1-\frac{n}{\varrho}\, , \, 1-\frac{n}{q} \right\}, \quad \text{for}\;\; q, \varrho \in (n,+ \infty], \;\, p\in (n,+\infty).
\end{align}
Note that if for instance $b\in L^\infty(\Omega)$ then the restriction on $\varrho$ never appears; that is why we can include also the case $\varrho=+\infty$ by understanding $1/\varrho$ to be zero; the same holds for $\mu$. In particular, if both $b,\mu\in L^\infty(\Omega)$ and $m\in (0,2)$, then we get the sharp exponent $\hat{\alpha}=1-{n}/{p}$.

\begin{rmk}
If\, $\bar{\alpha}=1$, and $f,b,\mu\in L^\infty(\Omega)$, then $u\in C^{1,\alpha}$ for all $\alpha<1$. Moreover, under the hypotheses of Theorem \ref{C1,logLip theorem} we recover an improved local regularity estimate in the space $C^{1,\text{Log-Lip}}$.
\end{rmk}

\vspace{0.05cm}

As mentioned in the Introduction, a first example beyond Pucci's operators $\mathcal{M}_{\lambda, \Lambda}^{\pm}(D^2 u)$, is a Belmann type equation which appear in stochastic control, such as
\begin{center}
$
F(x, D^2 u) =  \inf_{{\alpha} \in \mathcal{A}}\left(\mathcal{L}^{{\alpha}} u(x)\right) \;\, \left(\text{or}\,\sup_{{\alpha} \in \mathcal{B}}  \left(\mathcal{L}^{{\alpha}} u(x)\right)\right)$, \quad where $\mathcal{L}^{{\alpha}} u(x) = \sum_{i, j=1}^{n} a_{ij}^{{\alpha}}(x)\partial_{ij} u $,
\end{center}
i.e., each $\mathcal{L}^\alpha$ is a family of uniformly $(\lambda,\Lambda)$--elliptic operators with H\"{o}lder continuous (up to the boundary) coefficients, which are convex or concave in the Hessian entry for which a $C^{2,\alpha}$ theory is available.
More generally, we can consider concave/convex operators, for which the Evans-Krylov-Trudinger's theory of $C^{2, \alpha}$ estimates holds for the respective homogeneous problem, see for instance \cite{Ev82, Kry82, Tru83}.

As far as nonconvex/nonconcave operators are concerned, another example where Theorem \ref{C1,alpha regularity estimates geral} addresses sharp estimates is the class of nonconvex fully nonlinear equations studied by Cabr\'{e}-Caffarelli in \cite{CC03}, where they obtained $C_{\text{loc}}^{2,\alpha}$ estimates.
In this spirit, we discuss in Sections \ref{section recession} and \ref{section flat} ahead, two other classes of nonconvex problems for which we still obtain optimal $C^{1,\alpha}$ regularity.

\subsection{Sharp regularity via recession profiles}\label{section recession}

An alternative way to relax convexity/concavity assumptions on the nonlinearity $F$ but still obtain ``good'' regularity properties is through the concept of \textit{recession} operator developed in \cite{ST},
$$
 \displaystyle F^{\ast}(X) \defeq  \lim_{\tau \to 0+ } \tau F\left(\frac{1}{\tau}X\right).
$$
Roughly speaking, the hypothesis of $C^{1,1}$ a priori estimates for solutions of the associated pure second order equations is asked from $F$ only when $\|D^2 u\| \approx \infty$.
The authors proved in \cite{ST} that if solutions to the homogeneous tangential equation
$F^{\ast}(D^2 u) = 0 $ in $ B_1$ enjoy $C^{1, \alpha_0}$ a priori estimate for some $\alpha_0 \in (0, 1]$ then solutions to $F(D^2 u) = 0 $ in $ B_1$ (or with RHS $f \in L^{\infty})$
are of class $C^{1, \bar{\alpha}}_{\text{loc}}(B_1)$ for any $\bar{\alpha} <\alpha_0$. We also quote \cite{DongKry19} and  \cite{Kry13} for fundamental existence/regularity results in the context of weighted and mixed-norm Sobolev spaces under either relaxed or no convexity assumptions.

\vspace{0.02cm}

In this sense, we may prove our Theorem \ref{C1,alpha regularity estimates geral}, in a sharp way,  to operators under either relaxed or no convexity assumptions on $F$ such that $\alpha_0 = 1$, cf.\ \cite[Theorems 1.1 and 1.4]{ST}. Moreover, for $n< p, q, \varrho < \infty$, the following estimate holds true for an $L^p$-viscosity solution $u$ of \eqref{F=f},
\begin{align*}
\textstyle \|u\|_{C^{1, {\alpha}}(\overline{\Omega^\prime})} \leq C W (\|f\|_{L^p(\Omega)}),\; \textrm{ for all }\;\; \Omega^\prime\subset\subset\Omega, \;\; \alpha\le \bar{\alpha}, \;\; \alpha<\bar{\alpha},
\end{align*}
where $W$ is given in \eqref{def W geral}, and $\hat{\alpha}$ in \eqref{sharp hat alpha}.

For instance, if $F:\mathbb{S}^n \to \mathbb{R}$ is a $C^1$ uniformly elliptic operator, then $F^{\ast}$ should be understood as the ``limiting equation'' for the natural scaling on $F$, i.e.
\begin{center}
	$\displaystyle \mathfrak{A}_{ij} \defeq \lim_{\|X\|\to \infty} \frac{\partial F}{\partial X_{ij}}(X), $ \quad and \;\, $F^{\ast}(X) = \tr(\mathfrak{A}_{ij}X)$.
\end{center}
That is, in this case we recover $\alpha_0=1$ since the recession profile is the Laplacian operator, up to a change of variables.
Let us illustrate a couple of examples of this sort, see also \cite{daSilRic19, PT16}. Set $0 < \sigma_1, \ldots , \sigma_n< \infty$.

\begin{example}[{The $m^{\prime}$-momentum type operator}]
Let $m^{\prime}$ be a fixed odd number, then
  $$
\textstyle F_{m^{\prime}}(D^2 u) = F_{m^{\prime}}(e_1(D^2 u), \cdots, e_n(D^2 u))
    =\sum_{j=1}^{n} \sqrt[m^{\prime}]{\sigma_j^{m^{\prime}}+e_j(D^2 u)^{m^{\prime}}}-\sum_{j=1}^{n} \sigma_j
    $$
is uniformly elliptic, neither concave nor convex, for which $F_{m^{\prime}}^{\ast}$ is the Laplacian operator. For instance, this is also the case of the class of Hessian operators with $\sigma_j=1$ for every $1\le j \le n$.
\end{example}

\begin{example}[{Lagrangian perturbations type operators}]
Let $h:(0, \infty) \to (0, \infty)$ be a continuous function, and $g: \mathbb{R} \to \mathbb{R}$ be bounded function with $g(0) = 0$, thus
  $$
\textstyle    F(D^2 u) =  F(e_1(D^2 u), \cdots, e_N(D^2 u))
     = \sum_{j=1}^{n} \left[h(\sigma_j)e_j(D^2 u) + g(e_j(D^2 u))\right],
  $$
is uniformly elliptic, neither concave nor convex. Also,
$\textstyle  F^{\ast}(X) = \sum_{j=1}^{n} h(\sigma_j)e_j(X)$,
which is the Laplacian operator after a change of coordinates.
For instance, one may take $g(x)=\arctan x$.
\end{example}

\subsection{Sharp regularity for flat solutions of nonconvex operators}\label{section flat}

In this section, we apply the method developed in Sections \ref{section Approx and Scal} and \ref{proof main th} to obtain optimal regularity for a class of nonconvex operators.

In \cite[Theorem 1]{DD19} and \cite[Theorem 2.2]{dosPT16} it was established local $C^{2, \alpha}$ estimates for flat solutions (i.e.\ solutions whose $L^\infty$-norm are very small) of nonconvex fully nonlinear operators under H\"{o}lder continuity on the coefficients and on the source term (RHS). In particular, this implies that $\bar{\alpha}=1$, i.e. the regularity exponent for the homogeneous problem with frozen coefficients.

In this spirit, for such a class of solutions, we consider the following model problem
\begin{equation}\label{MEq1.1}
F(x, D^2 u) +b(x)|Du|+ \mu(x)|D u|^m = f(x) \quad \text{in} \quad B_2,
\end{equation}
in the context of unbounded ingredients.
\begin{itemize}
\item[(A5)]({{\bf Requirements for flat regime}})
We assume (A1), (A3) for $(x,\xi,X)\mapsto F(x,X)$, with $b\in L^\varrho_+(\Omega)$, $\mu\in L^q_+(\Omega)$, in addition to:
\begin{enumerate}[(i)]	
\item \label{gateaux} $F$ is differentiable at the origin, and the operator $X \mapsto D_F(0)(X)$ (the G\^{a}teaux derivative of F at $0$) is a (linear) second order operator with frozen coefficients;
\item  $F \in C^{1, \omega}(\mathbb{S}^n)$ for a modulus of continuity $\omega$, and there holds
\begin{equation*}	\|D_{M}F(x, X)-D_{M}F(x, Y)\| \leq\omega(\|X-Y\|), \;\;\textrm{ for all $x \in B_2$, \,and $X,Y\in \mathbb{S}^n$.}
		\end{equation*}
	\end{enumerate}
\end{itemize}

Let us explain the mechanisms and heuristic behind of flatness approach. For the sake of simplicity we assume $F$ has constant coefficients. Then, one can associate to $F$, the following family of operators
$$
\mathcal{F}_{\tau}(X) \defeq \frac{F(\tau X)-F(0)}{\tau}, \quad \tau>0.
$$
Now, notice that $(\mathcal{F}_{\tau})_{\tau>0}$ produces a continuous ``curve of operators'' preserving the same ellipticity parameters $\Lambda \geq \lambda>0$. Now, by item (A5) \eqref{gateaux}, and $F(x,0)=0$, one checks that
\begin{equation}\label{eqGTE}
\displaystyle \mathcal{F}_0(X)\defeq  \lim_{\tau \to 0^{+}} \mathcal{F}_{\tau}(X)  = D_F(0)(X).
\end{equation}
Thus, the (linear) second order operator in  \eqref{eqGTE} is the limit tangential equation of $\mathcal{F}_{\tau}$ as $\tau \to 0^{+}$. Thus, up to a change of coordinates we can say that $\mathcal{F}_0(X)=\mathrm{tr}(X)$, i.e.\ the Laplacian operator.

Moreover, if $u$ is a flat solution of \eqref{MEq1.1} with flatness degree $\tau$, i.e.\ $\|u\|_\infty\le \tau$, then $v \defeq \frac{u}{\tau}$ is a normalized solution of
\begin{equation}\label{eqGaux}
\mathcal{F}_{\tau}(D^2 v)  +b_\tau(x)|Dv|+ \mu_{\tau}(x)|D v|^m = f_{\tau}(x) \quad \text{in} \quad B_2,
\end{equation}
where $b_\tau(x)\defeq b( x)$, $\mu_{\tau}(x) \defeq \frac{1}{\tau^{1-m}}\mu(x)$, and $f_{\tau}(x) \defeq \frac{1}{\tau}f(x)$. Thus,  \begin{center}
	$\|b_\tau \|_{L^\varrho(B_2)}, \;\, \|\mu_{\tau}\|_{L^q(B_2)} , \;\, \|f_{\tau}\|_{L^p(B_2)} = \text{o}(1)$ \;\; as $\tau \ll 1$,
\end{center}
Let us then access the available regularity for the (linear) limiting profile via the geometric tangential device from \eqref{eqGTE} combined with compactness and stability-like methods, see Section \ref{section lemmas}.

In other words, we may understand the Laplace equation as the \textit{geometric tangential PDE} of the limit formed by the ``curve'' of fully nonlinear operators $\mathcal{F}_{\tau}$, provided the flatness degree $\tau$ of $u$, as well as the data, are universally under small control.

\begin{lem}[{\bf Approximation Lemma}]\label{Key lemma0} Let $F$ satisfy (A5), and $f\in L^p(\Omega)$ for $p>n$. Then there exists $\tau_0>0$ depending on $n,m, \lambda, \Lambda,p,q,\varrho,  \omega$, %$\|b\|_{L^\varrho (B_2)}$, and $\|\mu\|_{L^q(B_2)}$
such that if $v$ is an $L^p$-viscosity solution to
	$$
	\tau^{-1}F(x, \tau D^2v) +b(x)|Dv|+ \mu(x) |D v|^m= f(x) \quad  \mbox{in} \; B_2 \quad  \mbox{with} \;\; \|v\|_{L^{\infty}(B_2)}\le 1,
	$$
for $\tau \in (0,\tau_0)$,	and
	$$
 \max\left\{\tau , \,\,\, \|\beta_F(0,\cdot)\|_{L^p(B_2)}, \,\, \| \mu\|_{L^{q }(B_{2})}\,, \,  \| b\|_{L^{\varrho }(B_{2})}\, , \, \| f\|_{L^{p }(B_{2})}\right\} \leq \tau_0
	$$
then we can find $ \gamma \in (0,1)$, depending only on $n,m,p,q,\varrho, \lambda, \Lambda, \omega$, and a harmonic function $\mathfrak{h}$ so that
\begin{equation}\label{eqappro_est}
\Delta \mathfrak{h} (x)  =  0 \;\;\; \text{in} \; B_1, \quad \|\mathfrak{h}\|_{L^\infty (B_1)} \le  \mathrm{C}, \quad
	 \sup_{B_{\gamma }}|v-\mathfrak{h}(0)-D\mathfrak{h}(0)\cdot x| \leq \gamma^{1+\alpha}.
		\end{equation}
\end{lem}

In the sequel, we proceed with Caffarelli's iterative scheme used in Section \ref{proof gradient regul} in order to obtain a sequence of linear functions $\frak{a}_k+\mathfrak{b}_k\cdot x$ satisfying the geometric decay estimate
$$
\textstyle 	\sup_{B_{r_k }}|v-\frak{a}_k-\frak{b}_k\cdot x| \leq r_k^{1+\alpha}, \qquad r_k \defeq \gamma^k.
$$
Finally, by coming back to $u$ via such a geometric tangential path $u=v\tau$  we show that the graph of a solution to \eqref{MEq1.1} can be approximated in $B_\sigma$, for certain $\sigma \ll 1$, by an appropriated harmonic polynomial profile, whose error is of the order $\sim \mbox{O}(\tau\sigma^{1+\alpha})$, cf.\ \cite[Lemma 3.2]{dosPT16}.

Hence, the previous estimate and a covering argument imply, via Campanato's embedding, that flat solutions are $C_{\text{loc}}^{1, \alpha}(B_2)$. Therefore, we obtain the following sharp estimate:

\begin{teo}\label{teo1} Assume (A5) and let $u$ be an $L^p$ viscosity solution to \eqref{MEq1.1}. Then there exists $ \tau > 0$ depending only upon $n,m,p,q,\varrho,\lambda ,\Lambda ,\omega$, such that if $\|u\|_{L^\infty(B_2)} \leq \tau$, then $u\in C_{\text{loc}}^{1, \hat{\alpha}} \left(B_{2} \right)$ and
$$
\|u\|_{C^{1, \hat{\alpha}}(B_1)} \le \mathrm{C} (n,p,q,\varrho, \lambda, \Lambda, m )\,\tau, \;\; \textrm{ where $\hat{\alpha}$ is given by }\eqref{sharp hat alpha}.
$$
\end{teo}

\subsection{Fully nonlinear singular PDEs}\label{section p-Lap}

In this section, one considers an application to fully nonlinear PDEs, which are degenerate elliptic in the sense that they become singular when the gradient vanishes.

Some of our inspirations in this direction are the works \cite{BD15, BDL19-2, KoKo17}.
The model one might have in mind is the $\alpha$-Poisson equation. Thus, as in \cite{KoKo17} we consider the following class of equations,
\begin{align}\label{quasilinearPDE}
G(x,Du,D^2u)\defeq |Du|^{\alpha-2}F(x,D^2u) + b(x)|Du|^{\alpha-1}+f(x)|Du|^{\alpha-2}+\mu (x)=0 \;\; \mathrm{in} \; \Omega,
\end{align}
where $1<\alpha<2$,  $b\in L^\varrho_+(\Omega)$, $\mu\in L^q_+(\Omega)$, $f \in L^p_+(\Omega)$, and for $s>0$, $x\in \Omega$, $X,Y\in \mathbb{S}^n $,
\begin{align}\label{satisfying}
\;\mathcal{M}^-(X-Y) \leq \;F(x,X) - F(x,Y) \leq  \mathcal{M}^+(X-Y), \quad
F(x,sX) = sF(x,X)  .
\end{align}

\begin{defin}[{bf Julin-Juutinen}]\label{def JJ}
	A function $u\in C(\Omega)$ is called an $L^p$-viscosity JJ-subsolution (resp.\ JJ-supersolution) of \eqref{quasilinearPDE} if for $\phi \in W^{2,p}_{\mathrm{loc}} (\Omega)$, it follows that
	\begin{align}\label{limSubsolution}
	\mathrm{ess.}\varlimsup_{\epsilon\to 0} \,\{G(y,D\phi(y),D^2\phi(y)) \, | \, y \in B_\epsilon (x), D\phi(y) \neq 0\} \geq 0
	\end{align}\vspace{-0.7cm}
	\begin{align}\label{limSubsolutionResp}
	(\textrm{resp., ess.}\varliminf_{\epsilon\to 0}\, \{G(y,D\phi(y),D^2\phi(y)) \, | \, y \in B_\epsilon (x), D\phi(y) \neq 0\} \leq 0 )
	\end{align}
	whenever $u-\phi$ assumes its local maximum (resp.\ minimum) at $x\in\Omega$ and $D\phi \neq 0$ a.e.\ in $B_r(x)\setminus \{x\} $ for some $r>0$.
	On the other hand, $u\in C(\Omega)$ is called a $C$-viscosity JJ-subsolution (resp.\ JJ-supersolution) of \eqref{quasilinearPDE} if \eqref{limSubsolution} (resp.\ \eqref{limSubsolutionResp}) holds provided that $\phi \in C^2(\Omega)$, $u-\phi$ takes its local maximum (resp., minimum) at $x\in\Omega$ and $D\phi\neq 0$ in $B_r(x)\setminus \{x\}$ for some $r>0$.
\end{defin}

\begin{rmk}\label{rmk singular locally bounded}
We observe that any notion of viscosity solution for the singular equation \eqref{quasilinearPDE} requires $f$ to be zero whenever $u$ is locally constant, as pointed out in \cite{BDru, BD15}.
\end{rmk}

Through Definition \ref{def JJ}, the authors in \cite{KoKo17} proved that solutions of \eqref{quasilinearPDE} satisfy
\begin{align*}
\mathcal{L}^+_{m}[u] \ge f , \;\;
\quad\;\mathcal{L}^-_m[u]  \le f  \;\;\; \mathrm{in} \; \Omega, \quad m=2-\alpha,
\end{align*}
in the $C$-viscosity JJ-sense, and further in the usual $C$-viscosity sense, under continuity on the data $b,\mu,f$, as we state in the following lemma.

\begin{lem}[\cite{KoKo17}]\label{lem KoKo}
Assume (A2) with $m=2-\alpha \in (0,1)$, and \eqref{satisfying} for  $F$ continuous, $b,\mu,f\in C(\Omega)$. Let $u\in C(\Omega)$ be an $L^p$-viscosity JJ-subsolution (resp., JJ-supersolution) of \eqref{quasilinearPDE}. Then, $u$ is a $C$-viscosity JJ-subsolution and also a standard $C$-viscosity subsolution (resp.\ supersolution) of
	\begin{align}\label{quasilinearPDE m}
	F(x,D^2u) + b(x)|Du|+\mu (x)|Du|^m+f(x)=0 \;\; \mathrm{in} \; \Omega.
	\end{align}
\end{lem}
The proof of Lemma \ref{lem KoKo} follows exactly the same arguments used in \cite[Lemmas 5.2 and 5.4]{KoKo17}, by replacing the Pucci operator there by $F(x,D^2u)$ and using Remark \ref{rmk singular locally bounded}; so we do not repeat it here.

\begin{teo}
	Let $\Omega\subset\rn$ be a bounded $C^{1,1}$ domain.
	Let $u$ be an $L^p$-viscosity JJ-solution of \eqref{quasilinearPDE}, with $b \in L^\varrho(\Omega)\cap C(\Omega)$, $\mu\in L^q(\Omega)\cap C(\Omega)$, and $f \in L^p(\Omega)\cap C(\Omega)$. Then,  the results of Theorems \ref{Calpha regularity estimates geral}, \ref{Log-Lip estimates}, and \ref{C1,alpha regularity estimates geral} in the cases $p_0<p<n$, $p=n$, and $n<p<\infty$ respectively are satisfied.
\end{teo}

In particular, this generalizes the regularity results in \cite[Theorem 1.1]{BD15} and \cite[Theorem 1.1]{BDL19-2}, since the coefficients are not supposed to be bounded -- they might be unbounded near the boundary.

\subsection{Perron method and multiplicity phenomena}\label{section perron}

In this section, we give an application of the viscosity equivalence in Section \ref{section intrinsic} in what refers to a Perron method for our equations.
In \cite[Proposition 2.1]{KoikePerronLp} it is proven a version of Perron method for $L^p$-viscosity solutions of coercive equations related to quadratic growth and bounded coefficients. Here we extend this result to any nonlinear gradient regime and unbounded data.

A fundamental tool in order to apply the Perron method consists of showing that a supremum of subsolutions is still a subsolution.
It usually presupposes that this supremum is properly defined, by considering the set where $u^*$ is locally bounded.
This auxiliary result plays an important role in multiplicity problems, see \cite{multiplicidade}.
More precisely, with this result at hands and the very recent V\'{a}zquez strong maximum principle developed in \cite{SSvazquez}, one may extend all multiplicity results in \cite{multiplicidade} to unbounded drift $b(x)$; as well as to an unbounded noncoercive term $c(x)$. The latter is possible due to $C^{1,\alpha}$ regularity theory for explicit unbounded zero order terms from \cite[Theorem 1, Remark 1.2]{Norn19} -- indeed, in this case one may view the term $c(x)u$ as the source term (by putting it to the RHS of the equation).

Let us first introduce the kind of solution we have in mind. We observe that, in both $L^p$ or $C$--viscosity senses, we only need to define a subsolution as an upper semicontinuous function, and a supersolution being lower semicontinuous, which is sufficient to ensure the attainment of a maximum or minimum, respectively, over compact sets.
Moreover, if the function $u$ is not semicontinuous, we say that $u$ is a viscosity subsolution (resp.\ supersolution) provided $u^*$ (resp.\ $u_*$) is. Here, $u^*$ and $u_*$ are the upper and lower semicontinuous envelopes of $u$ respectively, namely
\begin{align*}
\textstyle u^{\ast} (x)\defeq \varlimsup_{y\rightarrow x} u(y),  \qquad u_{\ast} (x)\defeq\varliminf_{y\rightarrow x} u(y),
\end{align*}
see \cite[p.\ 22]{user},  \cite[p.\ 40]{Koike}. Note that $u^{\ast} \in USC(\Omega)$  and $u_{\ast} \in LSC(\Omega)$.

\begin{prop}\label{supLp unbounded}
	Let $\Omega\subset \rn$ be a bounded open set, and $F$ satisfying (A1), (A2). Let $\mathcal{A}\subset C(\Omega)$ be a nonempty set of $L^p$-viscosity subsolutions $($supersolutions$)$ of
	$F(x,Du,D^2u)+c(x)u\geq f(x)$ in $\Omega$, for $c,f\in L^p_{\mathrm{loc}}(\Omega)$, $c\ge 0$ a.e.\ in $\Omega$.
	Assume that $u(x)\defeq \sup_{v\in\mathcal{A}} v(x)$ $($resp. $u(x)\defeq\inf_{v\in\mathcal{A}} v(x))$ is locally bounded in $\Omega$.
	Then, $u$ is a viscosity subsolution $($supersolution$)$ of the same equation in $\Omega$.
\end{prop}

\begin{proof}
	The reasoning is similar to the classical proof in \cite[Theorem 4.2]{Koike}, see also \cite{KoikePerronLp}. The main difference comes from the fact that one needs to use Lemma \ref{lema strict} as far as unbounded coefficients are concerned; we sketch the proof for the sake of completeness. Since $u$ can be a discontinuous function, it is accomplished through the semicontinuous envelopes $u_{\ast}$, $u^{\ast}$.
	
	Next, we note that, since $c\ge 0$, we have $F[v]\geq f(x)-c(x)v\geq f(x)-c(x)u =:g(x)$ for all $v\in \mathcal{A}$. Then it is enough to prove that $u$ is a subsolution of $F[u]\ge g(x)$ in $\Omega$.
	We assume on the contrary that there exist $\varepsilon, r>0$, $x_0\in\Omega$ and $\phi\in W^{2,p}_{\mathrm{loc}}(\Omega)$ such that $u^{\ast}-\phi$ attains a local maximum at $x_0$ but
	\begin{align}\label{eq1}
	F(x,D\phi (x),D^2 \phi (x))-g(x) \leq -\varepsilon\; \textrm { a.e.\ in } B_r (x_0).
	\end{align}
	By  Lemma~\ref{lema strict} this maximum can supposed to be strict, i.e.\ we assume	
   \begin{center}
		$(u^{\ast}-\phi)(x)<(u^{\ast}-\phi)(x_0)=0$ \;\; in $\overline{B}_{r}(x_0)\setminus\{x_0\}$.
	\end{center}
	Further,
	$\phi (x_0)=\lim_{x_k\rightarrow x_0} u(x_k)=\lim_{k\rightarrow \infty} v_k(x_k)$, for some $x_k\in \overline{B}_r(x_0$), $x_k\rightarrow x_0$. Here, $u(x_k)=\sup_{v\in \mathcal{A}} v(x_k)=\lim_{m\rightarrow \infty} v_m^k(x_k) $, where $v_m^k$ is a maximizing sequence for each $k\in\n$. Set $v_k=v_k^k \in \mathcal{A}$.
	\smallskip
	
	Let $z_k\in \overline{B}_r(x_0)$ so that $\max_{\overline{B}_r(x_0)} (v_k-\phi)=(v_k-\phi)(z_k)$, with $z_k\rightarrow z \in \overline{B}_r(x_0)$ up to a subsequence. By the definition of $v_k$ as an $L^p$-viscosity subsolution we have
	\begin{align}\label{eq3}
	\textrm{ess} \varlimsup_{x\rightarrow z_k} \;\{ F(x,D\phi (x),D^2 \phi (x))-g(x)\} \geq 0.
	\end{align}
	Then, as in \cite{Koike} one shows that $z=x_0$ and $v_k(z_k)\rightarrow u^{\ast}(x_0)$ from the fact that $x_0$ is the unique strict global maximum point in the ball $\overline{B}_r(x_0)$.
	This fact and \eqref{eq3} contradicts \eqref{eq1}.
\end{proof}

The following Perron's type existence result is now a consequence of \cite[proof of Theorem 3.3]{KoikePerronLp}, by making use of Lemma \ref{lema strict} and Proposition \ref{prop Lp iff Lupsilon}.

\begin{prop}
	Assume $F$ satisfies (A1), (A2), \eqref{H equiv visc}. Let $\underline{u},\,\overline{u}\in C(\Omega)$ be a pair of viscosity sub and supersolutions of $F[u]=f(x)$ in $\Omega$ respectively, with $\underline{u}\le\overline{u}$ in $\Omega$.
	Then the function $u$ defined as $u(x)=sup_{v\in \mathcal{A}}\, v(x)$ for $x\in\Omega$, where
	$$
	\mathcal{A}\defeq \left\{ v\in C(\Omega); \; v \textrm{ is a viscosity subsolution of $F[u]=f(x)$ in $\Omega$,\;  $\underline{u}\le v \le \overline{u}$ in $\Omega$}   \right\},
	$$
	is a viscosity solution of $F[u]=f(x)$ in $\Omega$.
\end{prop}

As an application, we consider a class of problems with quadratic gradient growth given by
\begin{align}\label{Plambda} \tag{$P_\sigma$}
-F(x,Du,D^2u) =\sigma c(x)u+\langle M(x)D u, D u \rangle +h(x) \textrm{ in } \Omega  , \quad u=0 \textrm{ on } \partial\Omega,
\end{align}
under the following assumption:
\begin{itemize}
	\item[(A6)]({{\bf Multiplicity requirements}}) $\Omega$ is a bounded $C^{1,1}$ domain in $\rn$, $\sigma\in \mathbb{R}$, $n\geq 1$, $F$ satisfies (A1)--(A3) for $m=1$ and $p>n$, also $c,h\in L^p(\Omega)$, $c\ge 0$ a.e.\ in $\Omega$ with $c\not\equiv 0$, in addition to:
	\begin{enumerate}[(a)]
		\item\label{M}
		$M$ satisfies $\mu_1 I \leq M(x)\leq \mu_2 I$ \; a.e.\ in $\Omega$, for some $\mu_1, \mu_2>0$;
		\item\label{H0}
		the problem $(P_0)$ has a strong solution $u_0$;
		\item\label{ExistUnic F}
		for $f\in L^p(\Omega)$, there exists a unique viscosity solution of $F[u]=f(x)$ in $\Omega$, with $u=0$ on $\partial\Omega$;
		\item\label{Hstrong}
		\eqref{Plambda} enjoys $W^{2,p}$ estimates, in the sense of Definition \ref{def Ws,p estimates}.
	\end{enumerate}
\end{itemize}

We observe that item \eqref{H0} in (A6) holds true via \cite[Theorem 1 (ii)]{arma2010} if, for instance $Mh$ has small $L^p$ norm. Likewise, by
\cite[Theorem 1 (iii)]{arma2010}, such function $u_0$ is the unique viscosity solution of $(P_0)$. Recall that  \eqref{ExistUnic F} and \eqref{Hstrong} are both true if $F$ satisfies for example the hypotheses in Proposition \ref{prop regu W2,p}.

The supplementary hypothesis \eqref{Hstrong} in (A6) for the uniqueness or comparison results is unavoidable, since in the universe of $L^p$-viscosity solutions they are only achieved in the presence of a strong solution; see \cite{arma2010} and references therein.
One observes that if $h$ has a sign, then $u_0$ has the same sign via the maximum principle, see \cite[Remark 6.25]{multiplicidade}.

In this case, a priori bounds and multiplicity results in \cite{multiplicidade, NSS2020} become available for unbounded coefficients, in view of Proposition \ref{supLp unbounded} and the recent V\'{a}zquez's strong maximum principle in \cite{SSvazquez}, see the discussion in \cite{multiplicidade}.

\begin{teo} \label{apriori}
	Assume items \eqref{M}--\eqref{ExistUnic F} in (A6), and let $\Lambda_1,\Lambda_2>0$ so that $0<\Lambda_1<\Lambda_2$. Then any viscosity solution $u$ of \eqref{Plambda} satisfies the uniform a priori bounds:
	$$
	\|u^-\|_{L^\infty(\Omega)} \leq C\, , \;\textrm{ for all } \sigma\in [0,\Lambda_2],\qquad\|u^+\|_{L^\infty(\Omega)} \leq C\, , \;\textrm{ for all } \sigma\in [\Lambda_1,\Lambda_2],
	$$
	where $C$ depends on $n,p, \varrho,\mu_1,  \Lambda_1,\Lambda_2,  \|b\|_{\varrho},\|c\|_{p}, \|h\|_{p},\|u_0\|_{\infty},\lambda,\Lambda$, $\partial\Omega$, and the set where $c>0$.
\end{teo}

Recall that solutions of \eqref{Plambda} are $C^{1,\alpha}$ up to the boundary. Then it makes sense to consider the set $E \defeq C^1(\overline{\Omega})$ with the following ordering: For $u,v\in  E$, we say that $u\ll v$ if for every $x \in \Omega$ it holds $u(x)<v(x)$, and for $x_0\in\partial\Omega$ we have either $u(x_0)<v(x_0)$, or $u(x_0)=v(x_0)$ and $\partial_\nu u (x_0)<\partial_\nu v (x_0)$,  where $\vec{\nu}$ is the interior unit normal to $\partial \Omega$ at $x_0$.

\begin{teo} \label{th1.1,1.2,1.3}
	Assume items \eqref{M}--\eqref{ExistUnic F} in (A6). Then the set
	$
	\Sigma = \{ \,(\sigma , u) \in \real \times E\, ;  u \; \textrm{solves \eqref{Plambda}} \,\}
	$
	possesses an unbounded component $\mathcal{C}^+\subset [0,+\infty]\times E$ with $\mathcal{C}^+\cap ( \{0\}\times E )=\{u_0\}$, such that:
	\begin{enumerate}[1.]
		\item either it bifurcates from infinity to the right of the axis $\sigma =0$ with the corresponding solutions having a positive part blowing up to infinity in $C (\overline{\Omega})$ as $\sigma\rightarrow 0^+$; or its projection on the $\sigma$ axis is $[0,+\infty)$.
		
		\item There exists $\bar{\sigma} \in (0,+\infty]$ such that, for every $\sigma\in (0,\bar{\sigma})$, the problem \eqref{Plambda} has at least two viscosity solutions, $u_{\sigma, 1}$ and $u_{\sigma , 2}$, satisfying
		$u_{\sigma , 1} \rightarrow u_0$ in $E$ and $\max_{\overline{\Omega}} u_{\sigma , 2} \to  +\infty$ as $ \sigma\rightarrow 0^+$.
		When $\,\bar{\sigma}<+\infty$ the problem $(P_{\bar{\sigma}})$ has at least one viscosity solution; which is unique if $F(x,\xi,X)$ is convex in $(\xi,X)$.
		
		\item Assume in addition that item \eqref{Hstrong} holds in (A6). Then $u_\sigma\,$ for $\sigma\leq 0$ are unique among viscosity solutions; whereas the solutions in item 2 for $\sigma>0$ are ordered, i.e.\ $u_{\sigma, 1} \ll u_{\sigma, 2}$. Moreover:
		\begin{enumerate}[3.1]
			\item If \eqref{H0} in (A6) is such that $u_0 \leq 0$, $cu_0\lneqq 0$, then every nonpositive viscosity solution of \eqref{Plambda} with $\sigma>0$ satisfies $u\ll u_0$. Also, for every $\sigma >0$,
			\eqref{Plambda} has at least two nontrivial strong solutions $ u_{\sigma, 1} \ll u_{\sigma , 2}\,$ so that  $u_{\sigma_2 ,1}\ll u_{\sigma_1 ,1} \ll u_0$ if $\,0<\sigma_1<\sigma_2$, with $u_{\sigma, 1} \to  u_0$ in  $E$ and $\max_{\overline{\Omega}} u_{\sigma , 2} \to +\infty$ as $\sigma\rightarrow 0^+$.
			If $F(x,\xi,X)$ is convex in $(\xi,X)$ then $\max_{\overline{\Omega}}\, u_{\sigma,2}>0$ for all $\sigma>0$.
			
			\item If \eqref{H0} in (A6) is such that $u_0 \geq 0$, $cu_0\gneqq 0$, then every nonnegative viscosity solution of \eqref{Plambda} with $\sigma>0$ satisfies $u\gg u_0$. Further, there exists  $\bar{\sigma}_1 \in (0,+\infty)$ satisfying:
\begin{enumerate}[(i)]
\item for every $\sigma\in (0,\bar{\sigma}_1)$, \eqref{Plambda} has at least two nontrivial strong solutions $ u_{\sigma, 1} \ll u_{\sigma , 2}\,$, where $u_0 \ll u_{\sigma_1 ,1}\ll u_{\sigma_2 ,1} $ if $\,0<\sigma_1<\sigma_2$, with $u_{\sigma , 1} \to u_0$ in $E$ and $\max_{\overline{\Omega}} u_{\sigma , 2} \to +\infty$ as $\sigma\rightarrow 0^+$;
\item $(P_{\bar{\sigma}_1})$ has at least one strong solution $u_{\bar{\sigma}_1}$, which is unique if $F$ is convex in $(\xi,X)$;
\item for $\sigma > \bar{\sigma}_1$,  \eqref{Plambda} has no nonnegative viscosity solution.
\end{enumerate}
In addition there exists $\delta>0$ so that if $\mu_2\|h\|_{L^p(\Omega)}\leq \delta$, $h\gneqq 0$, then there is $\bar{\sigma}_2 > \bar{\sigma}_1$ such that:
\begin{enumerate}[(iv)]
	\item \eqref{Plambda} has at least two strong solutions for $\sigma> \bar{\sigma}_2$, with $u_{\sigma,1}\ll 0$ in $\Omega$ and $\min_{\overline{\Omega}} u_{\sigma,2}<0$;
\item  $(P_{\bar{\sigma}_2})$ has at least one nonpositive strong solution, which is unique if $F$ is convex in $(\xi,X)$;
\item for $\sigma < \bar{\sigma}_2$, the problem \eqref{Plambda} has no nonpositive  viscosity solution.
\end{enumerate}
	\end{enumerate}
\end{enumerate}
\end{teo}

We highlight that the definition of $L^p$-viscosity solutions is equivalent to the $L^n$-sense for instance, by Proposition \ref{prop Lp iff Lupsilon} with $m=2$. In other words, the results in \cite{NSS2020, multiplicidade} become intrinsically  depending on the data.

\medskip

\textbf{{Acknowledgment.}} Jo\~{a}o Vitor da Silva was partially supported by CNPq-Brazil grant 310303/2019-2. Gabrielle Nornberg was supported by FAPESP grants 2018/04000-9 and 2019/03101-9, São Paulo Research Foundation.

%\bibliography{bibtex}

{\small
\begin{thebibliography}{99}

\bibitem{Ahomo}
Armstrong, S.; Hung T. \emph{Stochastic homogenization of viscous Hamilton–Jacobi equations and applications.} Analysis \& PDE 7.8 (2015): 1969-2007.

\bibitem{BDru} Birindelli, I.; Demengel, F. \emph{Regularity and uniqueness of the first eigenfunction for singular fully nonlinear operators.} J. Differential Equations 249 (2010), no. 5, 1089–1110.


\bibitem{BD15} Birindelli, I.; Demengel, F.
\emph{H\"{o}lder regularity of the gradient for solutions of fully nonlinear equations with sub linear first order term}.
Geometric methods in PDE's, 257–268, Springer INdAM Ser., 13, Springer, Cham, 2015.

\bibitem{BDL19-2} Birindelli, I.; Demengel, F.; Leoni, F.
\emph{$C^{1, \gamma}$ regularity for singular or degenerate fully nonlinear equations and applications}.
NoDEA Nonlinear Differential Equations Appl. 26 (2019), no. 5, Paper No. 40, 13 pp.

\bibitem{BGMW} Braga, J.E.M.; Gomes, D.E.; Moreira, D.; Wang, L. \emph{Krylov's Boundary Gradient Type Estimates for Solutions to Fully Nonlinear Differential Inequalities with Quadratic Growth on the Gradient.} SIAM Journal on Mathematical Analysis, 52(5), (2020), 4469-4505.

\bibitem{CC03} Cabr\'{e}, X., Caffarelli, L.A.
\textit{Interior $C^{2, \alpha}$ regularity theory for a class of nonconvex fully nonlinear elliptic equations}.
J. Math. Pures Appl. (9) 82 (2003), no. 5, 573-612.

\bibitem{Caf89} Caffarelli, L. A.
\emph{Interior a priori estimates for solutions of fully nonlinear equations}.
Ann. of Math. (2) 130 (1989), no. 1, 189–213..

\bibitem{CafCab} Caffarelli, L. A.; Cabr\'{e}, X.
\emph{Fully nonlinear elliptic equations}.
American Mathematical Society Colloquium Publications, 43. American Mathematical Society, Providence, RI, 1995. vi+104 pp. ISBN: 0-8218-0437-5.

\bibitem{CCKS} Caffarelli, L.; Crandall, M.G.; Kocan, M.; \'{S}wi\c{e}ch, A.
\emph{On viscosity solutions of fully nonlinear equations with measurable ingredients.}
Comm. Pure Appl. Math. 49 (1996), 365–397.

\bibitem{CafHua03} Caffarelli, L.A.; Huang, Q.
\textit{Estimates in the generalized Campanato-John-Nirenberg spaces for fully nonlinear elliptic equations}.
Duke Math. J. 118 (2003), no. 1, 1–17.

\bibitem{CDLP10} Capuzzo Dolcetta, I.; Leoni, F.; Porretta, A.
\textit{H\"{o}lder estimates for degenerate elliptic equations with coercive Hamiltonians}.
Trans. Amer. Math. Soc. 362 (2010), no. 9, 4511–4536.

\bibitem{user} Crandall, M.G.; Ishii, H.; Lions, P. L.
\emph{User's Guide to viscosity solutions of second order partial differential equations.}
Bull. Amer. Math. Soc. (N.S.) 27 (1992), 1–67.

\bibitem{CKSS} Crandall, M.G.; Kocan, M.;  Soravia, P.; \'{S}wi\c{e}ch, A.
\emph{On the equivalence of various weak notions of solutions of elliptic PDEs with measurable ingredients.}
Pitman Res.\ Not.\ Math.\ Ser.\ (1996), 136-162.

\bibitem{DD19} da Silva, J.V.; dos Prazeres, D.
\textit{Schauder type estimates for viscosity solutions to non-convex fully nonlinear parabolic equations and applications}.
Potential Anal. 50 (2019), no. 2, 149-170.

\bibitem{daSilRic19} da Silva, J.V.; Ricarte, G.C.
\textit{An asymptotic treatment for non-convex fully non-linear elliptic equations: Global Sobolev and BMO type estimates}.
Comm. Contemp. Math. 21 (2019), no. 7, 1850053, 28 pp.

\bibitem{daST17} da Silva, J.V.; Teixeira, E. V.
\emph{Sharp regularity estimates for second order fully nonlinear parabolic equations.}
Math. Ann. 369(3–4) (2017), 1623–1648.

\bibitem{DKM14} Daskalopoulos, P.; Kuusi, T.; Mingione, G.
\emph{Borderline estimates for fully nonlinear elliptic equations}.
Comm. Partial Differential Equations 39 (2014), no. 3, 574–590.

\bibitem{DongKry19} Dong, H.; Krylov, N.V.
\textit{Fully nonlinear elliptic and parabolic equations in weighted and mixed-norm Sobolev spaces}.
Calc. Var. Partial Differential Equations 58 (2019), no. 4, Paper No. 145, 32pp.

\bibitem{Esc93} Escauriaza, L.
\textit{$W^{2,n}$ a priori estimates for solutions to fully nonlinear equations}.
Indiana Univ. Math. J. 42 (1993), no. 2, 413-423.

\bibitem{Ev82} Evans, L.C.
\textit{Classical solutions of fully nonlinear, convex, second-order elliptic equations}.
Comm. Pure Appl. Math., 35(3). 333-363, 1982.

\bibitem{Koike} Koike, S.
\emph{A Beginner's guide to the theory of viscosity solutions}, MSJ Memoirs, 13. Mathematical Society of Japan, Tokyo, 2004. viii+123 pp. ISBN: 4-931469-28-0.

\bibitem{KoikePerronLp} Koike, S.
\emph{Perron’s method for $L^p$-viscosity solutions}, Saitama Math.\ Journal,
23 (2005),9-28.

\bibitem{KoKo17} Koike, S.; Kosugi, T.
\emph{Maximum principle for Pucci equations with sublinear growth in Du and its applications}.
Nonlinear Anal.\ 160 (2017), 1–15.

\bibitem{KSmpite2007} Koike, S.; \'{S}wi\c{e}ch, A.
\emph{Maximum principle for fully nonlinear equations via the iterated comparison function method.}
Math. Ann. 339, no. 2 (2007), 461-484.

\bibitem{KSexist2009} Koike, S.; \'{S}wi\c{e}ch, A.
\emph{Existence of strong solutions of Pucci extremal equations with superlinear growth in $Du$.}
J. Fixed Point Theory Appl. 5, no. 2 (2009), 291-304.

\bibitem{KSweakharnack} Koike, S.; \'{S}wi\c{e}ch, A.
\emph{Weak Harnack inequality for fully nonlinear uniformly elliptic PDE with unbounded ingredients.}
J. Math. Soc. Japan. 61, no. 3 (2009), 723-755.

\bibitem{KoTa19} Koike, S.; Tateyama, S.
\emph{On $L^p$-viscosity solutions of bilateral obstacle problems with unbounded ingredients}.
Math. Ann. 377, (2020), 883–910.

\bibitem{Kry82} Krylov, N.V.
\textit{Boundedly inhomogeneous elliptic and parabolic equations}.
Izv. Akad. Nauk SSSR Ser. Mat. 46 (1982), no. 3, 487–523, 670

\bibitem{Kry13} Krylov, N.V.
\textit{On the existence of $W^{2, p}$ solutions for fully nonlinear elliptic equations under relaxed convexity assumptions}.
Comm. Partial Differential Equations 38 (2013), no. 4, 687–710.

\bibitem{Kry20} Krylov, N.V.
\textit{Linear and fully nonlinear elliptic equations with $L_d$-drift}.
To appear in Comm. Partial Differential Equations (2020) https://doi.org/10.1080/03605302.2020.1805462.

\bibitem{Kry-Saf79} Krylov, N.V.; Safonov, M.V.
\textit{An estimate for the probability of a diffusion process hitting a set of positive measure}. (Russian)
Dokl. Akad. Nauk SSSR 245 (1979), no. 1, 18–20.

\bibitem{LadUralt} Ladyzhenskaya, O.A.; Ural'tseva, N.N.
\textit{Estimates on the boundary of the domain of first derivatives of functions satisfying an elliptic or a parabolic inequality}. (Russian)
Translated in Proc. Steklov Inst. Math. 1989, no. 2, 109–135. Boundary value problems of mathematical physics, 13 (Russian).
Trudy Mat. Inst. Steklov. 179 (1988), 102--125, 243.

\bibitem{Wangpointwise2020} Lian, Y.; Wang,L.; Zhang, K.
\emph{Pointwise Regularity for Fully Nonlinear Elliptic Equations in General Forms}, arXiv:2012.00324,  2020.

\bibitem{Na3} Nadirashvili, N.; Vl\u{a}du\c{t}, S.
 \textit{Singular solutions of Hessian elliptic equations in five dimensions}. J. Math. Pures Appl. (9) 100 (2013), no. 6, 769-784.\label{Na3}

\bibitem{tese} Nornberg, G.S.
\emph{Methods of the regularity theory in the study of partial differential equations with natural growth in the gradient.}
Ph.D.\ thesis. PUC-Rio, Brazil (2018).

\bibitem{Norn19} Nornberg, G.
\emph{$C^{1,\alpha}$ regularity for fully nonlinear elliptic equations with superlinear growth in the gradient}.
J. Math. Pures Appl. (9) 128 (2019), 297–329.

\bibitem{NSS2020} Nornberg, G; Schiera, D.; Sirakov, B.
\emph{A priori estimates and multiplicity for systems of elliptic PDE with natural gradient growth}.
(2020) DCDS 40(6): 3857-3881.

\bibitem{multiplicidade} Nornberg, G.; Sirakov, B.
\emph{A priori bounds and multiplicity  for fully nonlinear equations with quadratic growth in the gradient.},
J. Funct. Anal. 276 (2019), no. 6, 1806-1852.

\bibitem{PT16} Pimentel, E.; Teixeira, E.V.
\textit{Sharp Hessian integrability estimates for nonlinear elliptic equations: An asymptotic approach},
J. Math. Pures Appl. (9) 106(4) (2016) 744-767.

\bibitem{dosPT16} Prazeres, D.; Teixeira, E. V.
\textit{Asymptotics and regularity of flat solutions to fully nonlinear elliptic problems}. Ann. Sc. Norm. Super. Pisa Cl. Sci. Vol. XV (2016).

\bibitem{ST} Silvestre, L.; Teixeira, E. V.
\emph{Regularity estimates for fully non linear elliptic equations which are asymptotically convex. Contributions to nonlinear elliptic equations and systems.}
Progr. Nonlinear Differential Equations Appl., 86, Birkh\"{a}user/Springer, Cham (2015), 425–438.

\bibitem{arma2010} Sirakov, B.
\emph{Solvability of uniformly elliptic fully nonlinear PDE.}
Arch. Ration. Mech. Anal. 195 (2010), no. 2, 579-607.

\bibitem{Sir2018} Sirakov, B.
\emph{Boundary Harnack estimates and quantitative strong maximum principles for uniformly elliptic PDE}. Int. Math. Res. Not. IMRN 2018, no. 24, 7457–7482.

\bibitem{SSvazquez} Sirakov, B.; Souplet, P.
\emph{The V\'{a}zquez maximum principle and the Landis conjecture for elliptic PDE with unbounded coefficients.} arXiv:2010.08511, 2020.

\bibitem{Swiech} \'{S}wi\c{e}ch, A.
\emph{$W^{1,p}$-interior estimates for solutions of fully nonlinear, uniformly elliptic equations}.
Adv. Diff. Eqs. 2 (6) (1997), 1005–1027.

\bibitem{Swiech20} \'{S}wi\c{e}ch, A.
\emph{Pointwise properties of $L^p$-viscosity solutions of uniformly elliptic equations with quadratically growing gradient terms}. Discrete and Continuous Dynamical Systems 40(5) (2020), 2945-2962.

\bibitem{Tate20} Tateyama, S.
\emph{The Phragm\'{e}n-Lindel\"{o}f theorem for $L^p-$viscosity solutions of fully nonlinear parabolic equations with unbounded ingredients}.
J. Math. Pures Appl. (9) 133 (2020), 172–184.

\bibitem{Tei14} Teixeira, E.V.
\emph{Universal moduli of continuity for solutions to fully nonlinear elliptic equations}.
Arch. Ration. Mech. Anal. 211 (2014), no. 3, 911–927.

\bibitem{Tru83} Trudinger, N.S.
\textit{Fully nonlinear, uniformly elliptic equations under natural structure conditions}.
Trans. Amer. Math. Soc. 278 (1983), no. 2, 751-769.

\bibitem{Tru88} Trudinger, N.S.
\textit{H\"{o}lder gradient estimates for fully nonlinear elliptic equations}.
Proc. Roy. Soc. Edinburgh Sect. A 108 (1988), no. 1-2, 57-65.

\bibitem{W2} Wang, L.
\emph{On the regularity theory of fully nonlinear parabolic equations: II}.
Comm. Pure Appl. Math., 45, no.2 (1992), 141–178.

\bibitem{Winter} Winter, N.
\emph{$W^{2,p}$ and $W^{1,p}$ estimates at the boundary for solutions of fully nonlinear, uniformly elliptic equations}.
Z. Anal. Anwend. 28 (2009), 129-164.

\end{thebibliography}
}
%\bibliographystyle{abbrv}
%%%%%%%%%%%%%%%%%%%%%%

\end{document}